\documentclass[a4letter,11pt,reqno]{amsart}
\usepackage[T1]{fontenc}
\usepackage[utf8]{inputenc}
\usepackage{lmodern}
\usepackage[english]{babel}
\usepackage{amsmath,a4wide}
\usepackage{xfrac}
\usepackage{esint}
\usepackage{stmaryrd,mathrsfs,bm,amsthm,mathtools,yfonts,amssymb,color,braket,booktabs,graphicx,graphics,amsfonts}
\usepackage{latexsym,microtype,indentfirst,hyperref}
\usepackage{xcolor}
\usepackage{courier}
\usepackage{pdfpages}
\usepackage{constants}

\usepackage{graphicx}
\usepackage[dvips]{eps fig} 

\newconstantfamily{c}{symbol=c}

\usepackage{epsfig}
\usepackage{epstopdf}

\theoremstyle{definition}
\newtheorem{thm}{Theorem}[section]
\newtheorem{defi}[thm]{Definition}
\newtheorem{lem}[thm]{Lemma}
\newtheorem{prop}[thm]{Proposition}
\newtheorem{rk}[thm]{Remark}

\numberwithin{equation}{section}
\def\E{\mathcal E}
\def\R{\mathbb R}
\def\e{\varepsilon}
\def\N{\mathbb N}
\def\H{\mathcal H}

\newcommand{\mres}{\mathbin{\vrule height 1.6ex depth 0pt width 
0.13ex\vrule height 0.13ex depth 0pt width 1.3ex}}

\begin{document}

\allowdisplaybreaks[4]

\title[Existence and regularity theorems of one-dimensional Brakke flows]{Existence and regularity theorems \\ of one-dimensional Brakke flows}

\author{Lami Kim}
\address{Center for Mathematical Analysis \& Computation, Yonsei University, Seoul 03722 Republic of Korea}
\email{lm.kim@yonsei.ac.kr}

\author{Yoshihiro Tonegawa}
\address{Department of Mathematics, Tokyo Institute of Technology, Tokyo 152-8551 Japan}
\email{tonegawa@math.titech.ac.jp}

\medskip
\thanks{L.\,K. is partially supported by the National Research Foundation of Korea (NRF) Grant NRF-2019R1I1A1A01062992, NRF-2015R1A5A1009350 and Y.\,T. by Japan Society for the Promotion of
Science (JSPS) Grant 18H03670, 19H00639 and 17H01092.} 

\begin{abstract}
Given a closed countably $1$-rectifiable set in $\R^2$ with locally finite $1$-dimensional 
Hausdorff measure, we prove that there exists a Brakke flow
starting from the given set with the following regularity property.
For almost all time, the flow locally consists of a finite number of embedded
curves of class $W^{2,2}$ whose endpoints meet at junctions with
angles of either 0, 60 or 120 degrees. 
\end{abstract}

\maketitle \setlength{\baselineskip}{18pt}

\section{Introduction}
A family of $n$-dimensional surfaces $\{\Gamma_t\}_{t\geq 0}$ in $\R^{n+1}$ is called the
mean curvature flow (abbreviated hereafter as MCF) if the velocity is equal to its mean
curvature at each point and time. 
Given a smooth compact surface $\Gamma_0$, there exists a
smoothly evolving MCF starting from $\Gamma_0$ 
until some singularities such as vanishing or pinching occur. 
There are numerous notions of generalized MCF past singularities and
the Brakke flow is the earliest one conceived by Brakke \cite{Brakke} within the
framework of geometric measure theory. 

In \cite{KimTone}, by reworking the proof of 
Brakke's general existence theorem \cite[Chapter 4]{Brakke}, the following time-global existence 
theorem was established: {\it Let $\Gamma_0$ be a closed
countably $n$-rectifiable set whose complement $\R^{n+1}\setminus \Gamma_0$
is not connected and whose $n$-dimensional Hausdorff measure $\mathcal H^n(\Gamma_0)$ is finite or at most of exponential
growth near infinity. Then there exists a non-trivial Brakke flow starting from $\Gamma_0$. }
The similar existence problem was considered with fixed boundary conditions \cite{Stuvard-Tonegawa},
where the Brakke flow converges subsequentially to a solution for the Plateau problem as $t\rightarrow\infty$. 
The Brakke flow is a family of varifolds which satisfies the motion law
of MCF in a distributional sense, and the flow 
may not be smooth in general. In fact, for $n=1$,
a typical Brakke flow is expected to look like an evolving network of curves joined by junctions, 
and the stable junctions are the ones with
three curves meeting with equal angles of 120 degrees. Let us call such junction simply as the {\it triple junction}. Physically, one can associate the motion of grain boundaries in polycrystalline materials
to this mathematical model. 
As the network goes though various topological changes, junctions of more than 
three curves are formed when two or more triple junctions collide. They then 
should instantaneously break up into triple junctions connected by short curves since the latter
has shorter curve length.
With this formal intuition, it is reasonable to speculate that, for general initial data, 
there should exist a $1$-dimensional Brakke flow
which locally consists of finite number of smooth curves meeting only at
triple junctions for almost all time, or even better, for all co-countable time. 

The present paper edges towards the positive resolution of this speculation.
We prove that the $1$-dimensional Brakke flow constructed in \cite{KimTone,Stuvard-Tonegawa}
has this property, albeit degenerately. Roughly speaking, we prove that
the support of Brakke flow in \cite{KimTone,Stuvard-Tonegawa} locally consists of embedded $W^{2,2}$ curves whose endpoints
meet at junctions with either 0, 60 or 120 degrees for almost all time
(Theorem \ref{KTmain2} and \ref{KTmain3}). Here, $W^{2,2}$ curve
means that it is locally represented as the graph of a function whose first and second weak derivatives are
square-integrable. 
In particular, these
curves are $C^{1,\sfrac12}$ by the standard embedding theorem and the angle condition says that they 
cannot intersect each other with angles other than 0, 60 or 120 degrees (see Figure \ref{Fig1} for some
examples). 
The configuration (a) is the typical picture that three curves meet at a triple junction. Note that
the configurations like (b)-(d) can arise as limits of curves connected only 
by triple junctions while keeping the curvatures bounded. In that sense, they may be seen as
degenerate triple junctions. For example,
(b) can arise as a limit of two triple junctions connected by an infinitesimally short curve, and 
(c) can be a limit of infinitesimally small regular hexagon with 6 curves leaving from the each vertex. 
The configuration like (d) can also arise as a limit of combination of triple junctions.
Degenerate configurations like (b)-(d) are unlikely to occur for a set of time with positive measure but we cannot exclude the possibility in the 
present paper
(see Section \ref{frr} for the related comment). 
In addition, we prove that the same regularity property holds for any tangent flow (Theorem \ref{KTmain4}). 
\begin{figure}[!h]
\centering
\includegraphics[width=16cm]{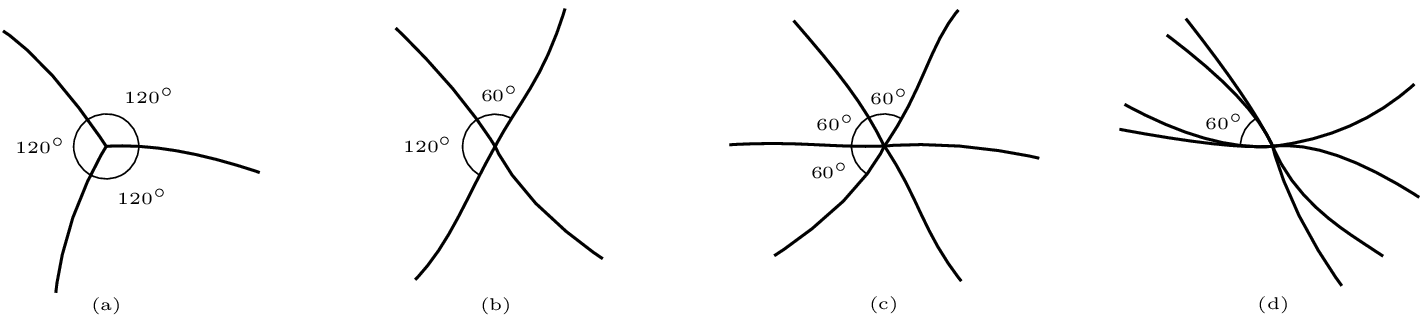}              
\caption{}
\label{Fig1}
\end{figure} 
The regularity statement for the support 
itself is a substantial improvement compared to the result deduced from the standard
regularity theorem in geometric measure theory. 
The Brakke flow by definition has locally square-integrable
generalized mean curvature and integer-multiplicity
for almost all time, and with the Allard regularity theorem \cite{Allard}, 
there exists a dense open set in which the flow consists of $W^{2,2}$ curves. 
But it is not known that the complement of such ``regular part'' has null $1$-dimensional 
measure in general when the Allard theorem is applied. 
As for the regularity near the junctions, 
we are not aware of any general theorem on the
uniqueness of tangent cone or regularity theorem 
other than the one for $1$-dimensional stationary varifolds due to Allard-Almgren \cite{AA}.

To avoid a possible confusion, one should note that the claim of the present paper is pertinent only
to the Brakke flows constructed in \cite{KimTone,Stuvard-Tonegawa}, and not necessarily to arbitrary 1-dimensional Brakke flows
as defined in \cite{Brakke,Ton1}. 
Just to clarify this point, here is a simple illustration: take a union of two lines crossing at 90 degrees as the initial data. Then, the Brakke flow starting from it in \cite{KimTone} cannot stay time-independent due to the 
result of the present paper, since such 
crossing, even though it is stationary and a Brakke flow itself, 
does not satisfy the angle condition as above. 
The implication is that
the method of construction in \cite{KimTone,Stuvard-Tonegawa} selects a special subclass of Brakke flows with this additional regularity property. This class is also compact with respect to the weak convergence (Lemma \ref{comp}) and because of this, it is a natural category to study measure-theoretically. 

There are numerous closely related results even if we focus on the 1-dimensional problem and we mention 
some of the most relevant works. The problem is also called curvature flow, curve-shortening flow or 
network flow in 
the literature. For an embedded closed curve as initial data, the well-known theorem due to 
Gage-Hamilton \cite{Gage} and Grayson \cite{Grayson} says that the flow stays embedded and becomes convex, and
eventually shrinks to a round circle. For a flow with a triple junction, Bronsard-Reitich \cite{Bronsard} 
proved the short-time existence and uniqueness for $C^{2+\alpha}$ initial data. The long-time 
behavior of the flow was studied in \cite{MNT,MMN} for fixed boundary condition. 
For less regular initial data, G\"{o}sswein-Menzel-Pluda \,\cite{Pluda} studied the short-time existence
and uniqueness within a $W^{2,p}$ class. For initial data with general junctions (non-120 degree triple 
junction or junction with more edges), Ilmanen-Neves-Schulze \cite{INS} showed the short-time existence 
of the flow. For more information on the existence and uniqueness for various related works, 
see \cite{MNPS} and the references therein. All of the aforementioned
flows may be called classical network flows in the sense that the junctions are all triple junctions
with equal angles as they evolve. For more general flow which allows topological changes,
in \cite{KimTone,Stuvard-Tonegawa}, the long-time existence was 
studied within the framework of geometric measure theory. It is expected that the solution is unique for ``regular enough'' initial network, but the precise 
condition is not known (see \cite{Laux} for a related uniqueness question). The present work 
establishes an interesting connection between the solution of \cite{KimTone,Stuvard-Tonegawa} and 
classical network flow in the sense that the former is found to be 
within the measure-theoretic closure of 
the class of classical network flows. 
 
Next we briefly describe the idea of the proof. In \cite{KimTone} (see Section \ref{Final} for
comments on \cite{Stuvard-Tonegawa}), the first task
is to construct a time-discrete approximate mean curvature flow. 
For $n=1$, it is proper to say simply as ``curvature'' in place of ``mean curvature'', but we 
may be referring to both 1-dimensional and general $n$-dimensional cases so that we continue
to use ``mean curvature''. 
Let $\Delta t_j$ be a time step size
which converges to $0$ as $j\rightarrow\infty$ and suppose that we inductively have 
$\Gamma_{t}$ at $t=(k-1)\Delta t_j$ for an integer $k$. To obtain $\Gamma_t$ at $t=k\Delta t_j$,
a restricted Lipschitz deformation is first applied to $\Gamma_{(k-1)\Delta t_j}$ so that the
resulting intermediate $\tilde \Gamma_{k\Delta t_j}$ is almost measure-minimized within a length
scale of $O(1/j^2)$. Then one computes a smooth analogue of 
mean curvature vector $h_{\varepsilon_j}$ of $\tilde \Gamma_{k\Delta t_j}$, and moves $\tilde
\Gamma_{k\Delta t_j}$ by $h_{\varepsilon_j}\Delta t_j$ to obtain $\Gamma_{k\Delta t_j}$. The 
parameter $\varepsilon_j$ is much smaller than $1/j^2$, and in view of the length scale 
of measure-minimization, $\varepsilon_j$ is so small that $h_{\varepsilon_j}$ behaves like
the real mean curvature vector of $\tilde\Gamma_{k\Delta t_j}$.  
For each $j$, we continue the induction for $k=1,\ldots,[j/\Delta t_j]$ and define $\Gamma_j(t)$
as a piecewise constant approximate flow for $t\in[0,j]$. The limit of $\{\Gamma_j(t)\}_{t\geq 0}$ as $j\rightarrow 
\infty$ corresponds roughly to the desired Brakke flow. 
Because of the accompanying estimates, for almost all $t\in\R^+$, we can make sure that $\Gamma_j(t)$
is almost measure-minimized within any ball of radius $o(1/j^2)$ and that we have a control of 
square-integral of approximate mean curvature vector. This minimality in the 1-dimensional situation gives
the result that $\Gamma_j(t)$ is very close to either a line or a triple junction within any
ball of radius $o(1/j^2)$. We patch these short lines (or triple junctions) globally. The variation of
these line or triple junction can be controlled by the square-integral of approximate mean curvature
and this gives us $C^{1,\sfrac12}$ control of the curves in a length scale of $O(1)$ independent of $j$.
Then we can make sure that $\Gamma_j(t)$ behaves more or less like a regular network of triple junctions,
and the limit with the $L^2$ curvature bound is expected to have the same type of regularity, except
that some triple junctions may collapse and create junctions with multiple edges. On the other hand,
the angle condition is preserved. To justify this argument, we need to evaluate various errors of 
approximations. We actually need to take a subsequence so that we have $\{j_\ell\}_{\ell=1}^{\infty}$
in place of $\{j\}_{j=1}^\infty$ in the above argument.   
 
The organization of the paper is as follows. In Section \ref{Main} the existence theorem 
of \cite{KimTone} is recalled and the main regularity theorems are stated. Section \ref{Pre}
gathers relevant definitions, lemmas and estimates for approximate MCF from \cite{KimTone}.
In Section \ref{Small}, the main result is Theorem \ref{chlim}, which shows that the 
converging sets are asymptotically close to some measure-minimizing hypersurfaces in any ball of 
radius $o(1/j^2_\ell)$ as $\ell\rightarrow\infty$.
In Section \ref{Large}, specializing for the case of 1-dimension, we prove Theorem \ref{regflat},
which is a regularity theorem for the ``flat'' portion of the varifold with multiplicity. 
Up to this point, we are left with the analysis of isolated singularities, and Section \ref{Pr}
shows Theorem \ref{KTmain2} and \ref{KTmain3} as well as the compactness property.
The last Section \ref{Final} gives some further comments. 

\section{Main results}\label{Main}

We use the same notation stated in \cite[Section 2]{KimTone}. 
Since some of the results obtained in this paper are 
for arbitrary dimensions, we recall the notation for general $n$-dimensional case in the following. 
For $a\in \mathbb R^{n+1}$ and $r>0$, we define
$$
B_r(a):=\{z\in\mathbb R^{n+1}\,:\, |z-a|\leq r\}\,\,\mbox{ and }\,\,
U_r(a):=\{z\in\mathbb R^{n+1}\,:\, |z-a|<r\}.
$$
Let ${\bf G}(n+1,n)$ be the space of $n$-dimensional subspaces of $\mathbb R^{n+1}$ and 
${\bf G}_n(\mathbb R^{n+1}):=\mathbb R^{n+1}\times{\bf G}(n+1,n)$. The symbol 
${\bf V}_n(\mathbb R^{n+1})$ denotes the set of all general 
$n$-varifolds in $\mathbb R^{n+1}$ and ${\bf IV}_n(\mathbb R^{n+1})$ denotes the set of integral 
$n$-varifolds in $\mathbb R^{n+1}$. 
The only minor notational difference from \cite{KimTone} is that we use $z$ for a point in $\R^{n+1}$,
and when we specialize in $n=1$, we use $z=(x,y)\in\R^2$, reserving $x$ and $y$ for the coordinates
on $\R^2$. Moreover, for the consistency of presentation, we state in this section the results for
the Brakke flow obtained in \cite{KimTone}, even though
Theorem \ref{KTmain2}-\ref{KTmain4} hold true for the Brakke flow in \cite{Stuvard-Tonegawa} 
(see Section \ref{Final1}). 

Let $\Omega\in C^2(\mathbb R^{n+1})$
be a weight function satisfying 
\begin{equation}
0<\Omega(z)\leq 1, \,\,|\nabla\Omega(z)|\leq \Cl[c]{c_1}\Omega(z),\,\,\|\nabla^2\Omega(z)\|\leq \Cr{c_1}
\Omega(z)
\label{omega}
\end{equation}
for all $z\in\R^{n+1}$ where $\Cr{c_1}$ is a constant. The function $\Omega$ is introduced
to handle initial data $\Gamma_0$ with infinite $\mathcal H^n$ measure. If 
$\Gamma_0$ has finite $\mathcal H^n$ measure, 
we may take $\Omega(z)
\equiv 1$ and $\Cr{c_1}=0$.
If there exists $c>0$ such that 
$\mathcal H^n(B_R\cap \Gamma_0)\leq \exp(cR)$ for all large $R$, for example, we may choose $\Omega(z)=\exp(-2c\sqrt{1+|z|^2})$ so that \eqref{omega} is satisfied for a suitable $\Cr{c_1}$
depending on $c$ and $n$. With such choice, we have the following \eqref{omegaA}. 
We excerpt the main existence theorem of Brakke flow 
from \cite[Theorem 3.2 and 3.5]{KimTone}:
\begin{thm}
\label{KTmain}
Suppose that $\Gamma_0\subset \R^{n+1}$ is a closed countably $n$-rectifiable set whose 
complement $\R^{n+1}\setminus \Gamma_0$ is not connected and
suppose 
\begin{equation}
\label{omegaA}
\mathcal H^n\mres_{\Omega}(\Gamma_0):=\int_{\Gamma_0}\Omega(z)\,d\mathcal H^n(z)<\infty.
\end{equation}
For some $N\geq 2$, choose a finite collection of non-empty open sets $\{E_{0,i}\}_{i=1}^N$ in $\R^{n+1}$
such that they are mutually disjoint and $\cup_{i=1}^N E_{0,i}=\R^{n+1}\setminus \Gamma_0$. Then 
there exist a family of $n$-dimensional varifolds $\{V_t\}_{t\in \R^+}\subset{\bf V}_n(\R^{n+1})$ and a family of open 
sets $\{E_i(t)\}_{t\in \R^+}$ in $\R^{n+1}$ for each $i=1,\ldots,N$ with the following property. 
\begin{enumerate}
\item $V_0=|\Gamma_0|$ and $E_i(0)=E_{0,i}$ for $i=1,\ldots,N$.
\item For $\mathcal L^1$ a.e.~$t\in\R^+$, $V_t\in {\bf IV}_n(\R^{n+1})$ and $h(\cdot,V_t)
\in L^2(\|V_t\|\mres_\Omega)$.
\item For all $0\leq t_1<t_2<\infty$ and $\phi\in C_c^1(\R^{n+1}\times\R^+;\R^+)$, we have
\begin{equation}\label{bineq}
\|V_t\|(\phi(\cdot,t))\big|_{t=t_1}^{t_2}\leq\int_{t_1}^{t_2}\delta(V_t,\phi(\cdot,t))(
h(\cdot,V_t))+\|V_t\|\big(\frac{\partial\phi}{\partial t}(\cdot,t)\big)\,dt,
\end{equation}
where $\|V_t\|(\phi(\cdot,t))\big|_{t=t_1}^{t_2}:=\|V_{t_2}\|(\phi(\cdot,t_2))-
\|V_{t_1}\|(\phi(\cdot,t_1))$. 
\item $E_1(t),\ldots,E_N(t)\subset\R^{n+1}$ are mutually disjoint open sets for each $t\in \R^+$. 
\item Let $d\mu:=d\|V_t\|dt$. Then $\{z\,:\,(z,t)\in {\rm spt}\,\mu\}=\R^{n+1}\setminus
\cup_{i=1}^N E_i(t)=\cup_{i=1}^N\partial E_i(t)$ for each $t>0$. 
\item ${\rm spt}\,\|V_t\|\subset \{z\,:\,(z,t)\in{\rm spt}\,\mu\}$ for each $t>0$. 
\item $\|V_t\|\geq \|\nabla\chi_{E_i(t)}\|$ for each $t\in\R^+$ and $i=1,\ldots,N$.
\item For each $i=1,\ldots,N$, $z\in\R^{n+1}$ and $R>0$, $\chi_{E_i(t)}\in C([0,\infty);L^1(B_R(z)))$.
\item If $\mathcal H^n(\Gamma_0\setminus \cup_{i=1}^N\partial^*
E_{0,i})=0$, then $\lim_{t\rightarrow 0+}\|V_t\|=\mathcal H^n\mres_{\Gamma_0}$.
\end{enumerate}
Here, $|\Gamma_0|$ is the $n$-dimensional varifold naturally induced from the countably $n$-rectifiable 
set $\Gamma_0$, 
$\|V_t\|$ is the weight measure of $V_t$, $h(x,V_t)$ is the generalized 
mean curvature vector of $V_t$, and
\begin{equation}\label{bineq2}
\delta(V_t,\phi(\cdot,t))(h(\cdot,V_t)):=\int_{\R^{n+1}}
\nabla\phi(z,t)\cdot h(z,V_t)-\phi(z,t)|h(z,V_t)|^2\,d\|V_t\|(z).
\end{equation}
\end{thm}
See more complete description of the properties of $V_t$ and $E_i(t)$ in \cite{KimTone}. The claim
(9) is not stated in \cite{KimTone}, but the same argument in \cite[Proposition 6.10]{Stuvard-Tonegawa} shows 
this continuity property. 

Specializing in the case of $n=1$, the main theorem of the present paper in technical terms is 
the following. 
\begin{thm}
\label{KTmain2}
Suppose $n=1$ and let $\{V_t\}_{t\in\R^+}$ be the Brakke flow obtained in \cite{KimTone} as in
Theorem \ref{KTmain}. Then for almost all $t\in\R^+$, $V_t$ has the following
description at each $z\in{\rm spt}\,\|V_t\|$. After translation of $z$ to the origin and 
an orthogonal rotation, 
there exist $\rho>0$, $k_R,k_L\in \N$, functions $f_{1,R},\ldots,f_{k_R,R}\in W^{2,2}([0,\rho])$ and $f_{1,L},\ldots,f_{k_L,L}
\in W^{2,2}([-\rho,0])$ such that 
$f_{1,R}(x)\leq\ldots\leq f_{k_R,R}(x)$ for $x\in[0,\rho]$ and $f_{1,L}(x)\leq\ldots\leq f_{k_L,L}(x)$ for $x\in [-\rho,0]$, 
\begin{equation}
\label{conc0}
f_{1,R}(0)=\ldots=f_{k_R,R}(0)=f_{1,L}(0)=\ldots=f_{k_L,L}(0)=0,
\end{equation}
\begin{equation}
\label{conc1}
f'_{1,R}(0),\ldots,f'_{k_R,R}(0),f'_{1,L}(0),\ldots,
f'_{k_L,L}(0)\in \{0,\pm \sqrt 3\},
\end{equation}
\begin{equation}
\label{conc3}
\sum_{i=1}^{k_R}\frac{(1,f_{i,R}'(0))}{(1+(f_{i,R}'(0))^2)^{1/2}}+\sum_{i=1}^{k_L}\frac{(-1,-f_{i,L}'(0))}{(1+(f_{i,L}'(0))^2)^{1/2}}=0,
\end{equation}
\begin{equation}
\label{conc2}
\|V_t\|\mres_{B_\rho}=\sum_{i=1}^{k_R} \mathcal H^1\mres_{B_\rho\cap \{(x,f_{i,R}(x))\,:\,x\in[0,\rho]\}}+\sum_{i=1}^{k_L} \mathcal H^1\mres_{B_\rho\cap \{(x,f_{i,L}(x))\,:\,x\in[-\rho,0]\}}.
\end{equation}
\end{thm}
Here, $f\in W^{2,2}([a,b])$ means that 
$f,f',f''$ are in $L^2([a,b])$. It is well-known that 
$W^{2,2}([a,b])\subset C^{1,\sfrac12}([a,b])$ so that $f'$ is well-defined as a 
$\sfrac12$-H\"{o}lder continuous function on $[a,b]$. Note that each term of \eqref{conc3} is the 
inward-pointing unit
tangent vector to the curve at the origin and the claim is that their vector sum is $0$. 
To see the content of the claim
clearly, consider the special case of $k_R=k_L=1$. Then, we have $f_{1,R}$ on $[0,\rho]$
and $f_{1,L}$ on $[-\rho,0]$ with $f_{1,R}(0)=f_{1,L}(0)=0$, $f'_{1,R}(0)=f'_{1,L}(0)\in \{0,\pm\sqrt 3\}$ 
by \eqref{conc0}-\eqref{conc3}. In this case, $f_1:[-\rho,\rho]\rightarrow\R$ 
defined by $f_1(x):=f_{1,R}(x)$ for $x\in[0,\rho]$ and $f_1(x):=f_{1,L}(x)$ for $x\in [-\rho,0]$
is in $W^{2,2}([-\rho,\rho])$ and we have $\|V_t\| \mres_{B_\rho}=\mathcal H^1\mres_{B_\rho
\cap\{(x,f_1(x))\,:\, x\in[-\rho,\rho]\}}$. Hence, this is the case that 
${\rm spt}\,\|V_t\|$ is represented locally as the graph of $f_1$. 
If $k_R=1$ and $k_L=2$ (and similarly for $k_R=2$ and $k_L=1$), 
because of \eqref{conc1} and \eqref{conc3}, we see that 
$f_{1,R}'(0)=0$, $f_{1,L}'(0)=\sqrt 3$ and $f_{2,L}'(0)=-\sqrt 3$ have to be true. This case
corresponds to the triple junction with multiplicity $1$, that is, three curves meet
at the origin with equal angles of 120 degrees. In other cases of $k_R,k_L\geq 2$, we have
$k_R$ curves coming from the right-hand side and $k_L$ curves from the left-hand side, and
they meet at the origin with angles of either $0$, $60$ or $120$ degrees, and so that
\eqref{conc3} holds true. If the derivatives of all the functions are equal at the origin, 
then $k:=k_R=k_L$ by \eqref{conc3}, and this case corresponds to the situation that 
${\rm spt}\,\|V_t\|$ is locally represented by $W^{2,2}$ functions $f_1\leq \ldots\leq f_k$ which are tangent 
at the origin. It is also clear from this description that genuine junctions (with
some non-zero angles between curves meeting at the junction) are isolated, and away from 
them, ${\rm spt}\,\|V_t\|$ is a union of embedded $W^{2,2}$ curves which may 
be tangent to each other but which do not cross transversally. 

If $N=2$ (i.e., the ``two-phase'' case of $\R^2\setminus\Gamma_0=E_{0,1}\cup E_{0,2}$), we can conclude in the next Theorem \ref{KTmain3} that there are no genuine junctions and 
the worst possible irregularities are curves being tangent. 
\begin{thm}\label{KTmain3}
Assume in addition that $N=2$ in Theorem \ref{KTmain2}. 
Then we have $k_R=k_L(=:k)$ and 
$f_{1,R}'(0)=\ldots=f_{k,R}'(0)=f_{1,L}'(0)=\ldots=f_{k,L}'(0)$. 
In particular,
for almost all $t$, ${\rm spt}\,\|V_t\|$ locally consists of a finite number of
embedded $W^{2,2}$ curves which are tangent if they intersect. 
\end{thm}

Next, we note that the class of Brakke flows with the regularity property in Theorem \ref{KTmain2}
(and \ref{KTmain3} for $N=2$) is compact with respect to the natural weak convergence of measures
(see Lemma \ref{comp}).
In particular, we have the following:
\begin{thm}\label{KTmain4}
Any tangent flow of Brakke flow obtained in \cite{KimTone} has the same regularity property
as in Theorem \ref{KTmain2} (and \ref{KTmain3} for $N=2$).
\end{thm}
See \cite{Ton1,White} for the definition and properties of tangent flow. 
One of the corollaries is that the support of any static tangent flow (the one which is homogeneous and 
independent of time) is either a line, a triple junction,
two lines crossing with 60 degrees, or three lines crossing with 60 degrees, all of them 
with possible integer multiplicities. If $N=2$, then the static tangent
flow is a line (with a possible integer multiplicity). We remark that the Brakke flow is $C^{\infty}$ 
in a space-time neighborhood of a point if there exists a static tangent flow at that point
which is a line with multiplicity $1$. This is due to Brakke's partial regularity theorem \cite{Brakke},
the proof of which has been given in \cite{Kasai-Tone,Ton-2}. If there exists a static tangent flow
which is a triple junction with multiplicity $1$, then the Brakke flow is $C^{1,\alpha}$ in the
parabolic sense in a space-time neighborhood of that point for all $\alpha\in (0,1)$ (see \cite{ToneWick} 
for the detail). For further discussion on the main results, see Section \ref{Final}.
 
\section{Preliminaries} \label{Pre}
\subsection{Basic definitions and lemmas}
We recall some essential definitions and lemmas from \cite{KimTone} for the subsequent proofs. 
See \cite[Section 4 \& 5]{KimTone}
for a more comprehensive treatment of these concepts. 
\begin{defi} \label{opdef}
An ordered collection $\mathcal E=\{E_1,\ldots,E_N\}$ of subsets in $\R^{n+1}$
is called an $\Omega$-finite {\it open partition of $N$ elements} if
\begin{itemize}
\item[(a)]
$E_1,\ldots,E_N$ are open and mutually disjoint;
\item[(b)]
$\mathcal H^n\mres_{\Omega}(\R^{n+1}\setminus \cup_{i=1}^N E_i)<\infty$;
\item[(c)]
$\cup_{i=1}^N\partial E_i$ is countably $n$-rectifiable.
\end{itemize}
\end{defi} 
We do not exclude the possibility that some of $E_i$'s are empty set $\emptyset$. 
The set of all $\Omega$-finite open partitions of $N$ elements is denoted by $\mathcal{OP}^N_{\Omega}$.
By Definition \ref{opdef}(b) and (c) as well as $\Omega>0$, we have $\mathcal H^n(B_R\setminus \cup_{i=1}^N
E_i)<\infty$ for all $R>0$ and
\begin{equation}
\label{finiteness}
\cup_{i=1}^N \partial E_i=\R^{n+1}\setminus \cup_{i=1}^N E_i.
\end{equation}
Given $\mathcal E=\{E_1,\ldots,E_N\}\in \mathcal{OP}^N_{\Omega}$, we define (with a slight abuse of 
notation)
\begin{equation}
\label{finiteness2}
\partial\mathcal E:=\big|\cup_{i=1}^N \partial E_i\big|\in {\bf IV}_n(\R^{n+1}),
\end{equation}
which is a unit density varifold naturally induced from the countably $n$-rectifiable
$\cup_{i=1}^N\partial E_i$. We also regard $\partial\mathcal E$ as a set $\cup_{i=1}^N\partial E_i$
with no fear of confusion. The weight measure of $\partial\mathcal E$ satisfies 
\begin{equation}
\label{finiteness3}
\|\partial\mathcal E\|=\mathcal H^{n}\mres_{\cup_{i=1}^N\partial E_i}.
\end{equation}
By Definition \ref{opdef}(b) and \eqref{finiteness}, we have $\|\partial\mathcal E\|(\Omega)<\infty$. 
\begin{defi} \label{adef}
Given $\mathcal E=\{E_1,\ldots,E_N\}\in \mathcal{OP}^N_{\Omega}$, 
a function $f\,:\,\R^{n+1}\rightarrow \R^{n+1}$ is called {\it $\mathcal E$-admissible} if it is 
Lipschitz continuous and satisfies the following.
Define $\tilde E_i:={\rm int}\,(f(E_i))$ for $i=1,\ldots,N$. Then:
\begin{itemize}
\item[(a)] $\tilde E_1,\ldots, \tilde E_N$ are mutually disjoint;
\item[(b)] $\R^{n+1}\setminus\cup_{i=1}^N\tilde E_i\subset f(\cup_{i=1}^N \partial E_i)$;
\item[(c)] $\sup_{z\in\R^{n+1}} |f(z)-z|<\infty$. 
\end{itemize} 
\end{defi}
From the definition, one can prove (see \cite[Lemma 4.4]{KimTone} for the proof) 
\begin{lem}
For $\mathcal E=\{E_1,\ldots,E_N\}\in \mathcal{OP}^N_{\Omega}$ and a $\mathcal E$-admissible function $f$,
define $\tilde{\mathcal E}:=\{\tilde E_1,
\ldots,\tilde E_N\}$ with $\tilde E_i:={\rm int}\,(f(E_i))$. Then we have
$\tilde{\mathcal E}\in \mathcal{OP}^N_{\Omega}$. 
\end{lem}
In the following, $f_{\star}\mathcal E$ denotes the above $\tilde{\mathcal E}$, that is, $f_\star
\mathcal E:=\tilde{\mathcal E}$.  
\begin{defi}\label{testdef}
For every $j\in\N$, the class $\mathcal A_j$ is defined as follows:
\begin{equation}
\label{testdef1}
\begin{split}
\mathcal A_j:=\{\phi\in C^2(\R^{n+1};[0,1])\,&:\,\phi(z)\leq \Omega(z),\,|\nabla\phi(z)|\leq j\phi(z), \\
&\|\nabla^2\phi(z)\|\leq j\phi(z)\,\mbox{ for every }z\in\R^{n+1}\}.
\end{split}
\end{equation}
\end{defi}

\begin{defi}
\label{boldE}
For $\mathcal E=\{E_1,\ldots,E_N\}\in \mathcal{OP}^N_{\Omega}$ and $j\in\N$, 
define ${\bf E}(\mathcal E,j)$ to be the set of all $\mathcal E$-admissible functions $f$
such that:
\begin{itemize}
\item[(a)]
$|f(z)-z|\leq 1/j^2$ for every $z\in \R^{n+1}$;
\item[(b)]
$\mathcal L^{n+1}(E_i\triangle \tilde E_i)\leq 1/j$ for all $i=1,\ldots,N$, where $\tilde E_i={\rm int}\,
(f(E_i))$;
\item[(c)]
$\|\partial f_{\star}\mathcal E\|(\phi)\leq \|\partial\mathcal E\|(\phi)$ for all $\phi\in \mathcal A_j$.
Here, $f_{\star}\mathcal E=\{\tilde E_1,\ldots,\tilde E_N\}$.
\end{itemize}
\end{defi}
Since the identity map $f(z)=z$ is in ${\bf E}(\mathcal E,j)$, ${\bf E}(\mathcal E,j)$ is not empty.
\begin{defi}\label{tridef}
Given $\mathcal E\in \mathcal{OP}^N_{\Omega}$ and $j\in\N$, we define the quantity
\begin{equation}\label{tridef1}
\Delta_j\|\partial\mathcal E\|(\Omega):=\inf_{f\in{\bf E}(\mathcal E,j)}\{\|\partial f_{\star}\mathcal E\|(\Omega)
-\|\partial\mathcal E\|(\Omega)\}.
\end{equation}
\end{defi} 
Since the identity map is in ${\bf E}(\mathcal E,j)$, we have $\Delta_j\|\partial\mathcal E\|(\Omega)\leq 0$.
We also define a localized version of ${\bf E}(\mathcal E,j)$ and $\Delta_j\|\partial\mathcal E\|(\Omega)$ as follows.
\begin{defi}
\label{lboldE}
For $\mathcal E\in \mathcal{OP}^N_{\Omega}$, $j\in\N$ and a compact set $C\subset \R^{n+1}$ we define
\begin{equation}
{\bf E}(\mathcal E,C,j):=\{f\in{\bf E}(\mathcal E,j)\,:\, \{z\,:\,f(z)\neq z\}\cup\{f(z)\,:\,f(z)\neq z\}
\subset C\},
\end{equation}
\begin{equation}
\Delta_j\|\partial\mathcal E\|(C):=\inf_{f\in{\bf E}(\mathcal E,C,j)}(\|\partial f_{\star}\mathcal E\|(C)
-\|\partial\mathcal E\|(C)).
\end{equation}
\end{defi}
We use the following (see \cite[Lemma 4.12]{KimTone} for the proof):
\begin{lem}
\label{echeck}
Suppose $\mathcal E=\{E_1,\ldots,E_N\}\in \mathcal{OP}^N_{\Omega}$, $j\in\N$, $C$ is a compact set,
$f$ is $\mathcal E$-admissible such that 
\begin{itemize}
\item[(a)] 
$\{z\,:\,f(z)\neq z\}\cup\{f(z)\,:\,f(z)\neq z\}\subset C$;
\item[(b)]
$|f(z)-z|\leq 1/j^2$ for all $z\in \R^{n+1}$;
\item[(c)]
$\mathcal L^{n+1}(E_i\triangle \tilde E_i)\leq 1/j$ for all $i=1,\ldots,N$
and where $\tilde E_i={\rm int}\,(f(E_i))$;
\item[(d)]
$\|\partial f_{\star}\mathcal E\|(C)\leq \exp(-j\,{\rm diam}\,C)\|\partial\mathcal E\|(C)$.
\end{itemize}
Then we have $f\in {\bf E}(\mathcal E,C,j)$.
\end{lem}
Let $\psi\in C^{\infty}(\R^{n+1})$ be a radially symmetric function such that 
\begin{equation}
\begin{split}
&\psi(z)=1\mbox{ for }|z|\leq 1/2,\,\,\psi(z)=0\mbox{ for }|z|\geq 1, \\
& 0\leq\psi(z)\leq 1,\,\,|\nabla\psi(z)|\leq 3,\,\,\|\nabla^2\psi(z)\|\leq 9\mbox{ for all }z\in\R^{n+1}.
\end{split}
\end{equation}
Define for each $\varepsilon>0$ 
\begin{equation} \label{tphi}
\hat\Phi_{\varepsilon}(z):=\frac{1}{(2\pi\varepsilon^2)^{\frac{n+1}{2}}}\exp\big(-\frac{|z|^2}{2\varepsilon^2}
\big),\,\,\,\Phi_{\varepsilon}(z):=c(\varepsilon)\psi(z)\hat\Phi_{\varepsilon}(z),
\end{equation}
where the constant $c(\varepsilon)$ is chosen so that $\int_{\R^{n+1}}\Phi_{\varepsilon}(z)\,dz=1$. 
\begin{defi}
For $V\in{\bf V}_n(\R^{n+1})$, we define $\Phi_{\varepsilon}\ast V\in {\bf V}_n(\R^{n+1})$ through
\begin{equation}
(\Phi_{\varepsilon}\ast V)(\phi):=V(\Phi_{\varepsilon}\ast\phi):=\int_{{\bf G}_n(\R^{n+1})}
\int_{\R^{n+1}}\phi(z-\hat z,S)\Phi_{\varepsilon}(\hat z)\,d\hat z dV(z,S)
\end{equation}
for $\phi\in C_c({\bf G}_n(\R^{n+1}))$. For a Radon measure $\mu$ on $\R^{n+1}$, we define a Radon measure $\Phi_{\varepsilon}\ast 
\mu$ on $\R^{n+1}$ through
\begin{equation}
(\Phi_{\varepsilon}\ast\mu)(\phi):=\mu(\Phi_{\varepsilon}\ast\phi):=\int_{\R^{n+1}} \int_{\R^{n+1}}
\phi(z-\hat z)\Phi_{\varepsilon}(\hat z)\,d\hat z d\mu(z)
\end{equation}
for $\phi\in C_c(\R^{n+1})$. 
\label{sf1}
\end{defi}
One may prove the following by the Fubini theorem (see \cite[Section 4 (4.28)]{KimTone}):
\begin{lem} The Radon measure
$\Phi_{\varepsilon}\ast\mu$ may be identified with the $C^{\infty}$ function 
\begin{equation}
(\Phi_{\varepsilon}\ast\mu)(z):=\int_{\R^{n+1}} \Phi_{\varepsilon}(\hat z-z)\,d\mu(\hat z)
\end{equation}
since $(\Phi_{\varepsilon}\ast\mu)(\phi)=\int_{\R^{n+1}} (\Phi_{\varepsilon}\ast\mu)(z)\phi(z)\,dz$ holds
for $\phi\in C_c(\R^{n+1})$. 
\end{lem}
We next define the ``smoothed first variation'' of $V$ as follows:
\begin{defi}
For $V\in {\bf V}_n(\R^{n+1})$, define the following $C^{\infty}$ vector field 
\begin{equation}
\label{defi2}
(\Phi_{\varepsilon}\ast \delta V)(z):=\int_{{\bf G}_n(\R^{n+1})} S(\nabla\Phi_{\varepsilon}(\hat z-z))\,
dV(\hat z,S).
\end{equation}
For any vector field $g\in C_c(\R^{n+1};\R^{n+1})$, we define
\begin{equation}
\label{defi1}
(\Phi_{\varepsilon}\ast\delta V)(g):=\int_{\R^{n+1}} (\Phi_{\varepsilon}\ast\delta V)(z)\cdot
g(z)\,dz.
\end{equation}
\end{defi}
The following can be verified (see \cite[Lemma 4.16]{KimTone} for the proof):
\begin{lem}
For $V\in {\bf V}_n(\R^{n+1})$, we have
\begin{equation}
\label{defi3}
\Phi_{\varepsilon}\ast\|V\|=\|\Phi_{\varepsilon}\ast V\|,
\end{equation}
\begin{equation}
\label{defi4}
\int_{\R^{n+1}}(\Phi_{\varepsilon}\ast\delta V)(z)\cdot g(z)\,dz=\delta V(\Phi_{\varepsilon}\ast g)
\,\,\mbox{ for }\,\,g\in C_c^1(\R^{n+1};\R^{n+1}),
\end{equation}
\begin{equation}
\label{defi5}
\Phi_{\varepsilon}\ast\delta V=\delta(\Phi_{\varepsilon}\ast V).
\end{equation}
\end{lem}
The following is the ``smoothed mean curvature vector'' of $V$:
\begin{defi}
For $V\in {\bf V}_n(\R^{n+1})$ and $\varepsilon>0$, define
\begin{equation}
h_{\varepsilon}(\cdot,V):=-\Phi_{\varepsilon}\ast
\Big(\frac{\Phi_{\varepsilon}\ast\delta V}{\Phi_{\varepsilon}\ast\|V\|+\varepsilon \Omega^{-1}}\Big).
\end{equation}
\end{defi}
We use the following quantity as a proxy for a weighted ``$L^2$-norm of smoothed mean curvature vector'' (see
\cite[Lemma 5.2]{KimTone}):
\begin{lem}
For $V\in {\bf V}_n(\R^{n+1})$ with $\|V\|(\Omega)\leq M$ and $\varepsilon\in (0,\epsilon_1)$
(where $\epsilon_1$ depends only on $n,\Cr{c_1}$ and $M$), we have
\begin{equation}
\int_{\R^{n+1}} \frac{|\Phi_{\varepsilon}\ast\delta V|^2\Omega}{\Phi_{\varepsilon}\ast\|V\|+\varepsilon
\Omega^{-1}}\,dz<\infty.
\end{equation}
\end{lem}
\subsection{Construction of approximate MCF}
In this subsection, we summarize the relevant results of approximate MCF established in 
\cite[Section 6]{KimTone}. The following is \cite[Proposition 6.1]{KimTone} which proves the existence
of a time-discrete approximate MCF starting from $\mathcal E_0$. Here, given $\Gamma_0$
and $\{E_{0,i}\}_{i=1}^N$ as in Theorem \ref{KTmain}, we set $\mathcal E_{0}:=\{E_{0,1},\ldots,E_{0,N}\}$.
Since $\Gamma_0$ is a closed countably $n$-rectifiable set satisfying \eqref{omegaA}, one can see that
$\mathcal E_0\in \mathcal{OP}_{\Omega}^N$ and $\partial\mathcal E_0=|\Gamma_0|$. 
\begin{prop}
\label{caMCF}
Given $\mathcal E_0\in\mathcal{OP}_{\Omega}^N$ and $j\in\N$ with $j\geq \Cr{c_1}$, there exist
$\varepsilon_j\in(0,j^{-6})$, $p_j\in \N$, a family of open partitions $\mathcal E_{j,\ell}\in
\mathcal{OP}_{\Omega}^N$ ($\ell=0,1,\ldots,j2^{p_j}$) with the following property:
\begin{equation}\label{caMCF1}
\mathcal E_{j,0}=\mathcal E_0\,\,\mbox{ for all }j\in \N
\end{equation}
and with the notation of 
\begin{equation}\label{caMCF2}
\Delta t_j:=\frac{1}{2^{p_j}},
\end{equation}
we have
\begin{equation} \label{caMCF3}
\|\partial\mathcal E_{j,\ell}\|(\Omega)\leq \|\partial\mathcal E_0\|(\Omega)\exp\big(\frac{\Cr{c_1}^2}{2}
\ell\Delta t_j\big)+\frac{2\varepsilon_j^{\frac18}}{\Cr{c_1}^2}\big(\exp\big(\frac{\Cr{c_1}^2}{2}\ell\Delta t_j\big)
-1\big),
\end{equation}
\begin{equation}\label{caMCF4}
\begin{split}
\frac{\|\partial\mathcal E_{j,\ell}\|(\Omega)-\|\partial\mathcal E_{j,\ell-1}\|(\Omega)}{\Delta t_j}
&+\frac14\int_{\R^{n+1}}\frac{|\Phi_{\varepsilon_j}\ast \delta(\partial\mathcal E_{j,\ell})|^2\Omega}{
\Phi_{\varepsilon_j}\ast\|\partial\mathcal E_{j,\ell}\|+\varepsilon_j\Omega^{-1}}\,dz \\
&-\frac{(1-j^{-5})}{\Delta t_j}\Delta_j\|\partial\mathcal E_{j,\ell-1}\|(\Omega) 
\leq \varepsilon_j^{\frac18}+\frac{\Cr{c_1}^2}{2}\|\partial\mathcal E_{j,\ell-1}\|(\Omega),
\end{split}
\end{equation}
\begin{equation} \label{caMCF5}
\frac{\|\partial\mathcal E_{j,\ell}\|(\phi)-\|\partial\mathcal E_{j,\ell-1}\|(\phi)}{\Delta t_j}
\leq \delta(\partial\mathcal E_{j,\ell},\phi)(h_{\varepsilon_j}(\cdot,\partial\mathcal E_{j,\ell}))+
\varepsilon_j^{\frac18}
\end{equation}
for $\ell=1,2,\ldots,j2^{p_j}$ and $\phi\in \mathcal A_j$. When $\Cr{c_1}=0$ (so that 
$\|\partial \mathcal E_0\|(\R^{n+1})=\mathcal H^n(\Gamma_0)<\infty$ and we may take
$\Omega=1$), the right-hand side of \eqref{caMCF3}
should be understood as the limit when $\Cr{c_1}\rightarrow 0+$ and is equal to $\mathcal H^n(\Gamma_0)+\varepsilon_j^{\frac18}\ell\Delta t_j$. 
\end{prop}
We remark on the relation of $\varepsilon_j$ and $\Delta t_j$ in the proof of Proposition \ref{caMCF}.
As is explained in \cite[p.83]{KimTone}, 
\begin{equation} \label{caMCF5.5}
\Delta t_j=\frac{1}{2^{p_j}}\in (2^{-1} \varepsilon_j^{3n+20},\varepsilon_j^{3n+20}].
\end{equation}
\begin{defi}
We define for each $j\in\N$ with $j\geq \max\{1,\Cr{c_1}\}$ a family $\mathcal E_j(t)\in \mathcal{OP}_{\Omega}^N$
for $t\in[0,j]$ by
\begin{equation}
\label{caMCF6}
\mathcal E_j (t):=\mathcal E_{j,\ell}\,\,\mbox{ if }\,t\in ((\ell-1)\Delta t_j,\ell \Delta t_j].
\end{equation}
\end{defi}
The next is \cite[Propisition 6.4]{KimTone}: 
\begin{prop}
\label{caMCF7}
There exist a subsequence $\{j_{\ell}\}_{\ell=1}^\infty$ and a family of Radon measures $\{\mu_t\}_{t\in
\R^+}$ on $\R^{n+1}$ such that
\begin{equation}
\label{caMCF8}
\lim_{\ell\rightarrow\infty}\|\partial\mathcal E_{j_\ell} (t)\|(\phi)=\mu_t(\phi)
\end{equation}
for all $\phi\in C_c(\R^{n+1})$ and for all $t\in \R^+$. For all $T<\infty$, we have
\begin{equation}\label{caMCF9}
\limsup_{\ell\rightarrow\infty} \int_0^T\Big(\int_{\R^{n+1}} \frac{|\Phi_{\varepsilon_{j_\ell}}\ast
\delta(\partial\mathcal E_{j_\ell}(t))|^2\Omega}{\Phi_{\varepsilon_{j_\ell}}\ast\|\partial\mathcal E_{j_\ell}
(t)\|+\varepsilon_{j_\ell}\Omega^{-1}}\,dz 
-\frac{1}{\Delta t_{j_\ell}} \Delta_{j_\ell}\|\partial\mathcal E_{j_\ell}(t)\|(\Omega)\Big)\,dt<\infty.
\end{equation}
\end{prop}
Because of \eqref{caMCF9} and Fatou's lemma, for a.e.~$t\in\R^+$, 
we may choose a time-dependent further subsequence (denoted by the same index) $\{j_{\ell}\}_{\ell=1}^{\infty}$
such that 
\begin{equation}
\label{boundht}
\sup_{\ell\in\N}\Big(\int_{\R^{n+1}}\frac{|\Phi_{\e_{j_\ell}}*\delta
(\partial\E_{j_\ell}(t))|^2\Omega}{\Phi_{\e_{j_\ell}}*\|\partial\E_{j_\ell}(t)\|+\e_{j_\ell}\Omega^{-1}}
\, dz 
-\frac{1}{\Delta t_{j_\ell}}\Delta_{j_\ell}\|\partial\E_{j_\ell}(t)\|(\Omega)\Big)
\leq \Cl[c]{c_2}
\end{equation}
for some $\Cr{c_2}$. The following is proved (see \cite[Theorem 8.6, Lemma 9.1, Theorem 9.3]{KimTone}):
\begin{thm} \label{limreg}
The limit of $\{\partial\mathcal E_{j_\ell}(t)\}_{\ell=1}^{\infty}$ satisfying \eqref{boundht} is necessarily
an integral varifold $V_t$ with $\|V_t\|=\mu_t$, 
and $V_t$ has a locally square integrable 
generalized mean curvature. Moreover, $\{V_t\}_{t\in\R^+}$ is a Brakke flow.
\end{thm}
The main purpose of the present paper is to prove
that, if $\|V_t\|$ is obtained as the limit of $\|\partial\mathcal E_{j_\ell}(t)\|$ satisfying \eqref{boundht},
then ${\rm spt}\,\|V_t\|$ consists of locally finite number of $W^{2,2}$
curves with junctions of specific type, namely, these curves at junctions meet either
at 0, 60 or 120 degrees. This proves Theorem \ref{KTmain2} since the Brakke flow obtained
in \cite{KimTone} is precisely the limit of sequence satisfying \eqref{boundht} for a.e.$\,t$. 

For the rest of the paper, dropping the variable $t$, we let $\{\partial\E_{j_\ell}\}_{\ell=1}^{\infty}$ be the subsequence
with the uniform bound \eqref{boundht}. 
\section{Small scale behavior of approximate sequence} \label{Small}

The purpose of this section is to prove that $\partial\mathcal E_{j_\ell}$ is almost measure-minimizing 
within a length scale smaller than $1/j_\ell^2$. This is an expected result due to 
Definition \ref{boldE}(a) and the uniform bound
$-\Delta_{j_\ell}\|\partial\mathcal E_{j_\ell}\|(\Omega)\leq \Cr{c_2} \Delta t_{j_\ell}$ as in \eqref{boundht}.
Since $\Delta t_{j_\ell}\leq \varepsilon_{j_\ell}^{3n+20}<j_\ell^{-6(3n+20)}$ by \eqref{caMCF5.5}
and $\varepsilon_j\in(0,j^{-6})$ in Proposition \ref{caMCF}, 
we have
\begin{equation}
\label{dsmall}
-\Delta_{j_\ell}\|\partial\mathcal E_{j_\ell}\|(\Omega)<\Cr{c_2} j_\ell^{-6(3n+20)},
\end{equation}
and even after
rescaling $\partial\mathcal E_{j_\ell}$ by $1/j_\ell^2$, it is still very close to being 
measure-minimizing
under Lipschitz deformations of admissible class, ${\bf E}(\mathcal E_{j_\ell},j_{\ell})$.  

Suppose throughout this section that $\{z^{(\ell)}\}_{\ell=1}^{\infty}\subset\R^{n+1}$ is a bounded sequence
and suppose that $\{r_\ell\}_{\ell=1}^{\infty}$ is a sequence of positive numbers such that
\begin{equation}
\lim_{\ell\rightarrow\infty}r_\ell(j_\ell)^2=0
\label{rg}
\end{equation}
and 
\begin{equation}
\lim_{\ell\rightarrow\infty} r_\ell(j_\ell)^3=\infty.
\label{rg2}
\end{equation}
It is not necessary but to fix the idea, we set $r_\ell:=1/(j_\ell)^{2.5}$ so that \eqref{rg} and \eqref{rg2} are satisfied. The motivation for the choice of $r_\ell$ is that we want $r_\ell=o(1/j_\ell^2)$ but not too small so that 
$\varepsilon_{j_\ell}\ll r_\ell$. 
For each $\ell\in \N$, we define 
$F_\ell:\R^{n+1}\rightarrow\R^{n+1}$ by 
\begin{equation}
\label{dilf}
F_\ell(z):=\frac{z-z^{(\ell)}}{r_\ell}
\end{equation}
and define
\begin{equation}
\label{vdef}
V_\ell:=(F_\ell)_{\sharp}(\partial \E_{j_\ell}).
\end{equation}
Using Lemma \ref{ub} below, for each $R>0$, we can prove
\begin{equation}
\label{ub2}
\limsup_{\ell\rightarrow\infty}\|V_\ell\|(B_R)\leq \H^n(\partial B_1)\,R^n.
\end{equation}
Once this is proved, by the standard compactness theorem of Radon measures, we have a converging subsequence 
(denoted by the same index) and a limit $V\in {\bf V}_n(\R^{n+1})$, namely, 
\[
\lim_{\ell\rightarrow\infty} V_\ell(\phi)=V(\phi)\]
for all $\phi\in C_c({\bf G}_n(\R^{n+1}))$. 
The main result in this section is the following characterization of this limit $V$. 
\begin{thm}
\label{chlim}
Let $V$ be obtained as above and suppose that $V\neq 0$. 
Then $V$ is measure-minimizing with respect to any 
compact diffeomorphism and belongs to ${\bf IV}_n(\mathbb R^{n+1})$ with unit density. Moreover, $V$ satisfies the following. 
\begin{enumerate}
\item
For $N\geq 3$ and $n=1$, ${\rm spt}\,\|V\|$ is either a line or a triple junction (with three
half-lines) of 120 degrees.  
\item
For $N\geq 3$ and $n\geq 2$, ${\rm spt}\,\|V\|$ consists of three mutually disjoint sets, ${\rm reg}\,V$, 
${\rm sing}_1\,V$ and ${\rm sing}_2\,V$. The set ${\rm reg}\,V$ is relatively open in 
${\rm spt}\,\|V\|$ and is a real-analytic minimal hypersurface. For any point in ${\rm sing}_1\,
V$, there exists a neighborhood in which ${\rm spt}\,\|V\|$ consists of three real-analytic 
minimal hypersurfaces with boundaries which meet along an $n-1$-dimensional real-analytic surface.
The set ${\rm sing}_2\,V$ is a closed set of Hausdorff dimension $\leq n-2$. In the case of $n=2$,
${\rm sing}_2\,V$ is a set of isolated points in $\mathbb R^3$. 
\item 
For $N=2$ and $1\leq n\leq 6$, ${\rm spt}\,\|V\|$ is a hyperplane.
\item
For $N=2$ and $n\geq 7$, ${\rm spt}\,\|V\|$ is a real-analytic minimal hypersurface
away from a closed set of Hausdorff dimension $\leq n-7$ and a set of isolated 
points if $n=7$. 
\end{enumerate}
\end{thm}

One can expect that this should be true due to the ``almost measure-minimizing property'' \eqref{dsmall}, and Theorem \ref{chlim}
is known in a variety of different settings in area-minimizing problems. In fact, the present
setting of the small scale is close to that of 
Almgren's $(F,\varepsilon,\delta)$ minimal sets \cite{Almgren}. On the other hand, since it is not
precisely the same with the use of open partitions and the admissible class, 
we give a self-contained proof (except that we cite results from \cite{KimTone} and well-known
results in geometric measure theory) and the rest of this section is devoted to the proof of Theorem \ref{chlim}. We use results for $n=1$ in the subsequent sections. 

We start with the following upper density ratio bound. 
\begin{lem}
For any $R>0$, \[\limsup_{\ell\rightarrow\infty}\frac{1}{(r_\ell R)^n}\|\partial\E_{j_\ell}\|(B_{r_\ell R}(z^{(\ell)}))\leq \H^n(\partial B_1) .\]
\label{ub}
\end{lem}
\begin{proof}
Suppose that $\ell$ is large so that $r_\ell R<\frac{1}{2(j_\ell)^2}$. Assume for a contradiction that
there exists $\beta>0$ such that 
\begin{equation}
\label{ch2}
\|\partial\E_{j_\ell}\|(B_{r_\ell R}(z^{(\ell)}))\geq
(\beta+\H^n(\partial B_1))(r_\ell R)^n
\end{equation}
for some large $\ell$. Define a Lipschitz map $\hat F:\R^{n+1}\rightarrow\R^{n+1}$ as follows (though 
$\hat F$ depends on $\ell$, we drop the dependence for simplicity). 
Since $U_{r_\ell R}(z^{(\ell)})\setminus
\partial\E_{j_\ell}$ is a non-empty open set, we can choose and fix a ball $B_{\epsilon}(z')\subset U_{r_\ell R}(z^{(\ell)})\setminus
\partial\E_{j_\ell}$ for some $z'$ and $\epsilon>0$. Let $\hat F$ be a retraction map such that $U_{r_\ell R}(z^{(\ell)})\setminus 
U_{\epsilon}(z')$ is projected onto $\partial B_{r_\ell R}(z^{(\ell)})$. The ball $B_{\epsilon}(z')$ is 
expanded onto $B_{r_\ell R}(z^{(\ell)})$ bijectively. For $z\in \R^{n+1}\setminus U_{r_\ell R}(z^{(\ell)})$, define $\hat F(z)=z$. 
Note this $\hat F$ is $\E_{j_\ell}$-admissible (see Definition \ref{adef}) 
as noted in \cite[4.3.4]{KimTone}. We also claim that 
$\hat F\in {\bf E}(\E_{j_\ell},j_\ell)$ (see Definition \ref{boldE}). 
To prove this, we use Lemma \ref{echeck} (note that ${\bf E}(\mathcal E,C,j)\subset{\bf E}(\mathcal E,j)$
by definition).
We take $C$ in the statement as $B_{r_\ell R}(z^{(\ell)})$. The first three conditions (a)-(c) (with 
$j$ there replaced by $j_\ell$) are satisfied for all large $j_\ell$. Here we used $r_\ell R<1/2j_\ell^2$ to prove (b) and (c) of Lemma \ref{echeck}. Thus we only need to check (d)
of Lemma \ref{echeck}:
\begin{equation}
\label{ch1}
\|\partial \hat F_{\star}\E_{j_\ell}\|(B_{r_\ell R}(z^{(\ell)}))\leq \exp(-2 j_\ell r_\ell R)\|\partial\E_{j_\ell}\|(B_{r_\ell R}
(z^{(\ell)})).
\end{equation}
Since $(\partial \hat F_{\star}
\E_{j_\ell})\cap U_{r_\ell R}(z^{(\ell)})=\emptyset$, we have 
\begin{equation}
\label{ch3}
\|\partial \hat F_{\star}\E_{j_\ell}\|(B_{r_\ell R}(z^{(\ell)}))=\|\partial \hat F_{\star}\E_{j_\ell}\|(\partial
B_{r_\ell R}(z^{(\ell)}))\leq \H^n(\partial B_{r_\ell R}(z^{(\ell)}))= (r_\ell R)^n\H^n(\partial B_1).
\end{equation}
By \eqref{rg} (which gives $\exp(-2j_\ell r_\ell R)\approx 1$), \eqref{ch2} and \eqref{ch3}, we have \eqref{ch1} for all large $\ell$. Thus we proved
that $\hat F\in {\bf E}(\E_{j_\ell},j_\ell)$ by Lemma \ref{echeck}. 
Recall that 
\begin{equation}
\label{ch5}
\Delta_{j_\ell}\|\partial\E_{j_\ell}\|(\Omega)=\inf_{f\in {\bf E}(\E_{j_\ell},j_\ell)}
(\|\partial f_{\star}\E_{j_\ell}\|(\Omega)-\|\partial\E_{j_\ell}\|(\Omega))
\end{equation}
which is bounded from below by $-c_2 \Delta t_{j_\ell}$ due to \eqref{boundht}. 
Since $\partial \hat F_{\star}\E_{j_\ell}$ and $\partial \E_{j_\ell}$ coincide outside 
$B_{r_\ell R}(z^{(\ell)})$, we have by \eqref{ch3} and \eqref{ch2} that
\begin{equation}
\label{ch4}
\begin{split}
&\|\partial \hat F_{\star}\E_{j_\ell}\|(\Omega)-\|\partial\E_{j_\ell}\|(\Omega)\\ 
&\leq \Big((\max_{B_{r_\ell R}(z^{(\ell)})}\Omega)\H^n(\partial B_1) -(\min_{B_{r_\ell R}(z^{(\ell)})}\Omega)
(\beta+\H^n(\partial B_1))\Big)(r_\ell R )^n\\
& \leq (\min_{B_{r_\ell R}(z^{(\ell)})}\Omega)\Big(\H^n(\partial B_1)\exp(2\Cr{c_1} r_\ell R) -(\beta+\H^n(\partial B_1))\Big)(r_\ell R)^n \\
&\leq - (\min_{B_{r_\ell R}(z^{(\ell)})}\Omega)\frac{\beta (r_\ell R)^n}{2}
\end{split}
\end{equation}
for all large $\ell$, 
where we also used $\Omega(z)\leq \Omega(z') \exp(\Cr{c_1}|z-z'|)$ which follows from $|\nabla\Omega|\leq 
\Cr{c_1}\Omega$ (cf. \eqref{omega}) and a standard Gr\"{o}nwall type argument. 
By \eqref{boundht}, \eqref{ch5} and \eqref{ch4}, for all large $\ell$, we have
\[
(\min_{B_{r_\ell R}(z^{(\ell)})}\Omega)\beta (r_\ell R)^n\leq 2\Cr{c_2}\Delta t_{j_\ell}.\]
Since $\{z^{(\ell)}\}_{\ell=1}^{\infty}$ is a bounded sequence and $\Omega>0$, we have
$\inf_\ell (\min_{B_{r_\ell R}(z^{(\ell)})}\Omega)>0$.
We now obtain a contradiction since $\Delta t_{j_\ell}\ll (j_\ell)^{-3n}$ and $(j_\ell)^{-3n}\ll (r_\ell)^n$ as $\ell\rightarrow\infty$
by \eqref{dsmall} and \eqref{rg2}.
\end{proof}
This proves \eqref{ub2} and the existence of a limit $n$-dimensional varifold $V$. 
It is also useful to consider the convergence of partitions. For this purpose, 
for each $\mathcal E_{j_\ell}$, write
\begin{equation}
\label{dname}
\{E_{j_\ell,k}\}_{k=1}^{N}=\mathcal E_{j_\ell}.
\end{equation}
Note that $|\cup_{k=1}^N \partial E_{j_\ell,k}|=\partial \mathcal E_{j_\ell}$ and 
each $E_{j_\ell,k}$ satisfies
$\|\nabla\chi_{F_\ell(E_{j_\ell,k})}\|\leq \|V_\ell\|$ by \cite[Proposition 3.62]{Ambrosio}. 
Thus by the compactness theorem of BV functions, there exist a further subsequence 
(denoted by the same index) and
Caccioppoli sets $E_1,\ldots,E_N\subset\mathbb R^{n+1}$ such that
\begin{equation}
\label{dname2}
\chi_{F_\ell(E_{j_\ell,k})}\rightarrow \chi_{E_k}
\end{equation}
for each $k=1,\ldots,N$ in $L^1_{loc}(\mathbb R^{n+1})$ and a.e.~pointwise $z\in\mathbb R^{n+1}$ as $\ell\rightarrow
\infty$. Since $\{F_\ell(E_{j_\ell,k})\}_{k=1}^N$ are open partitions of $\mathbb R^{n+1}$, one can also prove that 
$E_1,\ldots,E_N$ satisfy
\begin{equation}
\label{dname3}
\mathcal L^{n+1}(E_i\cap E_k)=0
\,\,\mbox{ for }i\neq k\mbox{ and }\sum_{k=1}^N \chi_{E_k}=1 \, \,\mbox{a.e.\,\,on }\R^{n+1}.
\end{equation}
By the lower-semicontinuity of the BV semi-norm, we also have
\begin{equation}
\label{dname4}
\|\nabla\chi_{E_k}\|\leq \|V\|
\end{equation}
for all $k=1,\ldots,N$. 

Next we
prove that $V$ is measure-minimizing in the following sense, proving the first
claim of Theorem \ref{chlim}.
\begin{lem} \label{ub3cl}
For any diffeomorphism $g\in C^1(\R^{n+1};\R^{n+1})$ with
$g\mres_{\R^{n+1}\setminus U_R}(z)=z$ for some $R>0$, 
we have
\begin{equation}
\label{ub3}
\|g_{\sharp}V\|(B_R)\geq \|V\|(B_R).
\end{equation}
\end{lem}
\begin{proof}
Suppose on the contrary that there exists a diffeomorphism $g$ as above such that 
$\|g_{\sharp} V\|(B_R)-\|V\|(B_R)< -\beta$ for some $\beta>0$. Then, by the definition of 
the push-forward of varifold, we have
\[
\int_{{\bf G}_n(B_R)} (|\Lambda_n\nabla g\circ S|-1)\, dV(z,S)< -\beta.
\]
Since $|\Lambda_n\nabla g(z)\circ S|-1=0$ for $z\in\R^{n+1}\setminus U_R$, $|\Lambda_n\nabla g(z)\circ
 S|-1$ is an element of
$C_c({\bf G}_n(\R^{n+1}))$. Thus, by the varifold convergence, for all sufficiently large $\ell$, we have
\begin{equation}
\label{ub4}
\|g_{\sharp}V_\ell\|(B_R)-\|V_\ell\|(B_R)=\int_{{\bf G}_n (B_R)}(|\Lambda_n\nabla g(z)\circ S|-1)\, dV_\ell(z,S)<-\beta.
\end{equation}
We want to interpret this inequality in terms of $\partial\E_{j_\ell}$. Consider the map
$\hat F_\ell=(F_\ell)^{-1}\circ g\circ F_\ell:\R^{n+1}\rightarrow \R^{n+1}$, where $F_\ell$ is defined
as in \eqref{dilf}. This is a diffeomorphism which is the identity map on $\R^{n+1}\setminus U_{r_\ell R}(z^{(\ell)})$
and maps $B_{r_\ell R}(z^{(\ell)})$ to itself bijectively. Since $g_{\sharp}V_\ell=(g\circ F_\ell)_{\sharp}
(\partial\E_{j_\ell})$, we have from \eqref{ub4}
\begin{equation}
\label{ub5}
\|\partial(\hat F_\ell)_{\sharp}\E_{j_\ell}\|(B_{r_\ell R}(z^{(\ell)}))-\|\partial\E_{j_\ell}\|(B_{r_\ell R}(z^{(\ell)}))
<-\beta r_\ell^n.
\end{equation}
We claim that $\hat F_\ell\in {\bf E}(\E_{j_\ell},j_\ell)$ for all sufficiently large $\ell$. If this holds,
note that we would obtain a contradiction just as in the proof of Lemma \ref{ub}
by the same argument following \eqref{ch5}, which would conclude the proof. Thus we only need to check $\hat F_\ell\in {\bf E}(\E_{j_\ell},j_\ell)$.
The $\E_{j_\ell}$-admissibility is fine since it is a diffeomorphism. 
Just as before, we use Lemma \ref{echeck} with $C=B_{r_\ell R}(z^{(\ell)})$ there. 
The first three conditions (a)-(c) are satisfied
since $r_\ell R<\frac{1}{2(j_\ell)^2}$. To check (d), that is, 
\[
\|\partial(\hat F_\ell)_{\sharp}\E_{j_\ell}\|(B_{r_\ell R}(z^{(\ell)}))\leq \exp(-2j_\ell r_\ell R)\|\partial
\E_{j_\ell}\|(B_{r_\ell R}(z^{(\ell)})),\]
we use \eqref{ub5}. For above to be true, we only need to see from \eqref{ub5}  that
\[\|\partial\E_{j_\ell}\|(B_{r_\ell R}(z^{(\ell)}))-\beta r_\ell^n\leq \exp(-2j_\ell r_\ell R)\|\partial
\E_{j_\ell}\|(B_{r_\ell R}(z^{(\ell)}))\] 
but this holds true since
\[\|\partial\E_{j_\ell}\|(B_{r_\ell R}(z^{(\ell)}))(1-\exp(-2j_\ell r_\ell R))\leq 2(r_\ell R)^n\H^n(\partial B_1)
(1-\exp(-2j_\ell r_\ell R))\leq \beta r_\ell^n
\]
for all sufficiently large $\ell$, where we used $j_\ell r_\ell\rightarrow 0$ and Lemma \ref{ub}. 
This proves that $\hat F_\ell\in {\bf E}(\E_{j_\ell},j_\ell)$. 
\end{proof}
In particular, we obtain
\begin{prop}
\label{limst}
The limit varifold $V$ is stable and stationary.
\end{prop}
Next, we prove the following lower density ratio bound of the limit varifold $V$ using 
\cite[Proposition 7.2]{KimTone}.
 
\begin{lem}
\label{ldb} 
There exists a constant $\Cl[c]{c_3}\in (0,1)$ depending only on $n$ with the following property.
For any $z\in {\rm spt}\,\|V\|$ and $R>0$, we have $\|V\|(B_R (z))\geq \Cr{c_3} R^n$.
\end{lem}
\begin{proof}
Let $\Cr{c_3}$ be the constant appearing in \cite[Proposition 7.2]{KimTone}
also as $\Cr{c_3}$ and assume for a contradiction that we have some $\hat z\in {\rm spt}\,\|V\|$ and $R>0$
such that $\|V\|(B_R(\hat z))<\Cr{c_3} R^n$. Then, since $\|V_\ell\|\rightarrow\|V\|$ and $\hat z\in
{\rm spt}\,\|V\|$, we have a sequence $\hat z^{(\ell)}\rightarrow\hat z$ such that $\hat z^{(\ell)}\in {\rm spt}\,
\|V_\ell\|$ and 
$\|V_\ell\|(B_R(\hat z^{(\ell)}))<\Cr{c_3} R^n$ for all sufficiently large $\ell$. In terms of $\partial\E_{j_\ell}$, this means
\[
\|\partial\E_{j_\ell}\|(B_{r_\ell R}(z^{(\ell)}+r_\ell\hat z^{(\ell)}))<\Cr{c_3} (r_\ell R)^n
\]
for all sufficiently large $\ell$ with $z^{(\ell)}+r_\ell \hat z^{(\ell)}\in {\rm spt}\,\|\partial\E_{j_\ell}\|$. 
We may apply \cite[Proposition 7.2]{KimTone} to $\E_{j_\ell}$ in $B_{r_\ell R}(z^{(\ell)}+r_\ell\hat z^{(\ell)})$, which 
gives a $\E_{j_\ell}$-admissible function $g_\ell$ and a radius $\hat r_\ell\in [r_\ell R/2, r_\ell R]$ such that,
writing $B^{(\ell)}:=B_{\hat r_\ell}(z^{(\ell)}+r_\ell\hat z^{(\ell)})$, 
\begin{itemize}
\item[(1)] $g_\ell(z)=z$ for $z\in \R^{n+1}\setminus B^{(\ell)}$,
\item[(2)] $g_\ell(z)\in B^{(\ell)}$ for $z\in B^{(\ell)}$,
\item[(3)] $\|\partial (g_\ell)_{\star}\E_{j_\ell}\|(B^{(\ell)})\leq \frac12
\|\partial\E_{j_\ell}\|(B^{(\ell)})$.
\end{itemize}
By Lemma \ref{echeck}, we also have $g_\ell\in {\bf E}(\E_{j_\ell},j_\ell)$
for all large $\ell$. Note that Lemma \ref{echeck}(c) is satisfied since the change takes place only
within $B^{(\ell)}$ and $\mathcal L^{n+1}(B^{(\ell)})\ll 1/j_\ell$ by \eqref{rg}. 
Then we have with \eqref{omega} and (3)
\begin{equation}
\begin{split}
\|\partial (g_\ell)_{\star}\E_{j_\ell}\|(\Omega)-\|\partial\E_{j_\ell}\|(\Omega)&\leq 
\max_{B^{(\ell)}} \Omega\, \|\partial (g_\ell)_{\star}\E_{j_\ell}\|(B^{(\ell)})-\min_{B^{(\ell)}} \Omega\, 
\|\partial \E_{j_\ell}\|(B^{(\ell)}) \\
&
\leq \min_{B^{(\ell)}} \Omega\, \Big(\exp(2 \Cr{c_1} r_\ell R)\frac12 -1\Big)\|\partial\E_{j_\ell}\|(B^{(\ell)}).
\end{split}
\label{ub6}
\end{equation}
Since the left-hand side of \eqref{ub6} is bounded below by $-\Cr{c_2}\Delta t_{j_\ell}$, we have
\begin{equation}
\label{ub7}
(\inf_\ell \min_{B^{(\ell)}}\Omega)\, \Big(1-\exp(2\Cr{c_1} r_\ell R)\frac12\Big)\|\partial\E_{j_\ell}\|(B^{(\ell)})
\leq \Cr{c_2} \Delta t_{j_\ell}.
\end{equation}
Since $\Delta t_{j_\ell}\ll r_\ell^n$ and $B_{r_\ell R/2}(z^{(\ell)}+r_\ell\hat z^{(\ell)})\subset B^{(\ell)}$, 
\eqref{ub7} implies $\|V_\ell\|(B_{R/2}(\hat z^{(\ell)}))
\leq r_\ell^{-n}\|\partial\E_{j_\ell}\|(B^{(\ell)})\rightarrow 0$ as $\ell\rightarrow \infty$. 
Since $\hat z^{(\ell)}\rightarrow\hat z$, this implies that
$\|V\|(U_{R/2}(\hat z))=0$, contradicting $\hat z\in {\rm spt}\,\|V\|$. 
\end{proof}

Proposition \ref{limst}, Lemma \ref{ldb} and Allard's rectifiability theorem \cite[5.5(1)]{Allard} show
that $V$ is rectifiable in particular. We next see that 
\begin{lem}
\label{integral}
The limit varifold $V$ is integral. 
\end{lem}
\begin{proof}
Since $V$ is rectifiable, for $\|V\|$ a.e.~$\hat z\in \R^{n+1}$, $V$ has the approximate
tangent space ${\rm Tan}^n(\|V\|,\hat z)$ and the blow-up of $V$ at $\hat z$ converges to
$\theta^n(\|V\|,\hat z)|{\rm Tan}^n(\|V\|,\hat z)|$ as varifolds. We prove that $\theta(\|V\|,\hat z)\in
\N$ in the following, which proves that $V$ is integral. For simplicity, we write
\begin{equation}
\label{ub8}
\beta:=\theta^n(\|V\|,\hat z)\,\mbox{ and }\, T:={\rm Tan}^n(\|V\|,\hat z).
\end{equation}
For each $\ell\in\N$, define $g_\ell(z):=\ell(z-\hat z)$. Because of the above, we have $\lim_{\ell\rightarrow\infty}
(g_\ell)_{\sharp}V=\beta|T|$. Since $V_\ell\rightarrow V$, we may choose a further subsequence
(denoted by the same index) such that $\lim_{\ell\rightarrow\infty} (g_\ell)_{\sharp}
V_\ell=\beta|T|$. 
Since $V_\ell=(F_\ell)_{\sharp}\partial\E_{j_\ell}$ and 
\[
(g_\ell\circ F_\ell)(z)=\frac{z-z^{(\ell)}-r_\ell \hat z}{(r_\ell/\ell)},
\]
we have
$(g_\ell)_{\sharp}V_\ell=(g_\ell\circ F_\ell)_{\sharp}\partial\E_{j_\ell}$.
We set 
\begin{equation}
\label{rlab}
\tilde r_\ell:=r_\ell/\ell,\,\, \tilde z^{(\ell)}:=z^{(\ell)}+r_\ell\hat z\,\,\mbox{ and }\,\tilde F_\ell(z):=(g_\ell\circ F_\ell)(z)
=\frac{z-\tilde z^{(\ell)}}{\tilde r_\ell}.
\end{equation}
In the following, we assume that $\tilde z^{(\ell)}=0$ for simplicity. The general case can be 
handled with suitable parallel translations and no difficulties arise. 
Writing 
\begin{equation}
\label{rlab3s}
\tilde V_\ell:=(g_\ell)_{\sharp}V_\ell=(\tilde F_\ell)_{\sharp}\partial\E_{j_\ell},
\end{equation}
and by the above discussion and notation, we have
\begin{equation}
\label{rlab2}
\lim_{\ell\rightarrow\infty} \tilde V_\ell=\beta|T|.
\end{equation}
By choosing a further subsequence if necessary, we have
\eqref{rg2} (and \eqref{rg} since $\tilde r_\ell<r_\ell$) with $r_\ell$ replaced by $\tilde r_\ell$, i.e.,
\begin{equation}
\label{rg2s}
\lim_{\ell\rightarrow\infty} \tilde r_\ell(j_\ell)^2=0 \,\,\mbox{ and }\,\,\lim_{\ell\rightarrow\infty} \tilde r_\ell (j_\ell)^3=\infty.
\end{equation}
Suppose that $\nu$ is the smallest positive integer strictly greater than $\beta$, namely,
\begin{equation}
\label{rlab3}
\nu\in \N\,\,\mbox{ and }\,\, \nu\in (\beta,\beta+1].
\end{equation}
We use \cite[Lemma 8.1]{KimTone}. In the assumption of \cite[Lemma 8.1]{KimTone}, we fix $\alpha=1/2$, and 
choose $\zeta\in (0,1)$ so that
\begin{equation}
\label{rlab4}
\nu-\zeta>\beta.
\end{equation}
This choice is possible due to \eqref{rlab3}. Let $\gamma\in (0,1)$ and $j_0\in \N$ be 
constants given as the result of \cite[Lemma 8.1]{KimTone}, which have the following properties.
We write
\begin{equation}
\label{rlab5}
 E^*(r):=\{z\in \R^{n+1}\,:\, |T(z)|\leq r,\, {\rm dist}\,(T^{\perp}(z),Y)\leq (1+R^{-1}r)\rho\}
\end{equation}
for $r,\,R,\,\rho\in (0,\infty)$, $T\in{\bf G}(n+1,n)$, $Y\subset T^{\perp}$
and assume
\begin{itemize}
\item[(1)] $\E=\{E_i\}_{i=1}^N\in \mathcal{OP}_{\Omega}^N$, $j\in \N$ with $j\geq j_0$, 
$R\in (0,\frac12 j^{-2})$, $\rho\in (0,\frac12 j^{-2})$,
\item[(2)] $\rho\geq \alpha R$,
\item[(3)] $Y\subset T^{\perp}$ satisfies $\H^0(Y)\leq \nu$ with ${\rm diam}\,Y< j^{-2}$ 
and $\theta^n(\|\partial\E\|,z')=1$ for all $z'\in Y$, 
\item[(4)]
$\int_{{\bf G}_n(E^*(r))}\|S-T\|\, d(\partial\E)(z,S)\leq \gamma \|\partial\E\|(E^*(r))$
for all $r\in (0,R)$,
\item[(5)]
$\Delta_j\|\partial\E\|(E^*(r))\geq -\gamma\|\partial\E\|(E^*(r))$ for all $r\in (0,R)$. 
\end{itemize}
The conclusion under these assumptions (1)-(5) is that
\begin{equation}
\label{rlab6}
\|\partial\E\|(E^*(R))\geq (\H^0(Y)-\zeta)\omega_n R^n.
\end{equation}
We will apply this result for $\partial \E_{j_\ell}$ with $j=j_\ell\geq j_0$,  
$R=\tilde r_\ell$ and $\rho=2\tilde r_\ell$. Note that $R,\rho\in (0,\frac12 j_\ell^{-2})$ is satisfied due to \eqref{rg2s}.
We consider a cylinder
\[
C_\ell:=\{z\in\R^{n+1}\,:\, |T(z)|\leq \tilde r_\ell,\, |T^{\perp}(z)|\leq
\tilde r_\ell\}
\]
for $\ell\in \N$ and consider the behavior of $\partial\E_{j_\ell}$ in $C_\ell$. Note that 
\[\tilde F_\ell(C_\ell)=\{z\in\R^{n+1}\,:\, |T(z)|\leq 1,\, |T^{\perp}(z)|\leq 1\}.\]
Let 
\[
G_\ell:=\{z\in \partial\E_{j_\ell}\cap C_\ell \, :\, \theta^n(\|\partial\E_{j_\ell}\|,z)=1\}
\]
and note that $\|\partial\E_{j_\ell}\|(C_\ell\setminus G_\ell)=0$ due to the rectifiability of $\partial
\E_{j_\ell}$. We then set 
\[
G_\ell^*:=\{z\in T\,:\, 
\H^0(T^{-1}(z)\cap G_\ell)\geq \nu\}.
\]
We next prove that for all $z\in G_\ell^*$, there exist a set $Y\subset T^{-1}(z)\cap G_\ell$ with
$\H^0(Y)=\nu$ and
$r_z\in (0,\tilde r_\ell)$ such that,
writing $E^*(z,r):=E^*(r)+z$ where $E^*(r)$ is defined with respect to $Y$ and $T$, either
\begin{equation}
\label{rlab7}
\int_{{\bf G}_n(E^*(z,r_z))} \|S-T\|\, d(\partial\E_{j_\ell})>\gamma\|\partial
\E_{j_\ell}\|(E^*(z,r_z))
\end{equation}
or 
\begin{equation}
\label{rlab8}
\Delta_{j_\ell}\|\partial\E_{j_\ell}\|(E^*(z,r_z))<-\gamma\|\partial\E_{j_\ell}\|(E^*(z,r_z)).
\end{equation}
In fact, let $Y\subset T^{-1}(z)\cap G_\ell$ be any subset with $\H^0(Y)=\nu$. By the 
definition of $G_\ell^*$, there exists such a $Y$. If there were no $r_z\in(0,\tilde r_\ell)$
with both \eqref{rlab7} and \eqref{rlab8}, then
we have all the assumptions (1)-(5) satisfied and \eqref{rlab6} shows
\begin{equation}
\label{rlab9}
\|\partial\E_{j_\ell}\|(E^*(z,\tilde r_\ell))\geq (\nu-\zeta)\omega_n\tilde r_\ell^n.
\end{equation}
On the other hand, $E^*(z,\tilde r_\ell)=\{z'\in\R^{n+1} \,:\, |T(z'-z)|\leq \tilde r_\ell,\,
{\rm dist}\,(T^{\perp}(z'),Y)\leq 4\tilde r_\ell\}$, and since $Y$ is not empty
and $Y\subset C_\ell$, we have
\begin{equation}\label{rlab9ex}
E^*(z,\tilde r_\ell)\subset \{z'\in \R^{n+1}\,:\, |T(z'-z)|\leq \tilde r_\ell,\, |T^{\perp}(z')|
\leq 5\tilde r_\ell\}=:\hat E(z,\tilde r_\ell).\end{equation}
Since $\|(\tilde F_\ell)_{\sharp}\partial\E_{j_\ell}\|\rightarrow \|\beta |T|\|=\beta\,\H^n\mres_T$, 
we have
\begin{equation}\label{rlab9a}
\limsup_{\ell\rightarrow\infty} \tilde r_\ell^{-n}\|\partial\E_{j_\ell}\|(\hat E(z,\tilde r_\ell))
=\beta\omega_n.\end{equation}
This shows with \eqref{rlab9} and \eqref{rlab9ex} that $\nu-\zeta\leq \beta$, which is a contradiction to \eqref{rlab4}. 
The convergence \eqref{rlab9a} may be made uniform in $z$. Thus for all sufficiently large $\ell$ and for all
$z\in G_\ell^*$ and $Y$ as above, we proved that either \eqref{rlab7} or \eqref{rlab8} hold 
for some $r_z\in (0,\tilde r_\ell)$. Now we use the Besicovitch covering theorem to the
family of $n$-dimensional closed balls $\{B^n_{r_z}(z)\,:\, z\in G_\ell^*\}$ in $T$. 
Then we have a subfamilies $\mathcal C_1,\ldots,\mathcal C_{{\bf B}(n)}$ each of which
consists of a family of disjoint balls and that 
\[
G_\ell^*\subset \cup_{i=1}^{{\bf B}(n)}\cup_{B^n_{r_z}(z)\in \mathcal C_i} B^n_{r_z}(z).
\]
Because of the definitions of $E^*(z,r)$ and $E^*(r)$, we have for any $z\in G_\ell^*$
\[
T^{-1}(B^n_{r_z}(z))\cap C_\ell\subset E^*(z,r_z)\subset T^{-1}(B^n_{r_z}(z)) \cap \{z'\in \R^{n+1}\, :\,
|T^{\perp}(z')|\leq 5\tilde r_\ell\}.
\]
Thus, we have
\begin{equation}
\label{rlab10}
\begin{split}
&\|\partial\E_{j_\ell}\|(T^{-1}(G_\ell^*) \cap C_\ell)\\ & \leq \|\partial\E_{j_\ell}\|\big(\cup_{i=1}^{{\bf B}(n)} 
\cup_{B^n_{r_z}(z)\in\mathcal C_i} T^{-1}(B^n_{r_z}(z))\cap C_\ell\big) \\
&\leq \sum_{i=1}^{{\bf B}(n)} \sum_{B^n_{r_z}(z)\in \mathcal C_i}
\|\partial\E_{j_\ell}\|(E^*(z,r_z)) \\
&\leq \sum_{i=1}^{{\bf B}(n)}\sum_{B^n_{r_z}(z)\in\mathcal C_i} 
\gamma^{-1}\Big(\int_{{\bf G}_n(E^*(z,r_z))}\|S-T\|\, d(\partial\E_{j_\ell})-\Delta_{j_\ell}
\|\partial\E_{j_\ell}\|(E^*(z,r_z))\Big) \\
&\leq {\bf B}(n)\gamma^{-1}\Big(\int_{{\bf G}_n(\hat C_\ell)}\|S-T\|\, d(\partial\E_{j_\ell})-\Delta_{j_\ell}
\|\partial\E_{j_\ell}\|(\hat C_\ell)\Big),
\end{split}
\end{equation}
where $\hat C_\ell:=\{z\in \R^{n+1}\,:\, |T(z)|\leq 2\tilde r_\ell,\, |T^{\perp}(z)|\leq 5\tilde r_\ell\}$. 
Since $(\tilde F_\ell)_{\sharp}\partial\E_{j_\ell}$ converges to $\beta|T|$ as varifolds, we have
\[
\lim_{\ell\rightarrow\infty} \tilde r_\ell^{-n}\int_{{\bf G}_n(\hat C_\ell)} \|S-T\|\, d(\partial\E_{j_\ell})(z,S)
=0
\]
and we conclude that
\begin{equation}
\label{rlab11}
\lim_{\ell\rightarrow\infty} \tilde r_\ell^{-n}\|\partial\E_{j_\ell}\|(T^{-1}(G_\ell^*)\cap C_\ell)=0.
\end{equation}
We have
\[
\|T_{\sharp}\partial\E_{j_\ell}\|(C_\ell)=\int_{T\cap C_\ell}\, \H^0(T^{-1}(z)\cap G_\ell)\, d\H^n(z).
\]
Since $\H^0(T^{-1}(z)\cap G_\ell)\leq \nu-1$ on $T\setminus G_\ell^*$, and because of \eqref{rlab11}, we 
obtain
\begin{equation}
\label{rlab12}
\begin{split}
\lim_{\ell\rightarrow\infty} \tilde r_\ell^{-n} \|T_{\sharp}\partial\E_{j_\ell}\|(C_\ell) &
\leq \limsup_{\ell\rightarrow\infty}\tilde r_\ell^{-n} \int_{C_\ell\cap T\setminus G_\ell^*} \H^0(T^{-1}(z)\cap G_\ell)\, d\H^n(z) \\ &\leq (\nu-1)\lim_{\ell\rightarrow\infty} \tilde r_\ell^{-n}\H^n(T\cap C_\ell)=(\nu-1)\omega_n.
\end{split}
\end{equation}
On the other hand, we have
\[
\lim_{\ell\rightarrow\infty} \tilde r_\ell^{-n}\|T_{\sharp}\partial\E_{j_\ell}\|(C_\ell)
=\lim_{\ell\rightarrow\infty}\|T_{\sharp} \tilde V_\ell\|(\tilde F_\ell(C_\ell))=\beta \||T|\|(B_1^n)=\beta\omega_n.
\]
This proves that $\beta\leq \nu-1$. By \eqref{rlab3}, we have $\beta=\nu-1$ and $\beta\in \N$. 
\end{proof}
We next show
\begin{lem}
\label{uniden}
The limit varifold $V$ is a unit density varifold, that is, $\theta^n(\|V\|,z)=1$ for $\|V\|$ a.e.~$z$.
Moreover, any tangent cone and any blow-down limit of $V$ are also unit density varifolds.
\end{lem}
\begin{proof}
Suppose the contrary. Then there exists $z\in\R^{n+1}$ such that the blowup of $V$ at $z$
converges to $\nu|T|$ for some $\nu\in \N$ with $\nu\geq 2$ and $T\in {\bf G}(n+1,n)$. 
As we saw in the proof of Lemma \ref{integral}, let $\tilde F_\ell$ be chosen so that
$(\tilde F_\ell)_{\sharp}\partial\E_{j_\ell}\rightarrow \nu|T|$. Without loss of generality, we may
assume that $T=\R^n\times\{0\}\subset \R^{n+1}$. In the following, we define a Lipschitz map
which reduces the measure in the cylinder of radius 1. Namely, in such a cylinder, we have 
a measure of $(\tilde F_\ell)_{\sharp}\partial \mathcal E_{j_\ell}$ close to that of $\nu$ parallel discs and we will reduce it to that of one disc. Let $\delta>0$ be a 
small number to be chosen. Write $z=(z_1,\ldots,z_n,z_{n+1})=(\hat z,z_{n+1})$. 
In the cylinder $|\hat z|\leq 1$, define
\begin{equation}
\label{def-g1}
g(\hat z,z_{n+1})=\left\{
\begin{array}{ll}
(\hat z,z_{n+1}) & \mbox{if } |z_{n+1}|\geq \delta,\\
(\hat z,0) & \mbox{if } |z_{n+1}|\leq \frac{\delta}{2},\\
(\hat z, 2z_{n+1}-\delta) & \mbox{if } \frac{\delta}{2}\leq 
z_{n+1}\leq \delta,\\
(\hat z, 2z_{n+1}+\delta) & \mbox{if }-\delta\leq
z_{n+1}\leq -\frac{\delta}{2},
\end{array}\right.
\end{equation}
and in the annular region $1\leq |\hat z|\leq 1+\delta$, define
\begin{equation}
\label{def-g2}
g(\hat z,z_{n+1})=\left\{
\begin{array}{ll}
(\hat z, z_{n+1}) & \mbox{if }|z_{n+1}|\geq \delta\mbox{ or }|z_{n+1}|\leq|\hat z|-1,\\
(\hat z, |\hat z|-1) & \mbox{if }|\hat z|-1\leq 
z_{n+1}\leq \frac{|\hat z|-1}{2}+\frac{\delta}{2},\\
(\hat z, 2z_{n+1}-\delta) & \mbox{if }\frac{|\hat z|-1}{2}+\frac{\delta}{2}\leq z_{n+1}\leq \delta,\\
(\hat z, 1-|\hat z|) & \mbox{if }\frac{1-|\hat z|}{2}-\frac{\delta}{2}\leq z_{n+1}\leq 1-|\hat z|,\\
(\hat z, 2z_{n+1}+\delta) & \mbox{if }-\delta\leq z_{n+1}\leq \frac{1-|\hat z|}{2}-\frac{\delta}{2}.
\end{array}\right.
\end{equation}
\begin{figure}[!h]
\centering
\includegraphics[width=9cm]{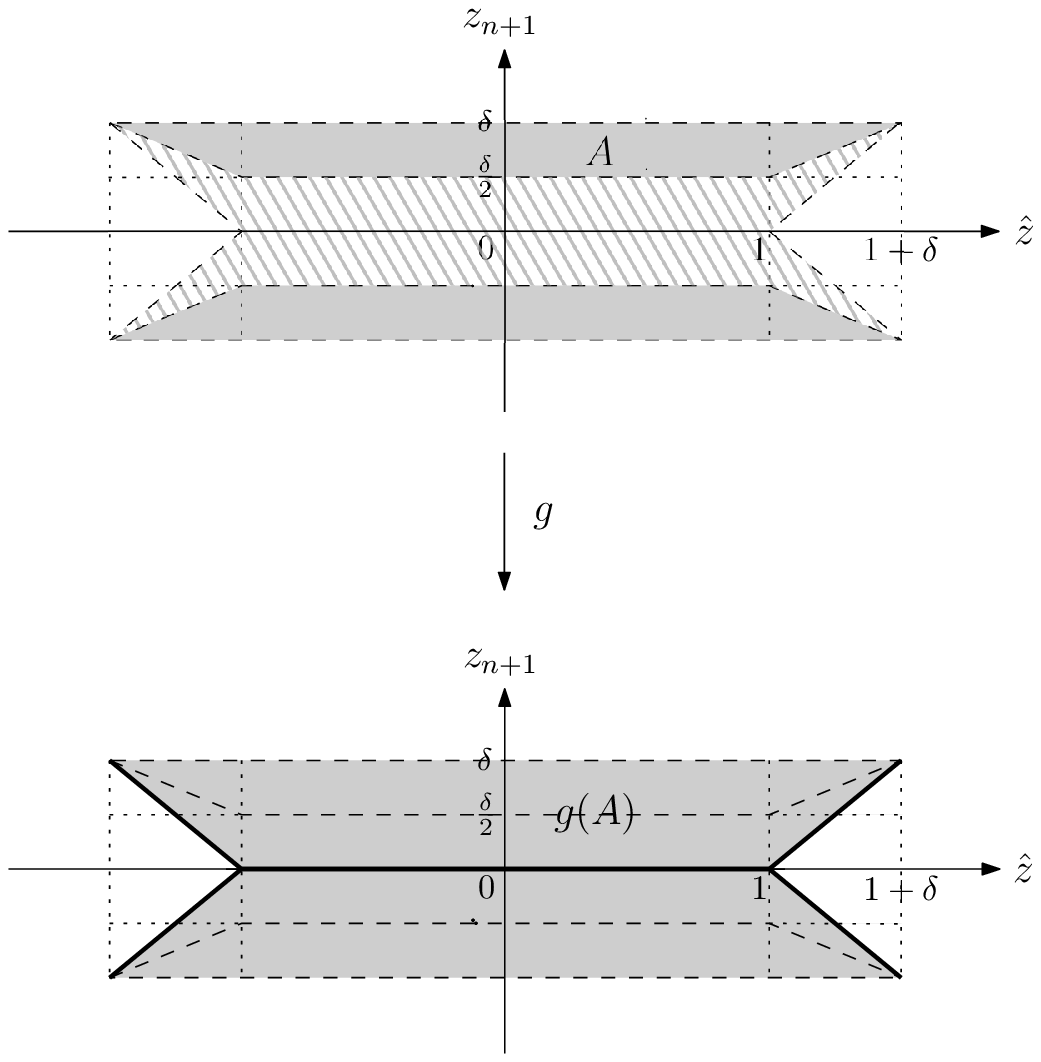}              
\caption{}
\label{Fig2}
\end{figure}
For $z$ with $|\hat z|>1+\delta$, define $g(z)=z$. 
The map $g$ is defined so that it 
crushes the region $\{z : |\hat z|\leq 1, |z_{n+1}|\leq \frac{\delta}{2}\}$ to the disc 
$B_1^{n}\subset T$ and stretches the top and the bottom. In the annular region, a small 
triangular region is also crushed to an $n$-dimensional set (Figure \ref{Fig2}). 
Since $g$ is a retraction, 
one can check that it is an admissible map for any open partition. Now consider 
$\partial(g\circ \tilde F_\ell)_{\star}\E_{j_\ell}$ for all large $\ell$ and compare its measure
to $(\tilde F_\ell)_{\sharp}\partial\E_{j_\ell}$. Since $(\tilde F_\ell)_{\sharp}\partial\E_{j_\ell}$
converges to $\nu|T|$, we have 
\begin{equation}
\label{tokei1}
\|(\tilde F_\ell)_{\sharp}\partial\E_{j_\ell}\|(\{z \,:\,
|\hat z|\leq 1+\delta,\,1\geq |z_{n+1}|\geq \delta/2\})\rightarrow 0
\end{equation}
as $\ell\rightarrow \infty$. Let $A\subset \R^{n+1}$ be open set defined by
\begin{equation*}
A=\big\{|\hat z|\leq 1, \, \frac{\delta}{2}<|z_{n+1}|<\delta\big\}\cup
\big\{1\leq |\hat z|<1+\delta,\,\frac{|\hat z|-1}{2}+\frac{\delta}{2}
<|z_{n+1}|<\delta\big\}.
\end{equation*}
The map $g$ stretches the set $A$ by twice, and $g$ maps $A$ bijectively to its image
$g(A)$. Thus, using \eqref{tokei1}, we obtain 
\begin{equation}
\label{tokei2}
\|\partial (g\circ \tilde F_\ell)_{\star} \E_{j_\ell}\|(g(A))
=\|(g\circ\tilde F_\ell)_{\sharp}\partial \E_{j_\ell}\|(g(A))
\leq 2^n \|(\tilde F_\ell)_{\sharp}\partial \E_{j_\ell}\|(A)\rightarrow 0
\end{equation}
as $\ell\rightarrow \infty$. Because of the property of $g$, in the region
$\{|\hat z|< 1+\delta,\, |z_{n+1}|<\delta\}\setminus g(A)$, we have
\begin{equation}
\label{tokei3}
\partial (g\circ\tilde F_\ell)_{\star}\E_{j_\ell}\subset 
\partial (g(A))\cup \big((\tilde F_\ell)_{\sharp}\partial\E_{j_\ell}\setminus \overline{ g(A)}\big).
\end{equation}
Note that $\H^n(\partial(g(A))\cap \{|\hat z|<1+\delta,\,  |z_{n+1}|<\delta\})\leq 
\omega_n+ c(n)\delta$. Combining \eqref{tokei2} and \eqref{tokei3}, we have
\begin{equation}
\label{tokei4}
\limsup_{\ell\rightarrow\infty} \|\partial (g\circ \tilde F_\ell)_{\star}\E_{j_\ell}\|
(\overline{g(A)})
\leq \omega_n+c(n)\delta.
\end{equation}
Since $\|(\tilde F_\ell)_{\sharp}\partial\E_{j_\ell}\|(\{|\hat z|<1+\delta,\,|z_{n+1}|<
\delta\})\rightarrow (1+\delta)^n\omega_n \nu$ with $\nu\geq 2$ as $\ell\rightarrow\infty$, for suitably
small $\delta$ depending only on $n$, we see that $g$ reduces the 
measure by definite amount for all large $\ell$. By the similar argument 
as in the previous proofs, we may obtain a contradiction to the almost measure-minimizing 
property. The argument up to this point is also valid for any limit varifold obtained from $V$
under dilations centered at any point. Since $V$ is stationary, the tangent cone is 
integral stationary varifold and we can prove that it has unit density for almost everywhere.
By \eqref{ub2}, we have $\|V\|(B_R)\leq\mathcal H^n(\partial B_1)R^n$ for all $R>0$, so that
the blow-down limit centered at arbitrary point exists and is integral stationary varifold which
is also a cone (the latter can be proved by the same argument for the proof of tangent cone 
being a cone). Thus any tangent cone or blow-down limit of $V$ has the same density property.
\end{proof}

Next we see that $V$ is ``two-sided'' in the following sense. Here recall the definition 
of $E_1,\ldots,E_N$ as in \eqref{dname2}. The stationarity of $V$ shows that
$\|V\|=\mathcal H^n\mres_{{\rm spt}\,\|V\|}$ (see \cite[17.9(1)]{Simon}). By \eqref{dname4}, 
each $\chi_{E_k}$ is constant on each connected component of $\R^{n+1}\setminus{\rm spt}\,\|V\|$.
By \eqref{dname3}, there is only one $k$ such that $\chi_{E_k}=1$ there. Thus, we may 
regard each $E_k$ to be open and $\cup_{k=1}^N E_k=\R^{n+1}\setminus {\rm spt}\,\|V\|$. 
\begin{lem}
\label{two-sided}
For any point $z$ with $\theta^n(\|V\|,z)=1$, there are two different indices $k_1,k_2\in\{1,\ldots,N\}$
and some $r>0$ such that $\mathcal L^{n+1}(U_r(z)\cap E_{k_i})>0$ for $i=1,2$ and
$U_r(z)\setminus{\rm spt}\|V\|=U_r(z)\cup (E_{k_1}\cup E_{k_2})$. 
\end{lem}
\begin{proof}
Since $V$ is stationary and integral, by the Allard regularity theorem \cite{Allard}, ${\rm spt}\,\|V\|$ is a 
real-analytic minimal hypersurface in some neighborhood of such $z$. Thus for sufficiently small 
neighborhood, $U_r(z)\setminus{\rm spt}\,\|V\|$ consists of two connected non-empty open sets.
To prove the claim, we only need to 
prove that this set is $(E_{k_1}\cup E_{k_2})\cap U_r(z)$ with $k_1\neq k_2$. 
For a contradiction, assume $k_1=k_2$ and we may assume $k_1=k_2=1$ 
without loss of generality. We proceed just as in the proof of Lemma \ref{uniden}. Since 
$\chi_{F_\ell(E_{j_\ell,1})}\rightarrow \chi_{E_1}=1$ in $L^1(U_r(z))$ as $\ell\rightarrow \infty$, 
we may choose $\{\tilde F_\ell\}$ so that $(\tilde F_\ell)_{\sharp}\partial\mathcal E_{j_\ell}
\rightarrow |T|$ (note that the multiplicity is $1$) and additionally so that 
\begin{equation}
\label{two-side}
\chi_{\tilde F_\ell (E_{j_\ell,1})}\rightarrow 1 \mbox{ in }L^1_{loc}(\mathbb R^{n+1}).
\end{equation} 
We use the same Lipschitz map $g$ as in Lemma \ref{uniden} to reduce the measure as follows.
Because of \eqref{two-side} and Fubini theorem, for a.e.~$\delta>0$ (as in the 
proof of Lemma \ref{uniden}), we have
\begin{equation}
\label{two-side2}
\lim_{\ell\rightarrow\infty} \mathcal H^n(\{z\,:\, |\hat z|\leq 1, |z_{n+1}|=
\delta/2\}\setminus \tilde F_\ell(E_{j_\ell,1}))=0.
\end{equation}
With such $\delta$, let $g$ be defined as in \eqref{def-g1} and \eqref{def-g2}. 
Consider $(g\circ \tilde F_\ell)_{\star}\mathcal E_{j_\ell}=:\{\tilde E_{j_\ell,k}\}_{k=1}^N$.
Recall that $\tilde E_{j_\ell,k}$ is the set of interior points of $(g\circ \tilde F_\ell)(E_{j_\ell,k})$
for each $k=1,\ldots,N$. We pay special attention to $\tilde E_{j_\ell,1}$. 
Suppose that $(\hat z,\delta/2)\in \tilde F_\ell(E_{j_\ell,1})$ and $(\hat z,-\delta/2)\in 
\tilde F_\ell(E_{j_\ell,1})$ for $\hat z$ with $|\hat z|<1$. Since $\tilde F_\ell(E_{j_\ell,1})$ is open,
there are some neighborhoods of $(\hat z,\pm \delta/2)$ which also belong to $\tilde F_\ell(E_{j_\ell,1})$.
The map $g$ sends both $(\hat z,\pm\delta/2)$ to the same point $(\hat z,0)$
and there will be some neighborhood of $(\hat z,0)$ which is included in 
$(g\circ \tilde F_\ell)(E_{j_\ell,1})$. Thus it is an interior point of $(g\circ \tilde F_\ell)(E_{j_\ell,1})$ and
$(\hat z,0)\in \tilde E_{j_\ell,1}$. 
Because of \eqref{two-side2} and the preceding discussion, $\tilde E_{j_\ell,1}$ has the property that
\begin{equation}
\label{two-side3}
\lim_{\ell\rightarrow\infty} \mathcal H^n(\{(\hat z,0) \, :\, |\hat z|\leq 1\}\setminus \tilde E_{j_\ell,1})
=0.
\end{equation}
Since $\partial(g\circ \tilde F_\ell)_{\star}\mathcal E_{j_\ell}\cap \{(\hat z,0)\,:\, |\hat z|\leq 1\}
\subset \{(\hat z,0)\notin \tilde E_{j_\ell,1}\,:\, |\hat z|\leq 1\}$, in the estimate of \eqref{tokei4},
we may obtain
\begin{equation}
\label{two-side4}
\limsup_{\ell\rightarrow\infty} \|\partial(g\circ\tilde F_\ell)_{\star} \mathcal E_{j_\ell}\|(\overline{g(A)})
\leq c(n)\delta
\end{equation}
instead, without $\omega_n$ on the right-hand side. Since $\|(\tilde F_\ell)_{\sharp}\partial\mathcal
E_{j_\ell}\|(\{|\hat z|< 1+\delta, |z_{n+1}|<\delta\})\rightarrow (1+\delta)^n\omega_n$ as $\ell\rightarrow
\infty$, we again see that $g$ reduces the measure by definite amount and we may argue similarly
to obtain a contradiction. Thus we proved $k_1\neq k_2$, ending the proof. 
\end{proof}
\begin{lem}
\label{triplelem}
Suppose that $V$ has a tangent cone or blow-down limit which is 
given as an orthogonal rotation of
$|S\times\mathbb R^{n-1}|\in {\bf IV}_n(\mathbb R^{n+1})$, where $S\subset \mathbb R^2$ is 
a finite union of half-lines emanating from the origin. Then, (i) $S$ consists of 
three half-lines with equal 120 degree angles or, (ii) $S$ is a line through the origin.  
\end{lem}
\begin{proof}
Consider the case $n=1$. We need to exclude the possibility that $S$ consists of more than three 
half-lines. For a contradiction, assume the contrary. If there are four or more half-lines, 
then there would be at least one pair of half-lines intersecting 
with an angle $\leq 90^\circ$. Since one 
can reduce the length of such pair of half-lines by a Lipschitz map, we may follow the similar
\begin{figure}[!h]
\centering
\includegraphics[width=15cm]{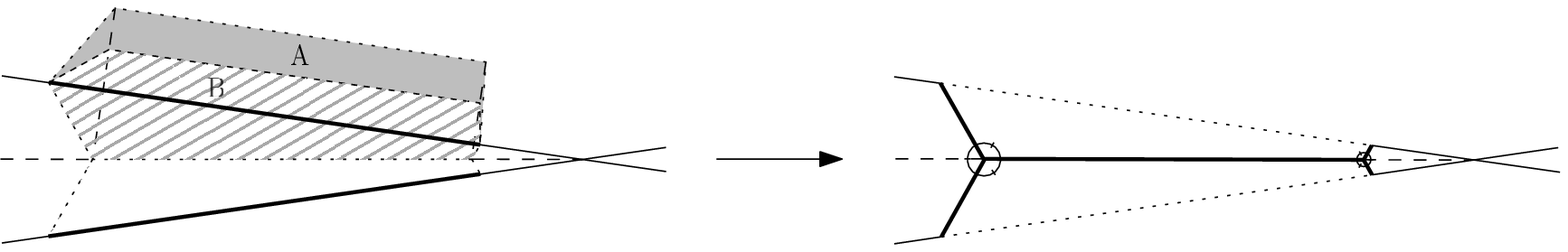}
\caption{}
\label{Fig3}
\end{figure}
procedure as in Lemma \ref{uniden}. Thus let $\{\tilde F_\ell\}$ be chosen so that $(\tilde F_\ell)_{\sharp}
\partial\mathcal E_{j_\ell}\rightarrow |S|$ as $\ell\rightarrow\infty$. 
For the Lipschitz map, we simply give the schematic picture in Figure \ref{Fig3} which describes the
map on the
upper half part of the bisected region between the two half-lines. On the lower half, the map
is symmetrically defined. Let $A$ and $B$ be
closed sets indicated in Figure \ref{Fig3}. The Lipschitz map $g$ is defined so that 
the region A is piece-wise smoothly expanded to cover the union of A and B bijectively 
while the region B is 
crushed to the solid line segments in the figure on the right-hand side. The two triple 
junctions are positioned so that they are sufficiently far apart and so that the map is 
length reducing map. Except for the neighborhood of endpoints, the region A is a
away from $S$ so that $\lim_{\ell\rightarrow\infty}\|(\tilde F_\ell)_{\sharp}\partial\mathcal E_{j_\ell}\|(A)
=0$. Consider $\partial(g\circ\tilde F_\ell)_{\star}\mathcal E_{j_\ell}$ inside ${\rm int}\,(A\cup B)$. 
Between ${\rm int}\,A$ and ${\rm int}\,(A\cup B)$, $g$ is bijective and 
\begin{equation}
\label{lipg}
\|\partial(g\circ\tilde F_\ell)_{\star}\mathcal E_{j_\ell}\|({\rm int}\,(A\cup B))\leq 
c({\rm Lip}\,g)\|(\tilde F_\ell)_{\sharp}\partial\mathcal E_{j_\ell}\|(A)\rightarrow 0
\end{equation}
as $\ell\rightarrow\infty$. Thus most of the measure of $\|\partial(g\circ \tilde F_\ell)_{\star}
\mathcal E_{j_\ell}\|$ in $A\cup B$ lies on $B\setminus (A\cup {\rm int}\,B)$.
Now let $C$ be the union of $A\cup B$ and its reflection with respect to the bisecting line,
and let $g$ be defined symmetrically on $C\setminus (A\cup B)$. We have the most of the measure
of $\|\partial(g\circ\tilde F_\ell)_{\star}\mathcal E_{j_\ell}\|$ in $C$ on the solid line segments, 
and the $\mathcal H^1$ measure is strictly less than $\mathcal H^1(S\cap C)$.
Thus for all
large $\ell$, $\|\partial(g\circ\tilde F_\ell)_{\star}\mathcal E_{j_\ell}\|(C)$ is strictly
smaller than $\|(\tilde F_\ell)_{\sharp}\partial\mathcal E_{j_\ell}\|(C)$, which would
lead to a contradiction to the almost measure-minimizing property just as in Lemma \ref{uniden}. 
For $n>1$, the picture is similar. If $S$ has more than three half-lines, again we have
a pair of two half-hyperplanes intersecting with an angle $\leq 90^{\circ}$. Then one 
can construct a Lipschitz map $g$ which is the same on 
$\mathbb R^2\times\{0\}$ as $g$ for $n=1$ and which is homogeneously extended in the $\mathbb R^{n-1}$
direction on $\{(z',\hat z)\in\mathbb R^2\times \mathbb R^{n-1}\,:\, |\hat z|\leq R\}$ for a large $R$. 
On $\{(z',\hat z)\in\mathbb R^2\times\mathbb R^{n-1}\,:\, R\leq |\hat z|\leq R+1\}$, as $|\hat z|$ changes from
$R$ to $R+1$, one can change $g$ on $\mathbb R^2\times\{\hat z\}$ piece-wise smoothly to the identity map.
Then, using the same idea, the reduction of measure in $\{(z',\hat z)\,:\, |\hat z|\leq R\}$ due to the map $g$
is proportional to $R^{n-1}$ for all large $\ell$. The possible increase of mass in $\{(z',\hat z)\,:\,
R\leq |\hat z|\leq R+1\}$ can be estimated by a constant multiple of $R^{n-2}$ and we again see a 
definite amount of reduction for sufficiently large $R$, from which we may derive a contradiction
as in Lemma \ref{uniden}.
\end{proof}

Finally we give a proof of Theorem \ref{chlim}. 
\begin{proof}
The claim that the limit $V$ is measure-minimizing is proved in Lemma \ref{ub3cl}, and the
claim of unit density is proved in Lemma \ref{uniden}. We prove (1)-(4) next. 
For $n=1$, since $V$ is a 1-dimensional stationary integral varifold, \cite{AA} shows that
${\rm spt}\,\|V\|$ consists of locally finite line segments 
with discrete junctions. At these junctions, Lemma \ref{triplelem} shows that they have to be
triple junctions. We claim that there can be only one triple junction at most. If there
is no triple junction, then it is a line, so assume that we have at least one junction. To see that
there cannot be more than one junction, 
we may consider a blow-down limit centered at one of the junctions. We shift the junction
to the origin. By the monotonicity formula,
we know that such limit $\tilde V$ is a cone, and again by Lemma \ref{triplelem},
the limit is also a triple junction (it cannot be a line). This implies that there exists a sequence $r_i\rightarrow\infty$
such that 
$\lim_{i\rightarrow\infty}\|V\|(B_{r_i})/(2r_i)=3/2$. Since the origin is a triple junction, we have
$\|V\|(B_r)/(2r)=3/2$ for sufficiently small
$r$. But then we have $\|V\|(B_r)/(2r)=3/2$ for all $r>0$ by the monotonicity formula and $V$ itself
has to be a cone. Thus $V$ can have at most one triple junction. This proves (1). 
The claim (2) follows from Federer's dimension reducing argument \cite{Federer_book}, Lemma \ref{triplelem}, and by the well-known free-boundary regularity theorem (see for example
\cite{Simon1} and the references therein). 
More precisely, at the top $(n-1)$-dimensional stratum ${\rm sing}_1 V$ of singularities, the tangent cone has to 
be some orthogonal rotation of $|S\times\R^{n-1}|$, and Lemma \ref{triplelem} specifies that 
$S$ has to be a triple junction. Then by \cite{Simon1}, ${\rm sing}_1 V$ has the desired regularity.
Federer's dimension reducing argument shows that the next dimensional stratum ${\rm sing}_2 V$ has Hausdorff dimension $\leq n-2$. 
For the case of $N=2$,
by Lemma \ref{two-sided}, note that there cannot be triple junctions for all dimensions. Thus,
for $n=1$, ${\rm spt}\,\|V\|$ is a line, and for $n\geq 2$, ${\rm sing}_1 V$ is empty. 
Since the tangent cone of $V$ is stable, again
by the dimension reducing argument, the Hausdorff dimension of ${\rm sing}_2 V$ is $\leq n-7$.
For $2\leq n\leq 6$, since the blow-down limit of $V$ is a stable minimal cone with unit
density, it is a multiplicity 1 hyperplane. But then, the monotonicity formula shows that 
${\rm spt}\,\|V\|$ itself needs to be the hyperplane. This proves (3) and (4). 
\end{proof}
\section{Behavior at larger scales} \label{Large}
In this section, we specialize in the case of $n=1$. The idea of the proof is the following.
We know asymptotically how $\partial\mathcal E_{j_\ell}$ looks like within the length scale of $O(r_\ell)$
due to Theorem \ref{chlim}. Namely, they are very close to either a line segment or a triple junction
with three half-lines. 
We would like to patch this local picture together globally. 
To do so, we take advantage of the $L^2$ bound of smoothed
curvature \eqref{boundht} which is a good enough quantity to control the variation of tangent lines
of curves. Smoothing parameter $\varepsilon_{j_\ell}$ is much smaller than $r_\ell$ so that
\eqref{boundht} serves like a real $L^2$ curvature bound for $\partial\mathcal E_{j_\ell}$.
In the following, we first single out 
a ``good portion'' of $\partial\mathcal E_{j_\ell}$
denoted by $ Z_\ell$. We show that $Z_\ell$ looks more or less like a network of $C^{1,\sfrac12}$ curves.
\begin{defi} \label{defZell}
Let $\{r_\ell\}_{\ell=1}^{\infty}$ be a
sequence satisfying \eqref{rg} and \eqref{rg2}. For each $\ell\in \mathbb N$, define
\begin{equation}
\label{defA}
Z_\ell:=\{z\in  \partial\mathcal E_{j_\ell}\,:\, \inf_{r\in(0,r_\ell)}\mathcal H^1(\partial\mathcal E_{j_\ell}\cap B_{r}(z))/2r
\geq  1\}, \hspace{.5cm} Z_\ell^c:=\partial\mathcal E_{j_\ell}\setminus Z_\ell.
\end{equation}
\end{defi}
\begin{lem}
The set $Z_\ell$ is closed and for any $R>0$ we have
\begin{equation}
\label{defB}
r_\ell^2\mathcal H^1(Z_\ell^c\cap B_R)\leq -\frac{\Cl[c]{ca}R^2}{\min_{B_{R+1}}\Omega}\Delta_{j_\ell}\|\partial\mathcal E_{j_\ell}\|(\Omega),
\end{equation}
where $\Cr{ca}>0$ is an absolute constant. 
\end{lem}
\begin{rk} \label{gork} Once we have \eqref{defB},
combined with \eqref{dsmall} and $r_\ell=1/(j_\ell)^{2.5}$, we have
\begin{equation}
\label{dsmall2}
\mathcal H^1(Z_\ell^c\cap B_R)< \frac{\Cr{ca}R^2}{\min_{B_{R+1}}\Omega} \Cr{c_2} j_\ell^{-138} r_\ell^{-2}
=\frac{R^2\Cr{ca}\Cr{c_2}}{\min_{B_{R+1}}\Omega}\, j_\ell^{-133},
\end{equation}
which is negligibly small even if we rescale $Z_\ell$ by $1/r_\ell$. Also
the limits of $\partial\mathcal E_{j_\ell}$ and $Z_\ell$ as measures are equal. 
\end{rk}
\begin{proof} If $z\in Z_\ell^c$, there exists some $r\in (0,r_\ell)$ such that $\mathcal H^1(\partial
\mathcal E_{j_\ell}\cap B_r(z))<2r$. Then, there exists some 
$\epsilon>0$ such that $\mathcal H^1(\partial
\mathcal E_{j_\ell}\cap B_r(\tilde z))<2r$ for all $\tilde z\in U_{\epsilon}(z)$ and thus $Z_\ell\cap
U_{\epsilon}(z)=\emptyset$. This shows that $Z_\ell^c$ is relatively open in $\partial\mathcal E_{j_\ell}$ and since $\partial\mathcal E_{j_\ell}$ itself is closed, $Z_\ell$ is closed in $\R^2$.  
Let $B_R$ be covered by a union of balls of radius $r_\ell$ so that the number of balls is bounded by
$16R^2(r_\ell)^{-2}$, which can be done easily. Let $B_{r_\ell}(\tilde z)$ be any of such balls. If we prove
that 
\begin{equation}
\label{defB2}
\mathcal H^1(Z_\ell^c\cap B_{r_\ell}(\tilde z))
\leq -\frac{c}{\min_{B_{R+1}}\Omega} \Delta_{j_\ell}\|\partial\mathcal E_{j_\ell}\|(\Omega)
\end{equation}
for an absolute constant $c$, we will be done. For simplicity, rewrite
$Z_\ell^c\cap B_{r_\ell}(\tilde z)$ as $Z_\ell^c$. 
For each point $z\in Z^c_\ell$, by \eqref{defA}, there exists $r\in(0,r_\ell)$ such that
$\mathcal H^1(\partial\mathcal E_{j_\ell}\cap B_r(z))<2r$. Then, since 
$$\int^{r}_{0}\mathcal H^0(\partial\mathcal E_{j_\ell}\cap\partial B_s(z))\, ds\leq \mathcal H^1(\partial\mathcal E_{j_\ell}\cap B_r(z))<2r,$$ 
we must have some $s_z\in (0,r)$ such that $\mathcal H^0(\partial\mathcal E_{j_\ell}\cap
\partial B_{s_z}(z))=1$ or $=0$. Consider a covering of $Z_\ell^c$ by $\{B_{s_z}(z)\}_{z\in Z_\ell^c}$.
By the Besicovitch covering theorem, there exists a subfamily which consists of mutually disjoint
balls (and at most countable) $\{B_{s_{z_k}}(z_k) \}_{k}$ such that 
\begin{equation}
\label{defB3}
\mathcal H^1(Z_\ell^c)\leq {\bf B}_2 \sum_{k} \mathcal H^1(Z_\ell^c\cap B_{s_{z_k}}(z_k)),
\end{equation}
where ${\bf B}_2$ is the constant appearing in the Besicovitch covering theorem. 
Let $k_0\in\mathbb N$ be chosen so that 
\begin{equation}
\label{defB4}
\sum_k\mathcal H^1(Z_\ell^c\cap B_{s_{z_k}}(z_k))\leq 2\sum_{k=1}^{k_0}\mathcal H^1(Z_\ell^c\cap 
B_{s_{z_k}}(z_k)).
\end{equation}
Since $Z_\ell^c\subset\partial\mathcal E_{j_\ell}$, \eqref{defB3} and \eqref{defB4} show
\begin{equation}
\label{defB4sup}
\mathcal H^1(Z_\ell^c)\leq 2{\bf B}_2\mathcal \|\partial\mathcal E_{j_\ell}
\|(\cup_{k=1}^{k_0} B_{s_{z_k}}(z_k)).
\end{equation}
We next fix a Lipschitz map $f$ which is the identity map on $\mathbb R^{2}\setminus 
\cup_{k=1}^{k_0}U_{s_{z_k}}(z_k)$.  On each $U_{s_{z_k}}(z_k)$, we define $f$ as follows. 
For each $k=1,\ldots,k_0$, 
$\partial\mathcal E_{j_\ell}\cap \partial B_{s_{z_k}}(z_k)$ consists of at most one point
and $\partial B_{s_{z_k}}(z_k)\setminus \partial\mathcal E_{j_\ell}$ is connected.  Thus 
for each $k$, there is
one component of $\mathcal E_{j_\ell}=\{E_{j_\ell,i}\}_{i=1}^N$, say $E_{j_\ell,i(k)}$, such that $\partial B_{s_{z_k}}(z_k)\setminus \partial\mathcal E_{j_\ell}\subset E_{j_\ell,i(k)}$. Choose a ball $B_{s'_{k}}(z'_{k})$ such that  $B_{s_k'}(z_k')\subset  U_{s_{z_k}}(z_k)\cap E_{j_\ell,i(k)}$ and consider a Lipschitz retraction map $f$ which expands $B_{s'_{k}}(z'_k)$ bijectively to
$B_{s_{z_k}}(z_k)$ and maps $B_{s_{z_k}}(z_k)\setminus B_{s'_k}(z'_k)$ onto $\partial B_{s_{z_k}}(z_k)$. Since $f(B_{s_k'}(z_k'))
=B_{s_{z_k}}(z_k)$ and $B_{s_k'}(z_k')\subset E_{j_\ell,i(k)}$, $f$ has the 
property that 
\begin{equation}
\label{defB5}
B_{s_{z_k}}(z_k)\setminus (\partial\mathcal E_{j_\ell}\cap\partial B_{s_{z_k}}(z_k))
\subset {\rm int}\,f(E_{j_\ell,i(k)}).
\end{equation}
Now one can check that this $f$ is $\mathcal E_{j_\ell}$-admissible since it is a retraction map on 
each disjoint balls. By writing $\tilde{\mathcal E}_{j_\ell}:=f_{\star}\mathcal E_{j_\ell}$ and 
$\{\tilde E_{j_\ell,i}\}_{i=1}^N:=\tilde{\mathcal E}_{j_\ell}$,  
note that $B_{s_{z_k}}(z_k)\setminus (\partial\mathcal E_{j_\ell}\cap\partial B_{s_{z_k}}(z_k))
\subset \tilde E_{j_\ell,i(k)}$ for each $k=1,\ldots,k_0$. Since 
$\partial\mathcal E_{j_\ell}\cap\partial B_{s_{z_k}}(z_k)$ is a point or empty set, 
we have
\begin{equation}
\label{defB6}
\|\partial\tilde{\mathcal E}_{j_\ell}\|(\cup_{k=1}^{k_0} B_{s_{z_k}}(z_k))=0.
\end{equation}
It follows from \eqref{defB6} and Definition \ref{boldE} that $f\in {\bf E}(\mathcal E_{j_\ell},
j_\ell)$. It follows also from Definition \ref{tridef} that
\begin{equation}
\label{defB7}
\Delta_{j_\ell}\|\partial\mathcal E_{j_\ell}\|(\Omega)\leq -\|\partial{\mathcal E}_{j_\ell}\|
\mres_{\cup_{k=1}^{k_0}B_{s_{z_k}}(z_k)}(\Omega)\leq -(\min_{B_{R+1}} \Omega)\,\|\partial
{\mathcal E}_{j_\ell}\|(\cup_{k=1}^{k_0}B_{s_{z_k}}(z_k)).
\end{equation}
Combining \eqref{defB4sup} and \eqref{defB7} and setting $c=2{\bf B}_2$, we obtain
\eqref{defB2}, and subsequently \eqref{defB} with $\Cr{ca}=16c$. 
\end{proof}

\begin{lem}
\label{haudi}
Fix an arbitrary large $R>0$. 
Depending only on $\Cr{c_2}$, $\min_{B_{R+1}}\Omega$ and $\sup_{\ell} \|\partial\mathcal E_{j_\ell}\|(B_{R+2})$,
there exists a positive constant $\Cl[c]{c_4}\in(0,1)$ such that the following property holds for all sufficiently large $\ell$. 
For $z\in Z_\ell\cap B_R$ and $r\in (0,\Cr{c_4}]$, we have
\begin{equation}
\label{haudire}
\|\partial\mathcal E_{j_\ell}\|(B_r(z))\geq r/8.
\end{equation}
\end{lem}
\begin{proof} By Definition \ref{defZell}, \eqref{haudire} is satisfied for $r\in (0,r_\ell]$ already.
Thus we need to prove the case for $r\in (r_\ell,\Cr{c_4}]$ for a suitable $\Cr{c_4}$.    
By \cite[5.1(1)]{Allard}, for general varifold $V\in {\bf V}_1(\R^{2})$ with locally bounded first variation
$\|\delta V\|$ and $z\in \R^{2}$, we have
\begin{equation}
\label{haudi1}
s_2^{-1}\|V\|(B_{s_2}(z))\exp\big(\int_{s_1}^{s_2}\delta V(\kappa(V,z,r))\,dr\big)-s_1^{-1}\|V\|(B_{s_1}
(z))\geq 0.
\end{equation}
Here, ${\rm dist}(z,{\rm spt}\|V\|)<s_1<s_2<\infty$ and $\kappa(V,z,r)$ is a vector field defined by 
\[
\kappa(V,z ,r)(\tilde z):=\left\{\begin{array}{ll} (r\|V\|(B_r(z)))^{-1}(\tilde z-z) & \mbox{if }
\tilde z\in B_r(z),\\
0 & \mbox{if }\tilde z\notin B_r(z). \end{array}\right.
\]
We use \eqref{haudi1} with $V=\Phi_{\varepsilon_{j_\ell}}\ast\partial\E_{j_\ell}$, $z\in Z_\ell\cap B_R$
and $r_\ell\leq s_1<s_2\leq 1$.  
By definition, $|\kappa(V,z,r)(\tilde z)|\leq (\|V\|(B_r(z)))^{-1}$ for any $\tilde z\in\mathbb R^2$ 
(note that $\kappa(V,z,r)$ vanishes outside of $B_r(z)$) and we may estimate
\begin{equation}
\label{haudi2}
\begin{split}
|\delta (\Phi_{\varepsilon_{j_\ell}}\ast\partial\mathcal E_{j_\ell})(\kappa(\Phi_{\varepsilon_{j_\ell}}
\ast\partial\mathcal E_{j_\ell},z,r))|&=|(\Phi_{\varepsilon_{j_\ell}}\ast\delta(\partial\mathcal E_{j_\ell}))(\kappa(\Phi_{\varepsilon_{j_\ell}}
\ast\partial\mathcal E_{j_\ell},z,r))| \\
&\leq ((\Phi_{\varepsilon_{j_\ell}}\ast\|\partial\mathcal E_{j_\ell}\|)(B_r(z)))^{-1}
\int_{B_r(z)}|\Phi_{\varepsilon_{j_\ell}}\ast\delta(\partial\mathcal E_{j_\ell})|.
\end{split}
\end{equation}
We used \eqref{defi5}, \eqref{defi3} and \eqref{defi1} here. Next, by the Cauchy-Schwarz inequality and
\eqref{boundht}, 
\begin{equation}
\label{haudi3}
\begin{split}
\int_{B_r(z)}|\Phi_{\varepsilon_{j_\ell}}\ast\delta(\partial\mathcal E_{j_\ell})|
&\leq \Big(\int\frac{|\Phi_{\varepsilon_{j_\ell}}\ast\delta(\partial\mathcal E_{j_\ell})|^2
\Omega}{\Phi_{\varepsilon_{j_\ell}}\ast\|\partial\mathcal E_{j_\ell}\|+\varepsilon_{j_\ell}\Omega^{-1}}
\Big)^{\frac12}\Big(\int_{B_r(z)}\frac{\Phi_{\varepsilon_{j_\ell}}\ast\|\partial\mathcal E_{j_\ell}\|
+\varepsilon_{j_\ell}\Omega^{-1}}{\Omega}\Big)^{\frac12} \\
&\leq c_2^{\sfrac12}(\min_{B_{R+1}}\Omega)^{-\sfrac12}\{(\Phi_{\varepsilon_{j_\ell}}
\ast\|\partial\mathcal E_{j_\ell}\|)(B_r(z))+\varepsilon_{j_\ell}(\min_{B_{R+1}}\Omega)^{-1}2\pi 
r^2\}^{\sfrac12}. 
\end{split}
\end{equation}
The last term can be controlled as follows. We have
\begin{equation}
\label{haudi4}
(\Phi_{\varepsilon_{j_\ell}}\ast\|\partial\mathcal E_{j_\ell}\|)(B_r(z))=\|\partial\mathcal E_{j_\ell}\|
(\Phi_{\varepsilon_{j_\ell}}\ast\chi_{B_r(z)})\geq \frac12 \|\partial\mathcal E_{j_\ell}\|(B_{r/2}(z))
\geq \frac12 r_\ell
\end{equation}
for all large $\ell$ using $r\geq r_\ell$ and $\varepsilon_{j_\ell}\ll r_\ell$. The last inequality
follows from
$\|\partial\mathcal E_{j_\ell}\|(B_{r/2}(z))\geq r_\ell$ for $r\geq r_\ell$ due to \eqref{defA} and
$z\in Z_\ell$. Again by $\varepsilon_{j_\ell}\ll r_\ell$ and \eqref{haudi4}, for any $r\in [r_\ell,1]$ and sufficiently large $\ell$, 
\begin{equation}
\label{haudi5}
\varepsilon_{j_\ell}(\min_{B_{R+1}}\Omega)^{-1}2\pi r^2\leq \varepsilon_{j_\ell}(\min_{B_{R+1}}\Omega)^{-1}2\pi\leq \frac12 r_\ell\leq (\Phi_{\varepsilon_{j_\ell}}\ast\|\partial\mathcal E_{j_\ell}\|)(B_r(z)).
\end{equation}
By \eqref{haudi2}, \eqref{haudi3} and \eqref{haudi5}, for $r\in [r_\ell,1]$, we obtain
\begin{equation}
\label{haudi6}
|\delta (\Phi_{\varepsilon_{j_\ell}}\ast\partial\mathcal E_{j_\ell})(\kappa(\Phi_{\varepsilon_{j_\ell}}
\ast\partial\mathcal E_{j_\ell},z,r))|
\leq (2c_2)^{\sfrac12}(\min_{B_{R+1}}\Omega)^{-\sfrac12} \{(\Phi_{\varepsilon_{j_\ell}}\ast\|\partial\mathcal E_{j_\ell}\|)(B_r(z))\}^{-\sfrac12}.
\end{equation}
Substituting \eqref{haudi6} into \eqref{haudi1} and writing $\xi(r):=(\Phi_{\varepsilon_{j_\ell}}\ast\|\partial\mathcal E_{j_\ell}\|)(B_r(z))$ and $c^2:=2c_2(\min_{B_{R+1}}\Omega)^{-1}$, 
we obtain for $r_\ell\leq s_1<s_2\leq 1$
\begin{equation}
\label{haudi7}
s_2^{-1}\xi(s_2)\exp\big(c\int_{s_1}^{s_2}\xi(r)^{-\sfrac12}\,dr\big)\geq s_1^{-1}\xi(s_1). 
\end{equation}
Note that $\xi$ is a smooth positive function, and \eqref{haudi7} shows that $s^{-1}\xi(s)\exp(c
\int^s\xi^{-\sfrac12})$ is a monotone increasing function. After setting $\hat\xi(s):=s^{-1}\xi(s)$
and by differentiation, we obtain $(\hat\xi^{\sfrac12}+cs^{\sfrac12})'\geq 0$. Since 
$\hat\xi(r_\ell)\geq \frac12$ by \eqref{haudi4}, for any $r\in[r_\ell,1]$ we obtain
\begin{equation}
\label{haudi8}
r^{-1}(\Phi_{\varepsilon_{j_\ell}}\ast\|\partial\mathcal E_{j_\ell}\|)(B_r(z))
=\hat\xi(r)^{\sfrac12}\geq -cr^{\sfrac12}+\hat\xi(r_\ell)^{\sfrac12}+cr_\ell^{\sfrac12}
\geq -cr^{\sfrac12}+1/2^{\sfrac12}.
\end{equation}
Thus, we restrict $r$ so that $cr^{\sfrac12}<1/4$, for example, we obtain a positive lower bound
for the density ratio of $\Phi_{\varepsilon_{j_\ell}}\ast\|\partial\mathcal E_{j_\ell}\|$. Finally, we can estimate as
\begin{equation}
\label{haudi9}
\Phi_{\varepsilon_{j_\ell}}\ast\chi_{B_{r/2}(z)}\mres_{\R^2\setminus B_r(z)}
\leq e^{-\varepsilon_{j_\ell}^{-1}}\chi_{B_{R+2}}
\end{equation}
for all large $\ell$. This is because, for $\tilde z\notin B_{R+2}$, $(\Phi_{\varepsilon_{j_\ell}}\ast\chi_{B_{r/2}(z)})(\tilde z)=0$ since ${\rm spt}\,\Phi_{\varepsilon}\subset B_1$, $z\in B_R$ and $r\leq 1$. For 
$\tilde z\notin B_r(z)$, (cf. \eqref{tphi})
$$(\Phi_{\varepsilon_{j_\ell}}\ast\chi_{B_{r/2}(z)})(\tilde z)\leq 
c(\varepsilon_{j_\ell})(2\pi\varepsilon_{j_\ell})^{-2}\exp(-r^2/8\varepsilon_{j_\ell}^2)\leq 
c(\varepsilon_{j_\ell})(2\pi\varepsilon_{j_\ell})^{-2}\exp(-r_\ell^2/8\varepsilon_{j_\ell}^2)\leq \exp(-\varepsilon_{j_\ell}^{-1})$$  
for all large $\ell$. This proves \eqref{haudi9}. 
We use inequality $\chi_{B_r(z)}+(\Phi_{\varepsilon_{j_\ell}}\ast\chi_{B_{r/2}(z)}
)\mres_{\R^2\setminus B_r(z)}\geq \Phi_{\varepsilon_{j_\ell}}\ast\chi_{B_{r/2}(z)}
$ (note that $\Phi_{\varepsilon_{j_\ell}}\ast\chi_{B_{r/2}(z)}\leq 1$ since $\int_{\mathbb R^{n+1}}
\Phi_{\varepsilon}=1$) and \eqref{haudi9} to derive
\begin{equation}
\label{haudi10}
\|\partial\mathcal E_{j_\ell}\|(B_r(z)) 
\geq (\Phi_{\varepsilon_{j_\ell}}\ast\|\partial\mathcal E_{j_\ell}\|)(B_{r/2}(z))-e^{-\varepsilon^{-1}_{
j_\ell}}\|\partial\mathcal E_{j_\ell}\|(B_{R+2}).
\end{equation}
The last exponentially small term converges to $0$ even after dividing by $r\geq r_\ell$. Thus, using \eqref{haudi8} with $r/2$
in place of $r$ and \eqref{haudi10}, we obtain the desired estimate. 
\end{proof}

\begin{lem}
\label{defC4}
Let $R,s>0$ be fixed and let $\epsilon>0$ be arbitrary. Then there exists $\ell_0\in\mathbb N$ 
such that the following holds for all $\ell\geq \ell_0$. If $N\geq 3$, for $z\in Z_\ell\cap B_R$, 
at least one of the following (a) or (b) holds (they are not mutually exclusive). If $N=2$, 
(a) holds for $z\in Z_\ell\cap B_R$. 
\newline
(a) There exists a line denoted by $S$ such that $z\in S$ and 
\begin{equation}
\label{defC1}
\begin{split}
{\rm dist}_H(Z_\ell\cap B_{r_\ell }(z),S\cap B_{r_\ell }(z))\leq \epsilon r_\ell.
\end{split}
\end{equation}
(b) There exists a triple junction with three half-lines denoted by $S$ such that $z\in S$ and with the junction in $U_{r_\ell}(z)$
and such that
\begin{equation}
\label{defC1tri}
{\rm dist}_H(Z_\ell\cap B_{2r_\ell }(z),S\cap B_{2r_\ell }(z))\leq \epsilon r_\ell.
\end{equation}
(c) Furthermore, in the case of (a), with the same $S$ and 
for any $\phi\in C_c^1(U_{r_\ell }(z))$ with $\sup\,(|\phi|+r_\ell|\nabla\phi|)\leq s$,
we have
\begin{equation}
\label{sdefC1}
\Big|\int_{B_{r_\ell}(z)}\phi(\tilde z)\,d\|\partial\mathcal E_{j_\ell}\|(\tilde z)
-\int_{B_{r_\ell}(z)\cap S}\phi(\tilde z)\,d\mathcal H^1(\tilde z)\Big|\leq \epsilon r_\ell
\end{equation}
and
\begin{equation}
\label{defC2}
 \Big|\int_{{\bf G}_1(B_{r_\ell }(z))} \tilde S\phi(\tilde z)\,d(\partial\mathcal E_{j_\ell})
 (\tilde z,\tilde S)-\int_{{\bf G}_1(B_{r_\ell }(z))} \tilde S\phi(\tilde z)
 \,d(|S|)(\tilde z,\tilde S)\Big|\leq \epsilon r_\ell.
\end{equation}
\end{lem}

\begin{rk}
Note that the integrand of \eqref{defC2} is $2\times 2$ matrix and the absolute value is 
the Euclidean norm of the matrix. One can check that the estimate \eqref{defC2} is independent of the coordinate system under orthogonal rotation. 
\end{rk}

\begin{proof} Consider the case $N\geq 3$. 
If the first claim were not true, we have some $\epsilon_0>0$ and subsequences (denoted by the same
index) $j_\ell$ and $z^{(\ell)}\in B_{R}\cap Z_\ell$ such that \eqref{defC1} and \eqref{defC1tri} are not true.
Since $\partial\mathcal E_{j_\ell}=Z_\ell\cup Z_\ell^c$ and $(r_\ell)^{-1}\mathcal H^1
(Z_\ell^c\cap B_{R+1})\rightarrow 0$ by \eqref{dsmall2}, the limit of 
$V_\ell$ (see the definition \eqref{vdef}) and that of $|(F_\ell)_{\sharp} Z_\ell|$ coincide. By 
Theorem \ref{chlim}(1), we know that the limit of $V_\ell$ is either a line or a triple junction, denoted by $\hat S$. 
For any $0<t<1/2$, $\mathcal H^1(B_{tr_\ell}(z^{(\ell)}) \cap \partial\mathcal E_{j_\ell})\geq 2tr_\ell$
by \eqref{defA}, hence we have $\mathcal H^1(B_t\cap \hat S)\geq 2t$ and $\hat S$ has to include the origin in particular. 
By the contradiction argument, we have ${\rm dist}_H(Z_\ell\cap B_{r_\ell}(z^{(\ell)}),S\cap B_{r_\ell}(z^{(\ell)}))>\epsilon_0 r_\ell$
for any line $S$ with $z^{(\ell)}\in S$ and the same inequality with $B_{2 r_\ell}(z^{(\ell)})$
in place of $B_{r_\ell}(z^{(\ell)})$
for any triple junction $S$ with the junction in $U_{r_\ell}(z^{(\ell)})$ and $z^{(\ell)}\in S$. After stretching by $F_\ell$, we have 
${\rm dist}_H((F_\ell)_{\sharp}Z_\ell\cap B_1,S\cap B_1)>\epsilon_0$, again for any line $S$ with $0\in S$
and the same inequality with $B_2$ in place of $B_1$ for any triple junction $S$ with $0\in S$ and the junction in $U_1$. With $S=\hat S$, on the other hand, we would have a contradiction since $(F_\ell)_{\sharp}Z_\ell$
converges to $\hat S$ in the Hausdorff distance, which follows from \eqref{defA}. Note that, if 
$\hat S$ is a triple junction with the junction in $U_1$, we would have a contradiction to 
${\rm dist}_H((F_\ell)_{\sharp}Z_\ell\cap B_2,\tilde S\cap B_2)>\epsilon_0$ for all large $\ell$. 
If $\hat S$ is either a line or a triple junction with the junction outside of $U_1$, we have
a contradiction to ${\rm dist}_H((F_\ell)_{\sharp}Z_\ell\cap B_1,\tilde S\cap B_1)>\epsilon_0$. 
This proves the first claim for $N\geq 3$. 
If $N=2$, the above argument using Theorem \ref{chlim}(3) proves the claim.

If (c) were not true, we have subsequences (again denoted by the same index) $z^{(\ell)}\in
Z_\ell\cap B_R$, lines $S^{(\ell)}$ with $z^{(\ell)}\in S^{(\ell)}$
and $\phi^{(\ell)}\in C_c^1(U_{r_\ell}(z^{(\ell)}))$ with $\sup\,(|\phi^{(\ell)}|+r_\ell|\nabla\phi^{(\ell)}|)\leq s$ such that 
$r_\ell^{-1}{\rm dist}_H(Z_\ell\cap B_{r_\ell}(z^{(\ell)}),S^{(\ell)}\cap B_{r_\ell}(z^{(\ell)}))\rightarrow
0$ and either \eqref{sdefC1} or \eqref{defC2} fails for these with $\epsilon_0$. 
As before, the limit of $V_\ell$ and $|(F_\ell)_{\sharp}Z_\ell|$ coincide, and since ${\rm dist}_H((F_\ell)_{
\sharp}Z_\ell\cap B_1,(F_{\ell})_{\sharp}S^{(\ell)}\cap B_1)\rightarrow 0$, $V_\ell$ subsequencially converges to a line $\hat S=\lim_{\ell\rightarrow\infty} (F_\ell)_{\sharp} S^{(\ell)}$
on $U_1$ as {\it varifolds}. 
If \eqref{defC2} is violated, in terms of $V_\ell$, we have
\begin{equation}
\label{defC3}
\big|\int_{{\bf G}_1(B_{1})} \tilde S\hat\phi^{(\ell)}(\tilde z)\,dV_\ell(\tilde z,\tilde S)
-\int_{{\bf G}_1(B_{1})} \tilde S\hat\phi^{(\ell)}(\tilde z)\, d(|(F_\ell)_{\sharp}S^{(\ell)}|)(\tilde z,\tilde S)\big|>\epsilon_0.
\end{equation}
Here, $\hat\phi^{(\ell)}(\tilde z):=\phi^{(\ell)}(r_\ell \tilde z+z^{(\ell)})$ and 
$\sup\,(|\hat \phi^{(\ell)}|+|\nabla\hat\phi^{(\ell)}|)\leq s$. Because of the latter uniform bound for 
$\hat\phi^{(\ell)}$, there exists a subsequence which converges uniformly to a Lipschitz function,
say, $\phi$, with support in $B_1$. Thus, with the varifold convergence, we have (for a not-relabeled subsequence)
\[
\lim_{\ell\rightarrow\infty} \int_{{\bf G}_1(B_{1})} \tilde S\hat\phi^{(\ell)}\,dV_\ell
=\int_{{\bf G}_1(B_{1})} \tilde S\phi\,d(|\hat S|)
=\lim_{\ell\rightarrow\infty}\int_{{\bf G}_1(B_{1})} \tilde S\hat\phi^{(\ell)}\, d(|(F_\ell)_{\sharp}S^{(\ell)}|).
\]
This cannot be compatible with \eqref{defC3} for all large $\ell$. Thus we obtain a desired contradiction. 
The case that \eqref{sdefC1} does not hold can be similarly handled. 
\end{proof}
\begin{thm}\label{regflat}
Given $\nu\in\mathbb N$, there exist $0<\Cl[c]{c_6},\Cl[c]{c_7}<1$ 
with the following property. Suppose $\{\partial\mathcal E_{j_\ell}\}$ is a sequence satisfying 
\eqref{boundht} and $\mu=\lim_{\ell\rightarrow\infty}\|\partial\mathcal E_{j_\ell}\|$ on $\R^2$. Assume that for $B_{2r}(a)\subset \R^2$,
we have 
\begin{equation}\label{pdefC1}
\limsup_{\ell\rightarrow\infty} \int_{B_{2r}(a)} \frac{ r |\Phi_{\e_{j_\ell}}\ast
\delta(\partial \mathcal E_{j_\ell})|^2}{\Phi_{\e_{j_\ell}}\ast\|\partial\mathcal E_{j_\ell}\|+\e_{j_\ell}\Omega^{-1}}\leq \Cr{c_6},
\end{equation}
\begin{equation}\label{pdefC2}
{\rm spt}\,\mu\cap B_{2r}(a)\subset\{a+(x,y)\in\mathbb R^2\,:\, |y|\leq \Cr{c_7} r\},
\end{equation}
\begin{equation}\label{pdefC3}
r\big(\nu-\frac12\big)\leq \mu \big( 
\{a+(x,y)\in \R^2\,:\, |y|\leq \Cr{c_7} r,\,|x|\leq r/2\}
\big),
\end{equation}
\begin{equation}\label{pdefC2sup}
\mu \big( 
\{a+(x,y)\in \R^2\,:\, |y|\leq \Cr{c_7} r,\,|x|\leq r\}
\big)\leq 2r\big(\nu+\frac12\big).
\end{equation}
Then there exist functions $f_i\,:\,[-r,r]\rightarrow [-\Cr{c_7} r,\Cr{c_7} r]$ ($i=1,\ldots,\nu$)
with 
\begin{equation} \label{pdefC4}
f_i\in
W^{2,2}([-r,r]), \hspace{.5cm}f_1(x)\leq f_2(x)\leq\ldots\leq f_\nu(x) \mbox{ for }x\in[-r,r]
\end{equation} 
and such that, writing ${\rm graph}\,f_i:=\{a+(x,f_i(x))\in\mathbb R^2\,:\,x\in[-r,r]\}$, we have
\begin{equation} \label{pdefC5}
\mu\mres_{B_r(a)}=\sum_{i=1}^\nu\mathcal H^1\mres_{B_r(a)\cap {\rm graph}\,f_i}.
\end{equation}
\end{thm}
\begin{proof} Choose and fix a large $R>1$ so that $B_{2r}(a)\subset B_{R/2}$, and let $\Cr{c_4}$ be the corresponding
constant obtained in Lemma \ref{haudi} with $\Cr{c_2}$ in \eqref{boundht}, $\min_{B_{R+1}}\Omega$ and $\sup_\ell\|\partial\mathcal E_{j_\ell}
\|(B_{R+2})$ with the fixed $R$. 
We will fix $\Cr{c_6}$ and 
$\Cr{c_7}$ later as absolute constants. Let $\{\epsilon_\ell\}_{\ell\in\mathbb N}$ be a sequence of
positive numbers converging to $0$, and let $\{s_\ell\}_{\ell\in \mathbb N}$ be a sequence of positive
numbers diverging to $\infty$. Using Lemma \ref{defC4}, we choose a subsequence
(denoted by the same index) such that $Z_\ell$ satisfies the properties listed in Lemma \ref{defC4} 
with $\epsilon=\epsilon_\ell$ and $s=s_{\ell}$ and
so that 
\begin{equation}\label{epdiv}
r_\ell\ll \epsilon_\ell\,\,\,\mbox{ as }\,\,\,\ell\rightarrow \infty.
\end{equation}
Define
\begin{equation}
Z_\ell^*:=\{z\in Z_\ell\cap B_{2r}(a)\,:\, \mbox{there is a line $S$ satisfying $z\in S$ and \eqref{defC1}}\}.
\end{equation}
By Lemma \ref{defC4}, for any point $z\in B_{2r}(a)\cap Z_\ell\setminus Z_\ell^*$, there exists
a triple junction $S$ with the junction in $U_{r_\ell}(z)$ satisfying $z\in S$ and 
\eqref{defC1tri}. 
First, we show that the 
set $Z_\ell^*$ is composed of a finite set of ``almost $C^{1,\frac12}$ curves'' for sufficiently
small $\Cr{c_6}$. 

{\bf Step 1 : selection of line segments}.
Pick an arbitrary point $z_0\in Z_\ell^*\cap B_{5r/4}(a)$. Then
there exists a line $S$ with the stated properties in \eqref{defC1}. 
Call the line segment $S\cap B_{r_\ell}(z_0)$ as $L_0$. For notational convenience, 
we consider a coordinate system so that $L_0$ is parallel to the $x$-axis and suppose
that $z_0=(x_0,y_0)$. Write the coordinates of the endpoints of $L_0$ as 
$z_0^-:=(x_0-r_\ell,y_0)$ and $z_0^+:=(x_0+r_\ell,y_0)$. 
Then, we inductively choose $z_k\in Z_\ell$ and a line segment $L_k$ for $k\geq 1$ 
as long as $L_k\subset B_{3r/2}(a)$ or there is an ``encounter with a triple junction'' as
detailed in the following (see Figure \ref{linefigure}). 
\begin{figure}[!h]
\centering
\includegraphics[width=16cm]{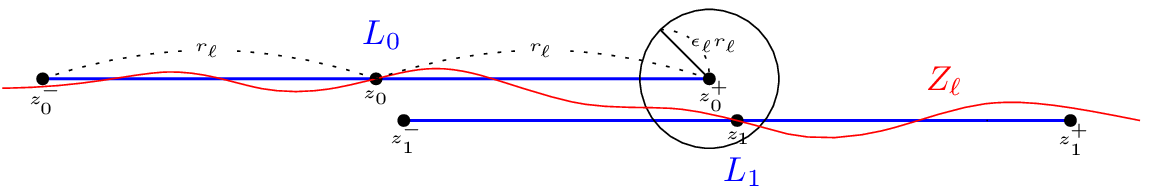}              
\caption{}
\label{linefigure}
\end{figure} 
Write the right-hand side endpoint of $L_k$ as $z_k^+$. By \eqref{defC1},
we have $B_{\epsilon_\ell r_\ell}(z_k^+)\cap Z_\ell\neq\emptyset$. There are two possibilities:
(i) $Z_\ell^*\cap B_{\epsilon_\ell r_\ell}(z_k^+)\neq \emptyset$ and (ii) $Z_\ell^*\cap B_{\epsilon_\ell r_\ell}(z_k^+)=\emptyset$. 
\newline
(i) We pick $z_{k+1}\in Z_\ell^*\cap B_{\epsilon_\ell r_\ell}(z_k^+)$. Then there exists a line $S$ with $z_{k+1}\in S$ and \eqref{defC1} in $B_{r_\ell}(z_{k+1})$. We let $L_{k+1}:=S\cap B_{r_\ell}(z_{k+1})$. 
If $L_{k+1}\setminus
B_{3r/2}(a)\neq \emptyset$ for the first time, we stop the induction and name this $k$ as $k_0$. Otherwise
we define $z_{k+1}^+$ as before and continue with the inductive process.
By \eqref{defC1} in particular, we have
\begin{equation}
\label{defC1sup1}
Z_\ell\cap B_{r_\ell}(z_{k+1})\subset (L_{k+1})_{\epsilon_\ell r_\ell}.
\end{equation}
Here, for shorthand, $(A)_t$ is the $t$-neighborhood of the set $A$, as in $(L_{k+1})_{\epsilon_\ell r_\ell}$
of \eqref{defC1sup1}. 
\newline
(ii) Since $B_{\epsilon_\ell r_\ell}(z_k^+)\cap Z_\ell\neq\emptyset$, there exists $z_{k+1}
\in B_{\epsilon_\ell r_\ell}(z_k^+)\cap Z_\ell\setminus Z_\ell^*$. By Lemma \ref{defC4}, 
there exists a triple 
junction $S$ with $z_{k+1}\in S$ and with the junction in $U_{r_\ell}(z_{k+1})$ such that
\eqref{defC1tri} holds with $z=z_{k+1}$. In the following, the occurrence of (ii) is
casually called {\it encounter with a triple junction}, since this is a place where $Z_\ell$ 
looks like a triple junction. 
When (ii) occurs for the first time, we end the induction, and let this $k$ be $k_0$ (so that $z_{k_0+1}$ 
is in $r_\ell$-neighborhood of a junction). 

There is a possibility that the case (i) occurs indefinitely without exit from $B_{3r/2}(a)$, but this
possibility will be eliminated in Step 2. For now, in this case, we let $k_0$ be arbitrary. 

{\bf Step 2: estimate on the slope of line segments.}
In the following, we prove
that the angle between the $x$-axis and $L_k$ ($1\leq k\leq k_0$) remains small if
$\Cr{c_6}$ is appropriately small, which effectively tells that $\cup_{k=0}^{k_0} L_k$ can
be approximated well by a graph over the $x$-axis with small slope.
For that purpose, we prove for $k=1,\ldots,k_0$ that
\begin{equation}
|(y_{k}^+-y_{k})/(x_{k}^+-x_{k})|\leq \Cl[c]{c_8}\Big(\epsilon_\ell+\Big(\int_{\cup_{i=0}^{k}
U_{r_\ell}(z_{i})} \frac{|\Phi_{\e_{j_\ell}}\ast\delta(\partial\mathcal E_{j_\ell})|^2}{
\Phi_{\e_{j_\ell}}\ast\|\partial\mathcal E_{j_\ell}\|+\e_{j_\ell}\Omega^{-1}}\Big)^{\frac12}
(k\,r_\ell+\epsilon_\ell)^{\frac12}\Big).
\label{defC6}
\end{equation}
Here, $z_{k}^+=(x_{k}^+,y_{k}^+)$ and $z_{k}=(x_{k},y_{k})$, thus the left-hand
side is the slope of $L_{k}$. The constant $\Cr{c_8}$ is an absolute constant. 
To prove \eqref{defC6}, the first remark is that for any $0\leq k\leq k_0$, 
the angle between the neighboring
line segments $L_k$ and $L_{k+1}$ is bounded by a fixed constant multiple of $\epsilon_\ell$. 
This follows from $|z_{k+1}-z_k^+|\leq \epsilon_\ell r_\ell$, ${\rm dist}(z_{k+1}^-,L_{k})\leq
2\epsilon_\ell r_\ell$ (due to \eqref{defC1} and a triangle inequality) and a simple geometric argument. 
Thus, starting from the slope of $L_0$ being $0$, we may assume that the slope of $L_k$ remains less than,
say, $1/10$, until some $\tilde k\leq k_0$. In the following, we prove \eqref{defC6} for all $k\leq \tilde k$. 
Once this is done, by assuming $\Cr{c_6}$ small, note that \eqref{defC6}
ensures that the slope of $L_{\tilde k}$ remains small (say, less than $1/20$). Here, note that
$\tilde k r_\ell$ in \eqref{defC6} is less than a constant multiple of $r$ since $\cup_{i=0}^{\tilde k} L_i$ is more or less
parallel to the $x$-axis and it has to remain within $B_{3r/2}(a)$. If we restrict 
$\Cr{c_6}$ depending only on $\Cr{c_8}$, we can make sure that the slope is less than 
$1/10$ for all $k\leq k_0$ in the end and we have \eqref{defC6} for all $k\leq k_0$. Thus assume that the slopes of $L_1,\ldots,L_k$ are smaller than $1/10$. Utilizing the fact that the difference of slopes on the neighboring segments is at most constant multiple of
$\epsilon_\ell$, we may construct a 
function $\psi_1$ satisfying the following properties.
\begin{itemize}
\item $\psi_1=1$ on the $r_\ell/8$-neighborhood of $\cup_{i=0}^{k} L_i$.
\item $\psi_1=0$ on the complement of $r_\ell/6$-neighborhood of $\cup_{i=0}^{k} L_i$.
\item $0\leq \psi_1\leq 1$ and $\sup( r_\ell|\nabla\psi_1|+r_\ell^2\|\nabla^2\psi_1\|)\leq c$, where
$c$ is an absolute constant. 
\end{itemize}
Next, we define a smooth function $\psi_2$ which depends only on the $x$-variable and 
such that
\begin{itemize}
\item
$\psi_2(x)=0$ for $x<x_0-r_\ell/4$ or $x>x_k+r_\ell/4$.
\item
$\psi_2(x)=1$ for $x_0-r_\ell/8<x<x_k+r_\ell/8$.
\item
$\sup(r_\ell|\psi_2'|+r_\ell^2|\psi_2^{''}|)\leq c$, where $c$ is an absolute constant.
\end{itemize}
Finally, we set $\psi(x,y)=\psi_1(x,y)\psi_2(x)$. 
Because of the properties of $\psi_1$ and $\psi_2$, one can check that 
\begin{equation}\label{defC18}
\psi=1\,\,\mbox{ on }\,\,(\cup_{i=1}^{k-1}L_i)_{r_\ell/10}\,\,\mbox{ and }\,\,
{\rm spt}\,\psi\subset  
(\cup_{i=1}^{k-1} L_i)_{r_\ell/2}
\end{equation} 
for small $\epsilon_\ell$. 
Another property we use is 
\begin{equation}
\label{defC7}
\frac{\partial\psi}{\partial y}=\psi_2\frac{\partial\psi_1}{\partial x}=0\, \mbox{ on } \, Z_\ell.
\end{equation}
To check this, since $\psi_2$ is independent of $y$, 
we only need to check that the set of points with $\nabla\psi_1\neq 0$ and 
$\psi_2>0$ does not intersect $Z_\ell$. The set satisfying these conditions is included in
$\{(x,y)\,:\,x_0-r_\ell/4\leq x\leq x_k+r_\ell/4\} \cap (\cup_{i=0}^k L_i)_{r_\ell/6}\setminus(\cup_{i=0}^k
L_i)_{r_\ell/8}$. But having a point $\tilde z$ of $Z_\ell$ in this set implies that $\tilde z\in 
\cup_{i=0}^k U_{r_\ell}(z_i)\setminus (\cup_{i=0}^k
L_i)_{r_\ell/8}$, which is a contradiction to \eqref{defC1sup1} for small $\epsilon_\ell$. This proves \eqref{defC7}.
We next estimate
\begin{equation}
\label{defC8}
\begin{split}
\delta (\partial \mathcal E_{j_\ell})((0,\psi))
&=\int_{{\bf G}_1(\mathbb R^2)}
\tilde S\cdot \nabla(0,\psi(\tilde z))\,d(\partial \mathcal E_{j_\ell})(\tilde z,\tilde S)\\
&=\int_{{\bf G}_1(\mathbb R^2)} \tilde S_{21}\frac{\partial \psi}{\partial x}
+\tilde S_{22}\frac{\partial \psi}{\partial y}\,d(\partial\mathcal E_{j_\ell})(\tilde z,\tilde S).
\end{split}
\end{equation}
Here $\tilde S_{21}$ is the $(2,1)$ component of $\tilde S\in{\bf G}(2,1)$ and similarly
for $\tilde S_{22}$. 
For the second term of \eqref{defC8}, by \eqref{defC7}, we have
\begin{equation}
\label{defC9}
\Big|\int_{{\bf G}_1(\mathbb R^2)} \tilde S_{22}\frac{\partial\psi}{\partial y}\,d(\partial
\mathcal E_{j_\ell})\Big|\leq \int_{{\rm spt}\,\psi\cap Z_\ell^c}
\Big|\frac{\partial\psi}{\partial y}\Big|d\|\partial\mathcal E_{j_\ell}\|
\leq c r_\ell^{-1} \mathcal H^1({\rm spt}\,\psi\cap Z_\ell^c)\leq \epsilon_\ell.
\end{equation}
The last inequality follows since $\mathcal H^1(Z_\ell^c\cap B_{R})\ll r_\ell^3$ by \eqref{dsmall2}, and by 
\eqref{epdiv}. For the first term of \eqref{defC8},
for $\frac{\partial\psi}{\partial x}=\psi_2'\psi_1+\psi_2\frac{\partial\psi_1}{
\partial x}$, the integral of the second term can be handled similarly as above using 
\eqref{defC7}. 
The term $\psi_2'\psi_1$ is nonzero only on $U_{r_\ell/2}(z_0)$ and $U_{r_\ell/2}(z_k)$. We use
\eqref{defC2} with $\phi=r_\ell\psi_2'\psi_1$ on these balls. For all large $\ell$, we have
$s_\ell\geq\sup(|\phi|+r_\ell|\nabla\phi|)$.   
Thus we obtain
\begin{equation}\label{defC10}
\Big|\int_{{\bf G}_1(U_{r_\ell/2}(z_i))} \tilde S_{21}\psi_2'\psi_1\,d(\partial\mathcal E_{j_\ell})
-\int_{{\bf G}_1(U_{r_\ell/2}(z_i))}\tilde S_{21}\psi_2'\psi_1\,d(|L_i|)\Big|\leq \epsilon_\ell
\end{equation}
for $i=0$ and $i=k$. When $i=0$, since $L_0$ is a line segment parallel to the $x$-axis,
we have $\tilde S_{21}=0$ and 
\begin{equation}\label{defC11}
\int_{{\bf G}_1(U_{r_\ell/2}(z_0))}\tilde S_{21}\psi_2'\psi_1\,d(|L_0|)
=0. 
\end{equation}
On the other hand, when $i=k$ and writing the slope of $L_k$ as 
$\alpha:=(y_k^+-y_k)/(x_k^+-x_k)$, $\tilde S_{21}$ on $L_k$ is $\alpha/(1+\alpha^2)$ and 
\begin{equation}\label{defC12}
\int_{{\bf G}_1(U_{r_\ell/2}(z_k))}\tilde S_{21}\psi_2'\psi_1\,d(|L_k|)
=-\frac{\alpha}{\sqrt{1+\alpha^2}}.
\end{equation}
Combining \eqref{defC8}-\eqref{defC12}, we obtain the estimate
\begin{equation}\label{defC13}
\Big|\delta(\partial\mathcal E_{j_\ell})((0,\psi))+\frac{\alpha}{\sqrt{1+\alpha^2}}\Big|\leq 
4\epsilon_\ell.
\end{equation}
We next estimate
\begin{equation}\label{defC14}
\big|\delta(\partial\mathcal E_{j_\ell})((0,\psi))-(\Phi_{\varepsilon_{j_\ell}}\ast
\delta(\partial\mathcal E_{j_\ell}))((0,\psi))\big|\leq
\int_{\mathbb R^2} |\nabla\psi-\Phi_{\varepsilon_{j_\ell}}\ast\nabla\psi|\,d\|\partial
\mathcal E_{j_\ell}\|.\end{equation}
Since $\psi=0$ outside of $B_{R}$ in particular, and since $\Phi_{\varepsilon_{j_\ell}}$
has support on $B_1$, the integrand of the above vanishes outside of $B_{R+1}$. Furthermore,
one can estimate
\begin{equation}\label{defC15}
\begin{split}
&|\nabla\psi(z)-(\Phi_{\varepsilon_{j_\ell}}\ast\nabla\psi)(z)|
\leq \int_{B_{1}(z)} |\nabla\psi(z)-\nabla\psi(\hat z)|\Phi_{\varepsilon_{j_\ell}}(z-\hat z)
\, d\hat z\\
&\leq c r_\ell^{-2} \int_{B_1(z)}|z-\hat z|\Phi_{\varepsilon_{j_\ell}}(z-\hat z)\, d\hat z
\leq cr_\ell^{-2}(\varepsilon_{j_\ell}^{9/10}+\varepsilon_{j_\ell}^{-2}
\exp(-1/(2\varepsilon^{1/5}_{j_\ell}))).
\end{split}
\end{equation}
The last inequality may be obtained by splitting the domain of integration to $\{|z-\hat z|
\leq \varepsilon_{j_\ell}^{9/10}\}$ and the complement, where $\Phi_{\varepsilon_{j_\ell}}$ is
exponentially small (see the analogous estimate \cite[(5.6)]{KimTone}).  
Thus, by \eqref{defC14} and \eqref{defC15}, we have
\begin{equation}\label{defC16}
\big|\delta(\partial\mathcal E_{j_\ell})((0,\psi))-(\Phi_{\varepsilon_{j_\ell}}\ast
\delta(\partial\mathcal E_{j_\ell}))((0,\psi))\big|\leq
cr_\ell^{-2}\varepsilon_{j_\ell}^{9/10}(\min_{B_{R+1}}\Omega)^{-1}\|\partial\mathcal E_{j_\ell}\|
(\Omega).
\end{equation}
Since $\varepsilon_{j_\ell}<j_\ell^{-6}$ and  $r_\ell=j_\ell^{-2.5}$, we have $r_\ell^{-2}
\varepsilon_{j_\ell}^{9/10}\ll 1$ and 
\eqref{defC16} is estimated by $\epsilon_\ell$
for all large $\ell$. 
Next, since $0\leq \psi\leq 1$, we have (writing $U={\rm spt}\,\psi$)
\begin{equation}\label{defC17}
|(\Phi_{\varepsilon_{j_\ell}}\ast
\delta(\partial\mathcal E_{j_\ell}))((0,\psi))|\leq \Big(\int_{U} 
\frac{|\Phi_{\varepsilon_{j_\ell}}\ast
\delta(\partial\mathcal E_{j_\ell})|^2}{\Phi_{\varepsilon_{j_\ell}}\ast\|\partial\mathcal E_{j_\ell}\|
+\varepsilon_{j_\ell}\Omega^{-1}}\Big)^{\frac12}\Big(\int_{U} \Phi_{\varepsilon_{j_\ell}}\ast\|\partial\mathcal E_{j_\ell}\|
+\varepsilon_{j_\ell}\Omega^{-1}\Big)^{\frac12}.
\end{equation}
Note that the second quantity on the right-hand side should correspond to the ``length of
curve''. On each ball $U_{r_\ell}(z_i)$, the measure $\|\partial \E_{j_\ell}\|$ is well-approximated
by $\mathcal H^n\mres_{L_i}$, and we take advantage of this in the following. 
To estimate $\int_U\Phi_{\varepsilon_{j_\ell}}\ast\|\partial\mathcal E_{j_\ell}\|$, 
consider a partition of of unity $\{\zeta_i\}_{i=0}^k$ 
subordinate to $
\{U_{3r_\ell/4}(z_i)\}_{i=0}^k$ such that $\zeta_i\in C_c^\infty (U_{3r_\ell/4}(z_i))$, $0\leq \zeta_i\leq 1$,
$\sup(\zeta_i+r_\ell|\nabla\zeta_i|)\leq s_\ell$ (which is true for all large $\ell$)
and $\sum_{i=0}^k\zeta_i=1$ on $(\cup_{i=1}^{k-1} L_i)_{r_\ell/2}$. Note that the latter
includes ${\rm spt}\,\psi$ by \eqref{defC18}. Thus we have
\begin{equation}
\label{sdefC2}
\int_U\Phi_{\varepsilon_{j_\ell}}\ast\|\partial\mathcal E_{j_\ell}\|\,dz\leq \sum_{i=0}^k
\int_{U_{3r_\ell/4}(z_i)} \Phi_{\varepsilon_{j_\ell}}\ast\|\partial
\mathcal E_{j_\ell}\|\zeta_i\,dz.
\end{equation}
We may estimate
\begin{equation}
\label{sdefC3}
\big|\int \Phi_{\varepsilon_{j_\ell}}\ast\|\partial\mathcal E_{j_\ell}\|
\zeta_i-\int\zeta_i\,d\|\partial\mathcal E_{j_\ell}\|\big|\leq \int_{\R^2} |\zeta_i
-\Phi_{\varepsilon_{j_\ell}}\ast\zeta_i|\,d\|\partial\mathcal E_{j_\ell}\|
\end{equation}
which can be estimated just like \eqref{defC14}-\eqref{defC16} (with $r_\ell^{-1}$ in place of $r_\ell^{-2}$) and we may assume that this is less than $\epsilon_\ell r_\ell$.
By \eqref{sdefC1}, 
\begin{equation}
\label{sdefC5}
\int \zeta_i\,d\|\partial\mathcal E_{j_\ell}\|\leq \int \zeta_i\,d(|L_i|)+\epsilon_\ell r_\ell
\leq r_\ell(2+\epsilon_\ell).
\end{equation}
By combining \eqref{sdefC2}-\eqref{sdefC5}, we obtain
\begin{equation} \label{sd1}
\int_U \Phi_{\varepsilon_{j_\ell}}\ast\|\partial\mathcal E_{j_\ell}\|\,dz
\leq 2(k+1) r_\ell(1+\epsilon_\ell).
\end{equation}
By \eqref{defC18}, we have $\mathcal L^2(U)\leq kr_\ell^2$ and  
$\int_U\varepsilon_{j_\ell} \Omega^{-1}\leq (\min_{B_{R}}\Omega)^{-1}k r_\ell^2\varepsilon_{j_\ell}$
and with \eqref{sd1}, we obtain (for all large $\ell$ so that $(\min_{B_R}\Omega)^{-1}r_\ell \varepsilon_{j_\ell}<2\epsilon_\ell$)
\begin{equation}\label{sdefC4}
\int_U (\Phi_{\varepsilon_{j_\ell}}\ast\|\partial\mathcal E_{j_\ell}\|+\varepsilon_{j_\ell}\Omega^{-1})\leq  2(k+1)r_\ell(1+2\epsilon_\ell).
\end{equation}
Now, \eqref{defC13}, \eqref{defC16}, \eqref{defC17} and \eqref{sdefC4} prove \eqref{defC6} 
with a suitable choice of absolute constant. 

{\bf Step 3: completion of selection of line segments.}
Using \eqref{defC6}, we can make sure that the selection of line segments does not continue indefinitely
and after a finite selection, we have either
$L_k$ exits from $B_{3r/2}(a)$ or there is an encounter with a triple junction.
We can similarly proceed to choose line segments starting again from $L_0$
in the opposite direction. Let these points and line segments be $z_{-1},\ldots,z_{-k'_0}$ and $L_{-1},\ldots,
L_{-k'_0}$. In this opposite direction, the same estimate for the slope of line segments 
holds. The similar slope estimate holds starting from any $L_{\pm k}$. Namely, for any
$k'<k$ in $\{-k_0',\ldots,k_0\}$, we have
\begin{equation}
\label{defC6re}
\begin{split}
|(y_k^+-y_k)/(x_k^+-x_k)&-(y_{k'}^+-y_{k'})/(x_{k'}^+-x_{k'})| \\
&\leq \Cr{c_8}\Big(\epsilon_\ell+\Big(\int_{\cup_{i=k'}^{k}
U_{r_\ell}(z_{i})} \frac{|\Phi_{\e_{j_\ell}}\ast\delta(\partial\mathcal E_{j_\ell})|^2}{
\Phi_{\e_{j_\ell}}\ast\|\partial\mathcal E_{j_\ell}\|+\e_{j_\ell}\Omega^{-1}}\Big)^{\frac12}
(|k-k'|\,r_\ell+\epsilon_\ell)^{\frac12}\Big).
\end{split}
\end{equation}
This gives 
an analogue of $\sfrac12$-H\"{o}lder estimate for the slopes of segments. 

{\bf Step 4: construction of smooth approximate curve from line segments.}
From $\cup_{k=-k_0'}^{k_0} L_k$, we can construct an approximate curve
as follows. Define $I_k:=\{x\in\mathbb R\,:\, (x,y)\in L_k\}$, i.e. the projection of $L_k$ to the $x$-axis. 
They are intervals of length approximately $2r_\ell$. 
Recall $z_k=(x_k,y_k)$ and $z_k^{\pm}=(x_k^{\pm},y_k^{\pm})$, and due to the construction, we have
$|x_{k}^+ - x_{k+1}|\leq \epsilon_\ell r_\ell$ for $k=0,\ldots,k_0$ and $|x_{-k}^- -x_{-k-1}|\leq 
\epsilon_\ell r_\ell$ for $k=0,\ldots,-k_0'$. With the slope $\leq 1/10$ (for example) and for small 
$\epsilon_\ell$, we can guarantee that at most three $I_k$'s can intersect each other. 
Let $\{\tilde \zeta_k\}_{k=-k_0'}^{k_0}$ be a partition of unity subordinate to $\{{\rm int}\,I_k\}_{k=-k_0'}^{k_0}$
such that $\tilde \zeta_k\in C_c^{\infty}({\rm int}\,I_k)$, $0\leq \tilde\zeta_k\leq 1$ and $\sum_{k=-k_0'}^{k_0}\tilde \zeta_k=1$ on $\cup_{k=-k_0'+1}^{k_0-1}I_k$.
We may in addition ask 
that $\sup(r_\ell |\tilde\zeta_k'(x)|+r_\ell^2 |\tilde \zeta_k''(x)|)\leq c$
for some absolute constant $c$. On each $I_k$, let $g_k:I_k\rightarrow\mathbb R$ be the function so that
$\{(x,g_k(x))\,:\,x\in I_k\}=L_k$ holds. For $x\in \cup_{k=-k_0'+1}^{k_0-1} I_k$, define
\begin{equation}
\label{gfdef}
f(x):=\sum_{k=-k_0'}^{k_0}g_k(x)\tilde\zeta_k(x).
\end{equation}
Note that $f'(x)$ involves at most three neighboring terms on $I_i$, say, $k=i-1,i,i+1$ (or $k=i-2,i-1,i$,
or $k=i,i+1,i+2$). Consider the first case and the other two cases are similar. Due to \eqref{defC1}, 
$(\epsilon_\ell r_\ell)^{-1}|g_i(x)-g_{i\pm 1}(x)|+\epsilon^{-1}_\ell |g_i'(x)-g_{i\pm 1}'(x)|$ is 
bounded by an absolute constant. Since
we have $\sum_{k=i-1}^{i+1}\tilde\zeta_k=1$, 
\begin{equation}\label{sva1}
f'(x)-g_i'(x)=\sum_{k=i-1}^{i+1}\{(g'_k(x)-g_i'(x))\tilde\zeta_k(x)+(g_k(x)-g_i(x))\tilde\zeta_k'(x) \}.
\end{equation}
The right-hand side may be estimated by a constant multiple of $\epsilon_\ell$ using the 
bounds on $\tilde\zeta'_k$. 
The variation of the slopes of $L_k$ (i.e. $g_k'(x)$) are estimated by \eqref{defC6re}, so we obtain  
for any $x,\tilde x\in \cup_{k=-k_0'+1}^{k_0-1} I_k$
\begin{equation}
\label{varf}
|f'(x)-f'(\tilde x)|\leq \Cr{c_8}\Big(\epsilon_\ell+\Big(\int_{\cup_{k=k_1}^{k_2}U_{r_\ell}(z_k)}
\frac{|\Phi_{\varepsilon_{j_\ell}}\ast\delta(\partial\mathcal E_{j_\ell})|^2}{\Phi_{\varepsilon_{j_\ell}}
\ast\|\partial\mathcal E_{j_\ell}\|+\varepsilon_{j_\ell}\Omega^{-1}}\Big)^{\frac12}(|x-\tilde x|
+\epsilon_\ell)^{\frac12}\Big).
\end{equation}
Here, $k_1,k_2\in \mathbb Z$ are chosen so that $|x-x_{k_1}|\leq r_\ell$ and $|\tilde x-x_{k_2}|\leq r_\ell$, and 
$\Cr{c_8}$ here may be different from $\Cr{c_8}$ in \eqref{defC6re} by a factor of absolute constant. 
From the construction, it is also clear that $\sup_{x\in I_k}|g_k(x)-f(x)|$ is estimated by a constant
multiple of $r_\ell \epsilon_\ell$, and due also to \eqref{defC1sup1}, 
there exists an absolute constant $\Cl[c]{c_9}$ such that
\begin{equation}
\label{varf1}
\begin{split}
Z_\ell\cap &\big\{(x,y)\,:\,|y-f(x)|\leq r_\ell/4,\,x\in  \cup_{k=-k_0'+1}^{k_0-1} I_k\big\} \\
&\subset \big\{(x,y)\,:\,  |y-f(x)|\leq \Cr{c_9}\epsilon_\ell r_\ell,\,x\in  \cup_{k=-k_0'+1}^{k_0-1} I_k\big\}.
\end{split}
\end{equation}

{\bf Step 5: Proof for ``no triple junction''.}
We next prove using the assumption \eqref{pdefC2} with small $\Cr{c_7}$ that there is no triple junction
in $B_{9r/8}(a)$ for all large $\ell$, or more precisely, there is no point of $Z_\ell
\setminus Z_\ell^*$. 
First note that, because of Lemma \ref{haudi},
for all sufficiently large $\ell$, we have
\begin{equation}
\label{nai1}
B_{3r/2}(a)\cap Z_\ell\cap\{a+(x,y)\,:\, |y|\geq 2\Cr{c_7} r\}=\emptyset.
\end{equation}
Otherwise, we would have a converging subsequence (with the same index) $\{z^{(\ell)}\}\in 
B_{3r/2}(a)\cap Z_\ell\cap \{a+(x,y)\,:\, |y|\geq 2\Cr{c_7} r\}$ which converges to $\hat z\in
B_{3r/2}(a)\cap \{a+(x,y)\,;\,|y|\geq 2\Cr{c_7} r\}$. By Lemma \ref{haudi}, we have $\|\partial\mathcal E_{j_\ell}\|
(B_{s}(z^{(\ell)}))\geq s/8$ for all large $\ell$ and $s\in (0,\Cr{c_4}]$. 
We then have $\mu (B_{s}(\hat z))
\geq s/8$ for $s\in (0,\Cr{c_4}]$, 
and in particular $\hat z\in {\rm spt}\,\mu$. This is a contradiction to \eqref{pdefC2}. 

Suppose for a contradiction that we had some $z_0\in B_{9r/8}(a)\cap Z_\ell\setminus Z_\ell^*$, thus there 
exists a triple junction $S$ with the junction in $U_{r_\ell}(z_0)$ satisfying $z_0\in S$ and \eqref{defC1tri}.
We may assume that $z_0\in \{a+(x,y)\,:\,|y|<2\Cr{c_7} r\}$ due to \eqref{nai1}. 
Out of the three half-lines of $S$, there is a half-line which goes upwards (i.e. towards the
positive $y$-direction) and which has at least 30 degrees with positive or negative $x$-axis (or 
the absolute value of slope $\geq 1/\sqrt3$). 
Let $z_0^+$ be the intersection of this half-line and $\partial B_{2r_\ell}(z_0)$.
By \eqref{defC1tri}, we know that there exists $z_1\in Z_\ell\cap B_{\epsilon_\ell r_\ell}(z_0^+)$. 
Starting from this $z_1$, we may start selecting line segments until 
the encounter with a triple junction or the exit from $B_{3r/2}(a)$ occur as in Step 1. Suppose 
that the encounter with a triple
junction does not occur. Since one can estimate $|y_k-y_0|/|x_k-x_0|\geq 3/(4\sqrt 3)$ 
(where $z_k=(x_k,y_k)$) due to \eqref{defC6re} with sufficiently small $\Cr{c_6}$,
the $y$-coordinate of $z_k$ keeps going upwards. Then by selecting a sufficiently small absolute
constant $\Cr{c_7}$, the union of line segments (which is close to a straight line due to \eqref{varf}
with small $\Cr{c_6}$) would enter $B_{5r/4}(a)\cap \{a+(x,y)\,:\,|y|\geq 2\Cr{c_7} r\}$
after a finite number of induction. 
Since $L_k$ and $Z_\ell$ in $B_{r_\ell}(z_k)$ are close
in the Hausdorff distance, this would be a contradiction to \eqref{nai1}. Thus, after a finite
number of induction, there must be another triple junction which is close to $Z_\ell$.
This triple junction has one half-line whose angle with the $x$-axis is at least 30 degrees and which 
is going upwards. Along the direction of this half-line, we can select line segments as before, 
which again has a definite slope going upwards. By the similar reasoning, there must be another triple 
junction, and then we choose a half-line going upwards just as before. In this process, we can make sure
that the slope of line segments have lower bound (say, $3/(4\sqrt 3)$, for example) and the line segments
inevitably enter into $B_{5r/4}(a)\cap \{a+(x,y)\,:\,|y|\geq 2\Cr{c_7} r\}$ after a finite number of induction. 
Thus, we inevitably have a contradiction to \eqref{nai1}. This proves that there is no point
of $Z_\ell\setminus Z_\ell^*$ inside $B_{9r/8}(a)$. 

{\bf Step 6: construction of graphs separated by a definite distance.} 
Consider $\tilde Z_\ell:=Z_\ell\cap \{a+(x,y)\,:\,|x|\leq r\}\cap B_{9r/8}(a)$. The set $\tilde Z_\ell$ is included in $\{a+(x,y)\,:\,
|y|\leq 2\Cr{c_7}r\}$ for all large $\ell$ due to \eqref{nai1}, and by Step 5, we have $B_{9r/8}(a)\cap
Z_\ell\setminus Z_\ell^*=\emptyset$. Thus, starting from any $z_0\in \tilde Z_\ell$, the inductive
step to choose $L_k$ (and $L_{-k}$) in Step 1 does
encounter a triple junction and will continue 
until $L_k$ (resp. $L_{-k}$) exits from
$\{a+(x,y)\,:\, |x|\leq r\}$. Namely, with the same notation in Step 1-4, we have some $k_0\geq 1$
and $k_0'\geq 1$ such that $r\in I_{k_0-1}$, $-r\in I_{-k_0'+1}$ and 
$[-r,r]\subset \cup_{k=-k_0-1}^{k_0+1} I_k$ 
hold. The function $f$ constructed in Step 4 over $x\in[-r,r]$ 
has a small slope for small $\Cr{c_6}$ and $\Cr{c_7}$. By \eqref{varf1}, note that 
\begin{equation}\label{nai2}
\tilde Z_\ell\cap \{a+(x,y)\,:\,|y-f(x)|\leq r_\ell/4,|x|\leq r\}\subset \{a+(x,y)\,:\,|y-f(x)|\leq 
\Cr{c_9}\epsilon_\ell r_\ell,|x|\leq r\}.
\end{equation}
Let this $f$ be renamed as $f_{(1,\ell)}$. 
If $\tilde Z_\ell\setminus \{a+(x,y)\,:\,|y-f_{(1,\ell)}(x)|\leq r_\ell/4,|x|\leq r\}\neq \emptyset$, then,
we can pick a point $z_0'$ from this set and start the construction of Step 1-4. Let $f_{(2,\ell)}$ be the
resulting function. Note that 
\begin{equation}\label{nai5}
|f_{(1,\ell)}(x)-f_{(2,\ell)}(x)|\geq r_\ell/8\,\,\,\mbox{ for }x\in[-r,r],
\end{equation}
which can be seen as follows. Because $z_0'=(x_0',y_0')\notin \{(x,y)\,:\, |y-f_{(1,\ell)}(x)|\leq r_\ell/4,|x|\leq r\}$ (which implies $|f_{(1,\ell)}(x_0')-y_0'|>r_\ell/4$)
and $|y_0'-f_{(2,\ell)}(x_0')|\leq \Cr{c_9}\epsilon_\ell r_\ell$ due to \eqref{nai2} satisfied by
$f_{(2,\ell)}$, we have $|f_{(1,\ell)}(x_0')-f_{(2,\ell)}(x_0')|>3r_\ell/16$ for large $\ell$.
If there exists $x\in[-r,r]$ such that $|f_{(1,\ell)}(x)-f_{(2,\ell)}(x)|<r_\ell/8$, then by 
the continuity of $f_{(1,\ell)}$ and $f_{(2,\ell)}$, there must exist some $\hat x\in [-r,r]$ such that
$|f_{(1,\ell)}(\hat x)-f_{(2,\ell)}(\hat x)|=3r_\ell/16$. Then there must be a point $\tilde z
\in Z_\ell\cap B_{c\epsilon_\ell r_\ell}((\hat x,f_{(2,\ell)}(\hat x)))$, which contradicts \eqref{nai2}
with $f=f_{(1,\ell)}$ and sufficiently small $c\epsilon_\ell$. 
Thus, we have \eqref{nai5},
and if $\tilde Z_\ell\setminus \cup_{i=1}^2 \{a+(x,y)\,:\,|y-f_{(i,\ell)}(x)|\leq r_\ell/4,|x|\leq 
r\}\neq\emptyset$, then we can pick a point from this set and construct another function, $f_{(3,\ell)}$.
By repeating this construction and due to the above disjointedness property of these graphs, 
we can exhaust all points of $\tilde Z_\ell$ after finite steps
and we have a sequence
of graphs $f_{(1,\ell)},\ldots,f_{(\nu_\ell',\ell)}$, that is, 
\begin{equation} \label{nai3}
\tilde Z_\ell\setminus \cup_{i=1}^{\nu_\ell'}\{a+(x,y)\,:\,|y-f_{(i,\ell)}(x)|\leq r_\ell/4,|x|\leq 
r\}=\emptyset.
\end{equation}

{\bf Step 7: proof that the number of graphs is $\nu$.}
In the following, we will see that $\nu_\ell'$ has to be necessarily equal to $\nu$ given in \eqref{pdefC3}
for all large $\ell$ if $\Cr{c_6}$ and $\Cr{c_7}$ are sufficiently small. 
Set 
\begin{equation}\label{nai3.5}
\mu_\ell:=\mathcal H^1\mres_{\cup_{i=1}^{\nu_\ell'} \{a+(x,f_{(i,\ell)}(x))\,:\, |x|\leq r\}},
\end{equation}
that is, $\mu_\ell$ is the 1-dimensional measure obtained from the union of above graphs. 
We claim that 
\begin{equation}
\label{nai4}
\mu=\lim_{\ell\rightarrow\infty} \|\partial\mathcal E_{j_\ell}\|=\lim_{\ell\rightarrow\infty}
\mu_\ell
\end{equation}
in $U_{9r/8}(a)\cap\{a+(x,y)\,:\, |x|\leq r\}$ as measures. The difference as measures between 
$\|\partial\mathcal E_{j_\ell}\|$ and $\mathcal H^1\mres_{Z_\ell}$ is negligible in the limit
due to Remark
\ref{gork}, and within the domain under consideration, $Z_\ell=\tilde Z_\ell$, thus we need to prove $\lim_{\ell\rightarrow\infty}\mathcal H^1\mres_{\tilde Z_\ell}=
\lim_{\ell\rightarrow\infty}\mu_\ell$. For this, fix a smooth function $\phi\in C_c^{\infty}(
U_{9r/8}(a)\cap\{a+(x,y)\,:\,|x|\leq r\})$ with $|\phi|\leq 1$. 
We know that $\tilde Z_\ell$ is within $\Cr{c_9}\epsilon_\ell
r_\ell$-neighborhood of graphs of $f_{(1,\ell)},\ldots,f_{(\nu_\ell',\ell)}$ by \eqref{nai2},
and these graphs are separated by $r_\ell/8$ as in \eqref{nai5}. Now, take a neighborhood of the graph
of $f_{(1,\ell)}$. Recall that $f_{(1,\ell)}$ is defined as in \eqref{gfdef}, with the partition of unity
$\{\tilde\zeta_k\}_{k=-k_0'}^{k_0}$ subordinate to the intervals $\{I_k\}_{k=-k_0'}^{k_0}$ such that
$-r\in I_{-k_0'+1}$, $r\in I_{k_0-1}$ and $[-r,r]\subset\cup_{k=-k_0'+1}^{k_0-1}I_k$ (these objects 
depend also on the indices of the functions but we drop the dependence). Let $\psi_{(1,\ell)}$ be a function such that 
\begin{equation}
\label{naiad}
\begin{array}{ll}
\psi_{(1,\ell)}=1 & \mbox{ on }
\{a+(x,y)\,:\, |y-f_{(1,\ell)}(x)|<r_\ell/32,\, |x|\leq r\}, \\
\psi_{(1,\ell)}=0 & \mbox{ on }
\{a+(x,y)\,:\, |y-f_{(1,\ell)}(x)|>r_\ell/16,\, |x|\leq r\},
\end{array}
\end{equation}
$0\leq \psi_{(1,\ell)}\leq 1$ and 
$\sup (r_\ell |\nabla\psi_{(1,\ell)}|)\leq c$ where $c$ is an absolute constant. Such $\psi_{(1,\ell)}$ can be 
constructed due to the definition of $f_{(1,\ell)}$. We then use \eqref{sdefC1} with $\tilde\zeta_k \psi_{(1,\ell)}
\phi$ in place of $\phi$ there. Note that the projection of $U_{r_\ell}(z_k)$ to the $x$-axis is $I_k$
in this case, and we have $\tilde\zeta_k \psi_{(1,\ell)}
\phi\in C_c^1(U_{r_\ell}(z_k))$ for each $k=-k_0'+1,\ldots,k_0-1$ and for sufficiently large $\ell$ 
\begin{equation}
\big|\int_{B_{r_\ell}(z_k)}\tilde \zeta_k\psi_{(1,\ell)}\phi\,d\|\partial\mathcal E_{j_\ell}\|-\int_{B_{r_\ell}
(z_k)\cap L_k}\tilde\zeta_k\psi_{(1,\ell)}\phi\,d\mathcal H^1\big|\leq \epsilon_\ell r_\ell.
\label{nai6}
\end{equation}
By the definition \eqref{gfdef} and the subsequent discussion, within $U_{r_\ell}(z_k)$, the distance
between $L_k$ and the graph of $f_{(1,\ell)}$ is within $c\epsilon_\ell r_\ell$
and the difference of the derivatives is within $c\epsilon_\ell$ (with absolute constant $c$), so that
we may estimate as
\begin{equation}
\label{nai8}
\big|\int_{B_{r_\ell}
(z_k)\cap L_k}\tilde\zeta_k\psi_{(1,\ell)}\phi\,d\mathcal H^1-\int_{B_{r_\ell}(z_k)\cap {\rm graph}\,f_{(1,\ell)}}
\tilde\zeta_k\psi_{(1,\ell)}\phi\,d\mathcal H^1\big|\leq c(\|\phi\|_{C^1})\epsilon_\ell r_\ell. 
\end{equation}
Now, summing over $k$ and using the fact that $\sum \tilde\zeta_k=1$ and $\psi_{(1,\ell)}=1$ on ${\rm graph}\,f_{(1,\ell)}$, we obtain from \eqref{nai6} and \eqref{nai8} that
\begin{equation}
\label{nai9}
\big|\int \psi_{(1,\ell)} \phi\,d\|\partial\mathcal E_{j_\ell}\|-\int_{{\rm graph}\,f_{(1,\ell)}}\phi\,d\mathcal H^1
\big|\leq c(\|\phi\|_{C^1})\epsilon_\ell r,
\end{equation} 
where we used $(k_0+k_0')r_\ell\leq cr$ with an absolute constant (for example, we may take $c=5$). 
Now, suppose that $\nu_\ell'\geq \nu+1$ for large $\ell$. Then, we can obtain \eqref{nai9} for 
each $f_{(1,\ell)},\ldots,f_{(\nu+1,\ell)}$, and summing over $i=1,\ldots,\nu+1$, we have
(with $\psi_{(2,\ell)},\ldots,\psi_{(\nu+1,\ell)}$ being the corresponding cut-off functions for
graphs of $f_{(2,\ell)},\ldots,f_{(\nu+1,\ell)}$)
\begin{equation}
\label{nai10}
\big|\int\sum_{i=1}^{\nu+1}\psi_{(i,\ell)}\phi\,d\|\partial\mathcal E_{j_\ell}\|-\sum_{i=1}^{\nu+1}
\int_{{\rm graph}\,f_{(i,\ell)}}\phi\,d\mathcal H^1\big|\leq c(\nu,\|\phi\|_{C^1})\epsilon_\ell r.
\end{equation}
We choose $\phi$ which is a smooth approximation to the characteristic function of the set
$\{a+(x,y)\,:\,|y|< 2\Cr{c_7}r, |x|<r\}$ but which has a compact support within the same set.
We may assume that the graphs of $f_{(1,\ell)},\ldots,f_{(\nu+1,\ell)}$ are within 
$\{a+(x,y)\,:\,|y|<3\Cr{c_7}r/2\}$ and we can make sure that the following inequality holds:
\begin{equation}
\label{nai11}
\sum_{i=1}^{\nu+1}\int_{{\rm graph}\,f_{(i,\ell)}}\phi\, d\mathcal H^1\geq 2r(\nu+\sfrac34).
\end{equation}
On the other hand, since the supports of $\psi_{(1,\ell)},\ldots,\psi_{(\nu+1,\ell)}$ are mutually
disjoint, \eqref{pdefC2sup} implies
\begin{equation}
\label{nai12}
\limsup_{\ell\rightarrow\infty}\int\sum_{i=1}^{\nu+1}\psi_{(i,\ell)}\phi\,d\|\partial\mathcal E_{j_\ell}\|
\leq \int\phi\,d\|\partial\mathcal E_{j_\ell}\|\leq 
2r(\nu+\sfrac12).
\end{equation}
Then \eqref{nai10}-\eqref{nai12} lead to a contradiction as $\ell\rightarrow\infty$. Thus we proved 
$\nu_\ell'\leq \nu$. To see that $\nu_\ell'=\nu$, assume otherwise, that is, $\nu_\ell'\leq \nu-1$.
This time, we choose $\phi$ which is a smooth approximation to the characteristic function of the set 
$\{a+(x,y)\,:\, |y|<2\Cr{c_7},|x|<r/2\}$ but which has a compact support within the same set. 
We may assume that the slopes of $f_{(1,\ell)},\ldots,f_{(\nu_\ell',\ell)}$ are small by restricting
$\Cr{c_6}$ and $\Cr{c_7}$ and we can make sure that
\begin{equation}
\label{nai13}
\sum_{i=1}^{\nu_\ell'}\int_{{\rm graph}\,f_{(i,\ell)}}\phi\, d\mathcal H^1\leq  r(\nu_\ell'+\sfrac14)
\leq r(\nu-\sfrac34).
\end{equation}
By \eqref{nai3}, on the other hand, we have
\begin{equation}
\label{nai14}
\liminf_{\ell\rightarrow\infty}\int\sum_{i=1}^{\nu_\ell'}\psi_{(i,\ell)}\phi\,d\|\partial\mathcal E_{j_\ell}\|
=\lim_{\ell\rightarrow\infty} \int_{\tilde Z_\ell} \phi\,d\mathcal H^1 =\int\phi\,d\mu.
\end{equation}
By \eqref{pdefC3} and choosing an appropriate $\phi$ to begin with, we may assume
\begin{equation}
\label{nai15}
\int\phi\,d\mu\geq r(\nu-\sfrac23).
\end{equation}
Since we have \eqref{nai10} with $\nu+1$ there replaced by $\nu_\ell'$, \eqref{nai13}-\eqref{nai15} lead
to a contradiction, which proves $\nu_\ell'=\nu$ at last. 

{\bf Step 8: the limit functions are in $W^{2,2}$.}
Thus, for sufficiently large $\ell$, we have \eqref{nai10} with $\nu+1$ there replaced by $\nu$, and 
we may assume that $f_{(1,\ell)}<f_{(2,\ell)}<\ldots<f_{(\nu,\ell)}$ for $|x|\leq r$. These
functions satisfy \eqref{varf}, and by Arzel\`{a}-Ascoli theorem (with a slight modification of the 
proof due to the vanishing error $\epsilon_\ell$), there exists a 
subsequence denoted by the same index converging in the $C^1$ norm
and the limit $C^{1,\sfrac12}$ functions $f_1=\lim_{\ell\rightarrow\infty}f_{(1,\ell)}\leq \ldots\leq f_\nu=\lim_{\ell\rightarrow\infty}f_{(\nu,\ell)}$ defined on $|x|\leq r$. 
Because of the uniform $C^1$ convergence and by \eqref{nai10} (with $\nu$ in place of $\nu+1$), we proved
\eqref{nai4} as well as \eqref{pdefC5}. 
Lastly, we prove that the limit functions are in
$W^{2,2}$. We prove that for any $\phi\in C_c^2((-r,r))$ and for any $i=1,\ldots,\nu$, we have
\begin{equation}
\label{nai17}
\int \phi'\frac{f_i'}{\sqrt{1+(f_i')^2}}\,dx\leq \Big(\frac{\Cr{c_6}}{r} \int  \phi^2\sqrt{1+
(f_i')^2}\,dx\Big)^{\frac12}.
\end{equation}
This proves that $f_i'/\sqrt{1+(f_i')^2}$ has the weak derivative
in $L^2$, and which shows that $f_i\in W^{2,2}$. Assume $i=1$ and the other cases are similar. 
By the $C^1$ convergence, we have
\begin{equation}
\label{nai18}
\int \phi'\frac{f_1'}{\sqrt{1+(f_1')^2}}\,dx=\lim_{\ell\rightarrow\infty} 
\int \phi'\frac{f_{(1,\ell)}'}{\sqrt{1+(f_{(1,\ell)}')^2}}\,dx.
\end{equation}
Fix a large $\ell$. We go back to the argument and notation following \eqref{nai4}. Since $\sum \tilde \zeta_k=1$, 
we have
\begin{equation}
\label{nai19}
\int \phi'\frac{f_{(1,\ell)}'}{\sqrt{1+(f_{(1,\ell)}')^2}}\,dx=\sum_{k=-k_0'}^{k_0} 
\int_{I_k} \tilde \zeta_k \phi'\frac{f_{(1,\ell)}'}{\sqrt{1+(f_{(1,\ell)}')^2}}\,dx.
\end{equation}
On each $I_k$, by \eqref{sva1},
\begin{equation}
\label{nai20}
\Big|\int_{I_k} \tilde \zeta_k \phi'\frac{f_{(1,\ell)}'}{\sqrt{1+(f_{(1,\ell)}')^2}}\,dx
-\int_{I_k} \tilde \zeta_k \phi'\frac{g_k'}{\sqrt{1+(g_k')^2}}\,dx\Big|\leq c(\sup|\phi'|)\epsilon_\ell r_\ell.
\end{equation}
Recall that ${\rm graph}\,g_k$ represents $L_k$, and the $(1,2)$-component of the orthogonal projection to 
the tangent space of $L_k$ is given by $g_k'/(1+(g_k')^2)$. Thus, in terms of the language of varifold, 
\begin{equation}
\label{nai21}
\int_{I_k} \tilde \zeta_k\phi'\frac{g_k'}{\sqrt{1+(g_k')^2}}\,dx=\int_{{\bf G}_1(U_{r_\ell}(z_k))}
\tilde \zeta_k(x)\phi'(x) \tilde S_{12}\,d(|L_k|)(z,\tilde S),
\end{equation}
where $z=(x,y)$. Since $\psi_{(1,\ell)}=1$ on $L_k$, we have
\begin{equation}
\label{nai21sub}
\int_{{\bf G}_1(U_{r_\ell}(z_k))}
\tilde \zeta_k(x)\phi'(x) \tilde S_{12}\,d(|L_k|)(z,\tilde S)=\int_{{\bf G}_1(U_{r_\ell}(z_k))} 
\tilde \zeta_k(x)\phi'(x) \psi_{(1,\ell)}(z)\tilde S_{12}\,d(|L_k|)(z,\tilde S).
\end{equation}
We next use \eqref{defC2} to obtain
\begin{equation}
\label{nai22}
\Big| 
\int_{{\bf G}_1(U_{r_\ell}(z_k))}  \tilde\zeta_k\phi'\psi_{(1,\ell)}\tilde S_{12}\,d(|L_k|)
-\int_{{\bf G}_1(U_{r_\ell}(z_k))}\tilde\zeta_k\phi'\psi_{(1,\ell)}\tilde S_{12}\,d(\partial\mathcal E_{j_\ell})
\Big| \leq \epsilon_\ell r_\ell. 
\end{equation}
For this to be true, we note that $\tilde\zeta_k\phi'\psi_{(1,\ell)}\in C_c^1(U_{r_\ell}(z_k))$ and
$r_\ell|\nabla(\tilde\zeta_k\phi'\psi_{(1,\ell)})|$ can be bounded by a
constant depending only on $\|\phi\|_{C^2}$, thus \eqref{defC2} is valid with this function
for sufficiently large $\ell$. 
By \eqref{nai19}-\eqref{nai22}, and using that $(k_0+k_0')r_\ell\leq 5r$, we obtain
\begin{equation}
\label{nai23}
\Big|\int \phi'\frac{f_{(1,\ell)}'}{\sqrt{1+(f_{(1,\ell)}')^2}}\,dx -
\int_{{\bf G}_1(\R^2)} \phi'\psi_{(1,\ell)} \tilde S_{12}\,d(\partial \mathcal E_{j_\ell})
\Big|\leq 5 (c(\sup|\phi'|)+1)\epsilon_\ell r. 
\end{equation}
In particular, \eqref{nai18} and \eqref{nai23} show
\begin{equation}
\label{nai24}
\int \phi'\frac{f_1'}{\sqrt{1+(f_1')^2}}\,dx=\lim_{\ell\rightarrow\infty} 
\int_{{\bf G}_1(\R^2)} \phi'\psi_{(1,\ell)} \tilde S_{12}\,d(\partial \mathcal E_{j_\ell}).
\end{equation}
Now $\phi'(x)\psi_{(1,\ell)}(z)=\frac{\partial}{\partial x}(\phi\psi_{(1,\ell)})-\phi\frac{\partial}{\partial
x}\psi_{(1,\ell)}$ and $\nabla\psi_{(1,\ell)}$ is non-zero only on $\{a+(x,y)\,:\, r_\ell/32\leq |y-f_{(1,\ell)}(x)|
\leq r_\ell/16\}$ by \eqref{naiad}. By \eqref{nai2}, on this set, there is no 
point of $\tilde Z_\ell$. Since $\partial\mathcal E_{j_\ell}=\tilde Z_\ell\cup \tilde Z_\ell^c$
and the measure of $\tilde Z_\ell^c$ is negligible as noted in Remark \ref{gork}, we have
\begin{equation}
\label{nai25}
\lim_{\ell\rightarrow\infty} 
\int_{{\bf G}_1(\R^2)} \phi'\psi_{(1,\ell)} \tilde S_{12}\,d(\partial \mathcal E_{j_\ell})
=\lim_{\ell\rightarrow\infty} 
\int_{{\bf G}_1(\R^2)} \Big\{\frac{\partial}{\partial x}(\phi\psi_{(1,\ell)})\tilde S_{12}
+\frac{\partial}{\partial y}(\phi\psi_{(1,\ell)})\tilde S_{22}\Big\}\,d(\partial \mathcal E_{j_\ell}).
\end{equation}
By the definition of the first variation, we have
\begin{equation}
\label{nai26}
\int_{{\bf G}_1(\R^2)} \Big\{\frac{\partial}{\partial x}(\phi\psi_{(1,\ell)})\tilde S_{12}+
\frac{\partial}{\partial y}(\phi\psi_{(1,\ell)})\tilde S_{22}\Big\}\,d(\partial \mathcal E_{j_\ell})
=\delta(\partial\mathcal E_{j_\ell})((0,\phi\psi_{(1,\ell)})).
\end{equation}
We estimate as in \eqref{defC14}-\eqref{defC16} with $\psi$ there replaced by $\phi\psi_{(1,\ell)}$,
which gives
\begin{equation}
\label{nai26ad}
\lim_{\ell\rightarrow\infty}\big|\delta(\partial\mathcal E_{j_\ell})((0,\phi\psi_{(1,\ell)}))
-(\Phi_{\varepsilon_{j_\ell}}\ast\delta(\partial\mathcal E_{j_\ell}))((0,\phi\psi_{(1,\ell)}))\big|=0.
\end{equation}
In place of \eqref{defC17}, we obtain (with $U={\rm spt}\,\phi\psi_{(1,\ell)}$ and \eqref{pdefC1}) 
\begin{equation}
\label{nai27}
\begin{split}
|(\Phi_{\varepsilon_{j_\ell}}\ast\delta(\partial\mathcal E_{j_\ell}))((0,\phi\psi_{(1,\ell)}))|
&\leq \big(\int_U \frac{|\Phi_{\varepsilon_{j_\ell}}\ast \delta(\partial\mathcal E_{j_\ell})|^2}
{\Phi_{\varepsilon_{j_\ell}}\ast\|\partial\mathcal E_{j_\ell}\|+\varepsilon_{j_\ell}\Omega^{-1}}
\big)^{\frac12}\big(\int (\phi\psi_{(1,\ell)})^2(\Phi_{\varepsilon_{j_\ell}}\ast
 \|\partial\mathcal E_{j_\ell}\|+\varepsilon_{j_\ell}\Omega^{-1})\big)^{\frac12} \\
&\leq (\Cr{c_6}/r)^{\sfrac12}\big(\int (\phi\psi_{(1,\ell)})^2(\Phi_{\varepsilon_{j_\ell}}\ast
 \|\partial\mathcal E_{j_\ell}\|+\varepsilon_{j_\ell}\Omega^{-1})\big)^{\frac12}.
 \end{split}
\end{equation}
Since $\phi\psi_{(1,\ell)}$ is bounded, it follows 
\begin{equation}
\label{nai28}
\lim_{\ell\rightarrow\infty} \int(\phi\psi_{(1,\ell)})^2\varepsilon_{j_\ell}\Omega^{-1}=0. 
\end{equation}
We can also estimate $|\Phi_{\varepsilon_{j_\ell}}\ast (\phi\psi_{(1,\ell)})-\phi\psi_{(1,\ell)}|$
as in \eqref{defC15} so that
\begin{equation}
\label{nai29}
\lim_{\ell\rightarrow\infty} \Big|\int(\phi\psi_{(1,\ell)})^2(\Phi_{\varepsilon_{j_\ell}}\ast
\|\partial\mathcal E_{j_\ell}\|)-\int(\phi\psi_{(1,\ell)})^2\,d\|\partial\mathcal E_{j_\ell}\|\Big|=0.
\end{equation}
By the similar argument leading to \eqref{nai9} and the $C^1$ convergence of $f_{(1,\ell)}$
to $f_1$, one can prove that
\begin{equation}
\label{nai30}
\lim_{\ell\rightarrow\infty} \int (\phi\psi_{(1,\ell)})^2\,d\|\partial\mathcal E_{j_\ell}\|=
\int_{{\rm graph} f_1}\phi^2\, d\mathcal H^1.
\end{equation} 
Combining \eqref{nai24}-\eqref{nai30}, we obtain \eqref{nai17}, proving that $f_i\in W^{2,2}([-r,r])$. 
This concludes the proof of
Theorem \ref{regflat}.
\end{proof}
\section{Proof of main theorems}\label{Pr}
Finally, we give a proof of Theorem \ref{KTmain2} and \ref{KTmain3}:
\begin{proof}
As stated in Theorem \ref{limreg}, for almost all $t\in [0,\infty)$, 
the limit varifold $V_t$ is integral with locally square-integrable generalized mean curvature, and we have
Theorem \ref{regflat} available for this $V_t$, where $\mu=\|V_t\|$ there. We omit $t$ in the following.
By the monotonicity formula (see \cite[17.7]{Simon} for the precise form we need),
the density function $\theta^1(\|V\|,\cdot)$ is an upper semicontinuous function on $\R^2$ and $
\theta^1(\|V\|,x)\geq 1$ for $x\in {\rm spt}\,\|V\|$. Moreover, for any sequence $R_i\rightarrow 0+$ and
$z_0\in {\rm spt}\,\|V\|$, define $F_i(z):=(z-z_0)/R_i$, and consider the sequence $(F_i)_{\sharp} V$.
By the well-known argument on the existence of tangent cone (see for example \cite[Sec. 42]{Simon}), there exists a subsequence (denoted
by the same index) and the limit $\tilde V=\lim_{i\rightarrow\infty} (F_i)_{\sharp} V$ which is stationary integral varifold and which is homogeneous degree $0$.
That means that there exist distinct half-lines emanating from the origin
$M_1,\ldots,M_m$ and integer multiplicities $\theta_1,\ldots,\theta_m$ such that $\tilde V=\sum_{k=1}^m
\theta_k|M_k|$. If $m=2$, the stationarity of $\tilde V$ implies $\theta_1=\theta_2$ and $M_1$ and $M_2$
are parallel, that is, ${\rm spt}\,\|\tilde V\|$ is a line through the origin. If $\theta_1=1$,
we can apply the Allard regularity theorem \cite{Allard} and conclude that ${\rm spt}\,\|V\|$ is
a $W^{2,2}$ curve in a neighborhood of $z_0$. If $\theta_1\geq 2$, then, we apply  
Theorem \ref{regflat} to $V$ with $\nu=\theta_1$, $a=z_0$ and $B_{R_i}(a)$. 
The condition \eqref{pdefC1} is satisfied for 
all sufficiently large $i$ since $R_i\rightarrow 0+$.
We may assume after rotation that $M_1$ is 
parallel to the $x$-axis, and ${\rm spt}\,\|(F_i)_{\sharp} V\|$ converges to the $x$-axis locally
in the Hausdorff distance due to the monotonicity formula. Thus \eqref{pdefC2} is satisfied for all large $i$,
and the convergence of $\|(F_i)_{\sharp}\,V\|$ to $\nu\mathcal H^1\mres_{\{y=0\}}$ shows that
\eqref{pdefC3} and \eqref{pdefC2sup} are satisfied for all large $i$. Thus, there exists a neighborhood
of $z_0$ having the description of \eqref{pdefC5}, that is, ${\rm spt}\,\|V\|$ is a union of $W^{2,2}$
curves tangent at $z_0$. Let ${\rm reg}\,V$ be the set of points where the tangent cone is a 
line with multiplicity as above, and
let ${\rm sing}\, V$ be ${\rm spt}\,\|V\|\setminus {\rm reg}\,V$. This is the set of points where
there exists a tangent cone which is not a line, that is, there exists a tangent cone of the form
$\tilde V=\sum_{k=1}^m\theta_k|M_k|$ with $m\geq 3$. Note that ${\rm sing}\,V$ is a closed set.
Moreover, by the following well-known argument, it is a discrete set: 
Otherwise, there is a sequence $\{z_i\}_{i=1}^\infty\subset{\rm sing}\,V$
converging to $z_0\in {\rm sing}\,V$. We may consider a map $F_i(z)=(z-z_0)/|z_i-z_0|$ and 
$(F_i)_{\sharp} V$. By the same argument, a subsequence converges to a tangent cone, and we may assume 
after choosing a further subsequence that $(z_i-z_0)/|z_i-z_0|$ converges to $\tilde z$ with $|\tilde z|=1$.
But then, $(F_i)_{\sharp} V$ approaches to a line with possible multiplicity 
near $\tilde z$ as $i\rightarrow\infty$, which 
implies from the preceding argument that $(F_i)_{\sharp} V$ is regular in a neighborhood of $\tilde z$, and that means that 
$z_i$ is in ${\rm reg}\, V$, a contradiction.

Next, fix $z_0\in{\rm sing}\,V$ and after a change of variables, assume $z_0=0$. Suppose $\lim_{\ell\rightarrow\infty} (F_i)_{\sharp} V=\sum_{k=1}^m
\theta_k|M_k|$, where $F_i(z)=z/R_i$ and $R_i\rightarrow 0+$. Assume that $M_1,\ldots,M_m$ are ordered
counterclockwise. Let us denote the annulus $\{z\,:\,R_i\leq |z|
\leq 2R_i\}$ as $A_i$. Since the convergence is also in the Hausdorff distance for $(F_i)_{\sharp} V$, 
we have
$\lim_{i\rightarrow\infty}(R_i)^{-1}d_H({\rm spt}\,\|V\|\cap A_i,\cup_{k=1}^m M_k\cap A_i)=0$. 
Depending only on the smallest angle between the half lines $M_1,\ldots,M_m$, we choose a sufficiently small
$\beta>0$
so that for each $k=1,\ldots,m$ and $\tilde z\in M_k\cap \{z\,:\,1\leq |z|\leq 2\}$, we have $B_{4\beta}
(\tilde z)\cap M_{k'}=\emptyset$ for all $k'\neq k$. By the stated convergence, 
we may apply Theorem \ref{regflat} to each $B_{2\beta R_i}(\tilde z)$ for $\tilde z\in M_k\cap A_i$
and conclude that $\|V\|$ in $A_i$
is represented by $\theta_k$ graphs over $M_k$ denoted by $f_1^{(k)}\leq\ldots\leq f_{\theta_k}^{(k)}$
of $W^{2,2}$ functions near $A_i\cap M_k$. Here, the index is chosen so that ${\rm graph}\,f_1^{(k)},\ldots,
{\rm graph}\, f_{\theta_k}^{(k)}$ are ordered counter-clockwise for each $k=1,\ldots,m$ (the ``graph over $M_k$'' means that $M_k$
is identified with the positive $x$-axis and the graph is considered as $(x,f_1^{(k)}(x))$
and similarly for others with this coordinate). For all large $i$, these graphs are $C^1$ close to $M_k$, 
so their slopes may be arbitrarily close to that of $M_k$ in $A_i$ by choosing a large $i$.  
Recall that we have a sequence $\{\partial\mathcal E_{j_\ell}\}$ converging to $V$ and
that we may consider $Z_\ell$ in place of $\{\partial\mathcal E_{j_\ell}\}$ due to Remark \ref{gork}. 
Since ${\rm sing}\,V$ is a discrete set, we may assume $B_{4R_i}\cap {\rm sing}\,V=\{0\}$. 
We repeat the argument in the proof of Theorem \ref{regflat}, with the same notation and with
$B_{4R_i}$ in place of $B_{2r}(a)$ there. For each $z\in {\rm spt}\,\|V\|\setminus \{0\}\subset {\rm reg}\,V$, we saw
in the proof of Theorem \ref{regflat} that there exists a neighborhood $B_{\epsilon_z}(z)$ 
and a subsequence $\{Z_{\ell'}\}\subset \{Z_\ell\}$ such that $B_{\epsilon_z}(z)\cap (Z_{\ell'}
\setminus Z_{\ell'}^*)=\emptyset$ for all sufficiently large $\ell'$, or, there is no triple junction
in a neighborhood. Thus, by a diagonal argument, 
we may choose a subsequence (denoted by the same index) such that 
\begin{equation}\label{nojun}
(Z_\ell\setminus Z_{\ell}^*)\cap 
(B_{4R_i}\setminus B_{\epsilon})=\emptyset\,\,\mbox{ for all sufficiently large $\ell$ for each fixed $\epsilon\in (0,R_i)$}
\end{equation}
and we consider this subsequence in the following. 
For the construction of the approximate $C^{1,\sfrac12}$
curves with junctions, we start in the neighborhood of $A_i\cap M_k$. As we saw in the
proof of Theorem \ref{regflat}, in $A_i$, we may construct $C^{1,\sfrac12}$ 
functions $f_{(1,\ell)}^{(k)}<\ldots<f_{(\theta_k,\ell)}^{(k)}$ over $M_k\cap A_i$ each of which converges
to $W^{2,2}$ functions $f_1^{(k)},\ldots,f_{\theta_k}^{(k)}$ in $C^1$ topology as $\ell\rightarrow\infty$. 
Thus we have 
\begin{equation}
\label{nojun1}
\|V\|\mres_{A_i}=\sum_{k=1}^m\sum_{j=1}^{\theta_k}\mathcal H^1\mres_{A_i\cap {\rm graph}\,f_j^{(k)}}.
\end{equation}
Consider
in particular $M_1$ and $f^*_\ell:=f_{(\theta_1,\ell)}^{(1)}$ and for convenience, consider the coordinate
system so that $M_1$ is the positive $x$-axis. We may continue choosing line segments to extend the 
${\rm graph}\, f^*_\ell$ in the negative direction until (a) the exit from $B_{2R_i}$ occurs or (b) the encounter with a triple junction occurs. 
Note that (a) is not possible: for a contradiction, this would mean that we can construct $f^*_\ell$ for
$x\in[-R_i,2R_i]$ with the property that (see \eqref{varf1})
\begin{equation*}
Z_\ell\cap \{(x,y)\,:\, |y-f^*_\ell(x)|\leq r_\ell/4,x\in[-R_i,2R_i]\}
\subset \{(x,y)\,:\, |y-f^*_\ell(x)|\leq \Cr{c_9}\varepsilon_\ell r_\ell,x\in[-R_i,2R_i]\},
\end{equation*}
which means that there is an empty horizontal strip devoid of $Z_\ell$ just above 
${\rm graph}\,f^*_\ell$ over $[-R_i,R_i]$ in particular. 
But then, as one construct a graph of $f_{(1,\ell)}^{(2)}$ starting near $M_2\cap A_i$ towards the 
origin, we would have the following contradiction. 
Since the angle between ${\rm graph}\,f_{(1,\ell)}^{(2)}$
and ${\rm graph}\,f^*_\ell$ is strictly smaller than $\pi$ (since the 
angle between $M_1$ and $M_2$ is), we cannot extend ${\rm graph}\,f_{(1,\ell)}^{(2)}$ past this horizontal strip
without having a triple junction. The argument of Step 5 shows that the resulting network of line
segments starting from the
triple junction encountered by ${\rm graph}\,f_{(1,\ell)}^{(2)}$ has to eventually needs to approach to this horizontal strip non-tangentially, but that would again be 
a contradiction.
Thus (a) is not possible, and we have the case (b). This implies that there exists 
$x_\ell:=x_{(\theta_1,\ell)}^{(1)}\in [-R_i,R_i]$ such that $f^*_\ell$ is defined
on $[x_\ell,2R_i]$, and 
that $(Z_\ell\setminus Z_\ell^*)\cap B_{\epsilon_\ell r_\ell}((x_\ell,f^*_\ell(x_\ell)))\neq 
\emptyset$. 
By \eqref{nojun}, $\lim_{\ell\rightarrow\infty} (x_\ell,f_\ell^*(x_\ell))=0$. We can similarly 
argue that each $f_{(j,\ell)}^{(1)}$ ($j=1,\ldots,\theta_1-1$) can be defined on
$[x_{(j,\ell)}^{(1)}, 2R_i]$ with $\lim_{\ell\rightarrow\infty}(x_{(j,\ell)}^{(1)},f_{(j,\ell)}^{(1)}(x_{(j,\ell)}^{(1)}))=0$. Since we have
$C^{1,\sfrac12}$ estimates \eqref{varf} for these functions, there exists a subsequence such that
$f_{(j,\ell)}^{(1)}$ converges as $\ell\rightarrow \infty$ 
to $f_j^{(1)}\in W^{2,2}$ defined on $[0,2R_i]$ with respect 
to the $C^1([\epsilon,2R_i])$-norm for fixed but arbitrary $\epsilon>0$ for each $j=1,\ldots,\theta_1$. 
We can make sure that $\sup_{x\in[0,2R_i]}|(f_j^{(1)})'(x)|$ is small due to \eqref{varf} for 
each $j=1,\ldots,\theta_1$, thus in particular $(f_j^{(1)})'(0)$ can be made so small that they are 
different from the slopes of
any of $M_2,\ldots, M_m$. Then, we claim that $(f_j^{(1)})'(0)=0$. Otherwise, note that 
$\lim_{i'\rightarrow\infty} (F_{i'})_{\sharp}V$ would have to have a half-line with slope given by 
$(f_j^{(1)})'(0)\neq 0$ with respect to the $x$-axis, which is different from any of $M_1,\ldots,M_m$, 
a contradiction. In summary, we may conclude the following: for arbitrary $\epsilon\in (0,R_i)$, $C^{1,\sfrac12}$ 
functions
$f_{(1,\ell)}^{(1)}<\ldots<f_{(\theta_1,\ell)}^{(1)}$ satisfy
\begin{equation}
\label{oritan3}
\lim_{\ell\rightarrow\infty} 
\|f_{(j,\ell)}^{(1)}-f^{(1)}_{j}\|_{C^1([\epsilon,2R_i])}=0,
\end{equation}
\begin{equation}
\label{oritan4}
Z_\ell\cap\{(x,y)\,:\, |y-f_{(j,\ell)}^{(1)}(x)|\leq r_\ell/4,x\in[\epsilon,2R_i]\}\subset
\{(x,y)\,:\, |y-f_{(j,\ell)}^{(1)}(x)|\leq \Cr{c_9}\epsilon_\ell r_\ell,x\in[\epsilon,2R_i]\},
\end{equation}
and 
\begin{equation}
\label{oritan5}
(f_j^{(1)})'(0)=0
\end{equation}
for $j=1,\ldots,\theta_1$. The similar conclusion follows near $M_2,\ldots,M_m$ as well. We next claim
that $Z_\ell\cap \{z\,:\, \epsilon\leq |z|\leq R_i\}$ is all included in the 
$\Cr{c_9}\epsilon_\ell r_\ell$-neighborhood of $\cup_{k=1}^m\cup_{j=1}^{\theta_k}\,{\rm graph}\,f_{(j,\ell)}^{(k)}$. Otherwise, we have a sequence $z_\ell
\in Z_\ell\cap \{z\,:\, \epsilon\leq |z|\leq R_i\}$ outside of the neighborhood. 
We can construct a $C^{1,\sfrac12}$ curve with small
slope variation starting from $z_\ell$, and since there is no point of $Z_\ell\setminus Z_\ell^*$ within
$B_{4R_i}\setminus B_\epsilon$ by \eqref{nojun} and it has to be disjoint from $\cup_{j,k}\,{\rm graph}\,f_{(j,\ell)}^{(k)}$, one can argue that this curve denoted by $Q_\ell$
is close to a straight
line intersecting $\partial B_{\epsilon}$ and $\partial B_{2R_i}$. In particular, $\lim_{\ell\rightarrow
\infty}\mathcal H^1\mres_{A_i\cap Q_\ell}$ contributes positively to $\|V\|$ on $A_i$ and we can argue that
\begin{equation*}
\|V\|\mres_{A_i}>\sum_{k=1}^m\sum_{j=1}^{\theta_k}\mathcal H^1\mres_{A_i\cap {\rm graph}\,f_j^{(k)}}.
\end{equation*}
But this is a contradiction to \eqref{nojun1}. Thus, we have the conclusion. 
Since $Z_\ell$ is covered by $\Cr{c_9}\epsilon_\ell r_\ell$-neighborhood of $\cup_{j,k}\,{\rm graph}\,f_{(j,\ell)}^{(k)}$, the same argument proving
\eqref{nai4} shows 
\begin{equation}
\label{oritan2}
\|V\|= \sum_{k=1}^m\sum_{j=1}^{\theta_k}\mathcal H^1\mres_{{\rm graph}\,f_j^{(k)}}\,\,\mbox{ on }B_{2R_i}.
\end{equation}
In fact, one shows \eqref{oritan2} on $B_{2R_i}\setminus B_{\epsilon}$ 
for arbitrary $\epsilon>0$ first, which is
enough to conclude the claim on $B_{2R_i}$. 

Finally we prove that the angle between the neighboring lines $M_k$ and $M_{k+1}$ is either 60 or 120 degrees.
It is enough to consider the case of $M_1$ and $M_2$, and let us identify $M_1$ with the positive $x$-axis.
As we saw already, we have a sequence of 
$C^{1,\sfrac12}$ functions $f_\ell^*:=f_{(\theta_1,\ell)}^{(1)}$ defined on $[x_\ell,
2R_i]$ (where $x_\ell:=x_{(\theta_1,\ell)}^{(1)}$) with $(Z_\ell\setminus Z_\ell^*)\cap B_{\epsilon_\ell r_\ell}((x_\ell,f^*_\ell(x_\ell)))\neq 
\emptyset$. Moreover, since $(f_{\theta_1}^{(1)})'(0)=0$, one can show that $\lim_{\ell\rightarrow\infty} 
(f_\ell^{*})'(x_\ell)=0$. 
Starting from $(x_\ell,f^*_\ell(x_\ell))$, one can construct two curves approximating $Z_\ell$
again (due to Lemma \ref{defC4}
and arguing as in Step 5 in the proof of Theorem \ref{regflat}). Follow the curve going up 
approximately in the direction of $(-1,\sqrt 3)$ denoted by $C_\ell$,
and note that the angle between $C_\ell$ and $M_1$ ($x$-axis)
at $(x_\ell,f^*_\ell(x_\ell))$ approaches to 120 degrees as $\ell\rightarrow\infty$. We consider the
following three cases, (a), (b), (c).
\newline
(a) Suppose that for each $\ell$, $C_\ell$ extends without 
encountering a triple junction until it exits from $B_{2R_i}$ for all large $\ell$. By the estimate of
\eqref{varf}, there is a limit curve $C_\infty$ of $C_\ell$ which contributes to 
$\|V\|$, and it has to be one of ${\rm graph}\,f_j^{(k)}$.
In fact, $C_\infty$ has to be ${\rm graph}\,f_1^{(2)}$. This is because of the following. Consider
the connected component of 
\begin{equation*}
B_{R_i}\cap \Big(\{z\,:\,\Cr{c_9}\epsilon_\ell r_\ell<d(z,C_\ell)<r_\ell/4\}
\cup \{(x,y)\, :\, \Cr{c_9}\epsilon_\ell r_\ell<y-f_\ell^*(x)<r_\ell/4,x\in[x_\ell,R_i]\}\Big). 
\end{equation*}
This looks like a wedge-shaped strip of width $(\sfrac14-\Cr{c_9}\epsilon_\ell)r_\ell$ just above
${\rm graph}\,f_\ell^*\cup C_\ell$, and due to the construction of $f_\ell^*$ and $C_\ell$ (see \eqref{oritan4}), 
it does not contain any point of $Z_\ell$. This implies that there is no point of $Z_\ell$ in the set
bounded by this wedge-shaped strip and $\partial B_{R_i}$ (and between $M_1$ and $M_2$), since otherwise we should be able to 
construct a curve (and possibly a network) approximating $Z_\ell$.
Then, arguing as in Step 5 in the proof of Theorem \ref{regflat}, this has 
to cross this wedge-shaped strip, but that is not possible. Thus there cannot be
any point of ${\rm spt}\,\|V\|$ in the region bounded by ${\rm graph}\,f_{\theta_1}^{(1)}$, $C_{\infty}$
and $\partial B_{R_i}$, and $C_{\infty}$
has to be ${\rm graph}\,f_1^{(2)}$. This also shows that the angle between $M_2$ and $M_1$ has to be 120 degrees.
\newline
(b) Suppose that for each $\ell$, as one extends $C_\ell$ by choosing line segments, $C_\ell$
meets some point of $Z_\ell\setminus Z_\ell^*$ (an encounter with a triple junction). Because of
\eqref{nojun}, this point has to converge to the origin as $\ell\rightarrow\infty$. We can again
construct two curves starting from this point for each $\ell$, and let $\tilde C_\ell$ be the
one going in the direction of (approximately) $(1,\sqrt 3)$. Assume that $\tilde C_\ell$ 
can be extended without encountering a triple junction and exits from $B_{2R_i}$. In this case, 
note that the length of $C_\ell$ approaches to $0$ as $\ell\rightarrow\infty$ and the slope of
$\tilde C_\ell$ at the junction approaches to $\sqrt 3$. By arguing as in the case (a), we can prove that
the limit curve $\tilde C_\infty$ has to be ${\rm graph}\, f_1^{(2)}$, proving that the angle between 
$M_1$ and $M_2$ is 60 degrees. 
\newline
(c) Continuing from (b), we consider the possibility that $\tilde C_\ell$ encounters another triple junction. 
There, we can again construct two curves starting from this point, and we follow the curve $\hat C_\ell$
going to the right.
Just to visualize the setting, imagine that we ``walk'' on ${\rm graph}\,f_{\ell}^*$ in the negative $x$ direction, then, at the triple junction $(x_\ell,
f_{\ell}^*(x_\ell))$, turn 60 degrees to the right, and walk along $C_\ell$. Then at another triple junction,
we turn 60 degrees to the right again and walk along $\tilde C_\ell$. In the present case, we assume that
we encounter another triple junction, and turn right by 60 degrees, and walk along $\hat C_\ell$.
All these triple junctions converge to the origin as $\ell\rightarrow\infty$ by \eqref{nojun}. 
Note that $\hat C_\ell$ has to be almost parallel to the $x$-axis and we would be walking in the
positive $x$ direction. If $\hat C_\ell$ can be 
extended without triple junction until it exits from $B_{2R_i}$, this means that we have another
curve above ${\rm graph}\,f_{(\theta_1,\ell)}^{(1)}$ of $Z_\ell$, and the limit $\hat C_{\infty}$ 
of $\hat C_\ell$ will be a curve tangent to $M_1$ at the origin which contribute to $\|V\|$
in addition to $f_1^{(1)},\ldots,f_{\theta_1}^{(1)}$. This would be a contradiction to
\eqref{oritan2}. Thus, $\hat C_\ell$ has to encounter a triple junction. But then, just as in the 
argument of Step 5, the resulting curve (or network) starting from this triple junction has to go down
in the negative $y$ direction until it intersects ${\rm graph}\, f_{(\theta_1,\ell)}^{(1)}$, but that
is not possible. To sum up, the case (c) actually does not occur, and we have either (a) or (b), 
proving the desired angle condition.

The proof of Theorem \ref{KTmain2} is now complete, by rotating the coordinate so that $M_1$ is the
positive $x$-axis. The equality \eqref{conc3} follows immediately from the fact that the tangent cone
is stationary. The proof of Theorem \ref{KTmain3} ($N=2$ case) can be completed by reviewing 
the present proof. The point is that, if $N=2$ we have $Z_\ell=Z_\ell^*$ due Lemma 5.5 with $N=2$,
that is, there is no triple junction in $Z_\ell$. Then, the above analysis around ${\rm sing}\,V$ shows that
the tangent cone has to be a line without junction and ${\rm sing}\,V$ is empty.
\end{proof}
\begin{lem}\label{comp}
Suppose that $\{V_t^i\}_{t\in I}$ ($i\in\N$) is a sequence of 1-dimensional Brakke flows in $U_R(a)\times I$ ($I$ is an interval)
such that $\sup_{i\in\N}\sup_{t\in I}\|V_t^i\|(U_R(a))<\infty$ and such that the property stated in Theorem \ref{KTmain2}
(resp.\,\ref{KTmain3}) holds true in $U_R(a)$ for a.e.\,$t\in I$. Then there exist a subsequence 
(denoted by the same index) and a Brakke flow $\{V_t\}_{t\in I}$ such that $\lim_{i\rightarrow\infty}
\|V_t^i\|=\|V_t\|$ in $U_R(a)$ for all $t\in I$ and $V_t$ satisfies Theorem \ref{KTmain2} (resp. \ref{KTmain3}) for a.e.\,$t\in I$.
\end{lem}
\begin{proof}
One can find the proof of compactness 
(the existence of convergent subsequence and the limit being a Brakke flow) 
in \cite{Ilm1,Ton1} so we only discuss the limit satisfying the regularity property. 
For any non-negative function $\phi\in C^2_c(U_R(a))$ and $t_1<t_2$ in $I$, by \eqref{bineq} and \eqref{bineq2},
\begin{equation}
\begin{split}
\|V^i_{t}\|(\phi)\Big|_{t=t_1}^{t_2}&\leq \int_{t_1}^{t_2}\int_{U_R(a)}
|\nabla\phi||h^i|-\phi|h^i|^2\,d\|V_t^i\|dt 
\leq \int_{t_1}^{t_2}\int_{U_R(a)} \frac{|\nabla\phi|^2}{2\phi}-\frac12\phi|h^i|^2\,d\|V_t^i\| dt
\\
&\leq (t_2-t_1)\sup\|\nabla^2\phi\|\sup_{t\in I}\|V_t^i\|({\rm spt}\phi)-\frac12\int_{t_1}^{t_2}
\int_{U_R(a)}\phi|h^i|^2\,d\|V_t^i\|dt,
\end{split}
\end{equation}
where $h^i=h(\cdot,V_t^i)$ and we used $ab\leq \frac12(a^2+b^2)$ and $|\nabla\phi|^2/\phi\leq 2\sup\|\nabla^2\phi\|$. By Fatou's Lemma and the given uniform bound, we have
\begin{equation}
\int_{t_1}^{t_2}\liminf_{i\rightarrow\infty}\Big(\int_{U_R(a)}\phi|h^i|^2\,d\|V_t^i\|\Big)dt<\infty.
\end{equation}
Thus, for a.e.\,$t\in I$, we can choose a subsequence (denoted by the same index) such that
the $L^2$ norm of $h^i$ is uniformly bounded on a compactly supported domain in $U_R(a)$. 
Consider $t\in I$ such that such a subsequence exists and additionally assume that the 
regularity property described in Theorem \ref{KTmain2} holds, which is true for a.e.\,$t$. 
By the lower semicontinuity argument,
the limit $V_t$ has also $h(\cdot,V_t)\in L_{loc}^2(\|V_t\|)$ for such $t$ and 
we may also assume that $V_t$ is integral for this $t$ (since we already know that $V_t$ is 
integral for a.e.\,$t$). For any $z_0\in {\rm spt}\,\|V_t\|$, a tangent cone exists, which 
consists of a finite number of half-lines with integer multiplicities. Let $F_i(z):=(z-z_0)/r_i$,
where $(F_i)_{\sharp}V_t$ converges to a tangent cone. 
We can choose a further subsequence $\{V_t^i\}_{i\in\N}$ (with the same index) so that
$(F_i)_{\sharp}V_t^i$ also converges to the same tangent cone as $i\rightarrow\infty$. Since
$\int_{B_{r_i}(z_0)} |h^i|^2\,d\|V_t^i\|$ is uniformly bounded and due to the regularity property of 
$V_t^i$, the variation of the slope of tangent line along each curve of 
${\rm spt}\,\|V_t^i\|$ in $B_{r_i}(z_0)$
is $O((r_i)^{\sfrac12})$, and any of their junctions are of the type described in Theorem \ref{KTmain2}.
The rest of the argument is similar to the one given in the proof of Theorem \ref{KTmain2}. In fact,
it is much easier since the converging objects are regular curves $V_t^i$ 
instead of ``approximately regular'' curves $Z_\ell$. 
If the tangent cone of $V_t$ at $z_0$ is a line with multiplicity $\nu$, then, for all sufficiently large 
$i$, $\|(F_i)_{\sharp}V_t^i\|$ will be close to it in measure in $B_1$, 
and we may conclude that ${\rm spt}\,\|V_t^i\|$ cannot have a junction and it
is a union of stacked $\nu$ curves of class $W^{2,2}$ with uniform $W^{2,2}$ bound. 
With the radius fixed at this
point and letting $i\rightarrow \infty$, we conclude that $V_t$ is also made up by a stacked $W^{2,2}$
curves near $z_0$. If the tangent cone is not a line but a union of half-lines with multiplicities, 
we know that such point is isolated from the same argument as before. Also 
$(F_i)_{\sharp}V_t^i$ is locally close to half-lines away from the origin for large $i$, and arguing 
as before, we can conclude that angles of the tangent cone have to be either 60 or 120 degrees 
depending on how many triple junctions there are along the curve. We then prove that
the limit $\|V_t\|$ near each half-line is obtained as the limit of graphs of $W^{2,2}$ functions.
We omit the detail since the idea is similar. The case of $N=2$ is simpler since $V_t^i$ is 
a union of embedded $W^{2,2}$ curves which may be tangent to each other at some point. Since there
is a uniform $W^{2,2}$ bound, one can check that $V_t$ cannot have a tangent cone which
is not a line. Since any tangent cone is a line with multiplicity, one can show the desired local 
property.  
\end{proof}
Theorem \ref{KTmain4} is a direct corollary of Lemma \ref{comp}: 
\begin{proof} 
Any tangent flow is obtained as a 
limit of parabolically rescaled Brakke flow, and the regularity property in Theorem \ref{KTmain2} and \ref{KTmain3}
is not affected under the rescaling. Thus Lemma \ref{comp} shows that the tangent flow has the same
property.
\end{proof}
\section{Concluding remarks}\label{Final}

\subsection{Results for Brakke flows in \cite{Stuvard-Tonegawa}}\label{Final1}
In Section \ref{Main}, mainly to avoid confusion, 
we state results for Brakke flows obtained in \cite{KimTone}.
However, since the results are local in nature, the same regularity properties hold true 
for Brakke flows obtained in \cite{Stuvard-Tonegawa}, where the existence of Brakke flow
with fixed boundary condition in a strictly convex domain $U\subset \R^{n+1}$ was established. 
The proof of the present paper relies on the estimate \eqref{caMCF9}, and the same but localized
version \cite[Proposition 4.13]{Stuvard-Tonegawa} is available and
the proof proceeds by the same argument in the interior of $U$ (with the weight function
$\Omega=1$). Additionally, we recall the following
\begin{thm} (\cite[Theorem B]{Stuvard-Tonegawa}) \\
Let $\{V_t\}_{t\in\R^+}$ be a Brakke flow obtained in \cite{Stuvard-Tonegawa}. Then there 
exists a sequence of times $\{t_k\}_{k=1}^{\infty}$ with $\lim_{k\rightarrow\infty}t_k=\infty$
such that the corresponding varifolds $V_k:=V_{t_k}$ converge to a stationary integral varifold 
$V_{\infty}$ in $U$ such that $({\rm clos}\,({\rm spt}\,\|V_{\infty}\|))\setminus U=\partial\Gamma_0$.
\label{STthm}
\end{thm} 
Here $\partial\Gamma_0$ is the given boundary condition. Since $V_{\infty}$ is stationary and we 
are concerned with $n=1$, by \cite{AA},
${\rm spt}\,\|V_{\infty}\|$ consists of finite line segments with junctions. 
A corollary of the present paper is the following.
\begin{thm}
For $n=1$, the limit $V_{\infty}$ of Theorem \ref{STthm} has to satisfy the same
angle condition at junctions in $U$, namely, angles at junctions are either 60 or 120 degrees. 
For $N=2$, ${\rm spt}\,\|V_{\infty}\|$ consists of
lines with no junction in $U$. 
\end{thm}
The proof is by observing that we may choose the sequence 
$\{t_k\}_{k=1}^\infty$ so that the $L^2$ curvatures of $V_{k}$ converge to 0 and so that
$V_k$ has the regularity property stated in Theorem \ref{KTmain2} (or \ref{KTmain3} if $N=2$). 
See \cite[Section 7]{Stuvard-Tonegawa} for the detail of the choice of the sequence. Then the 
same argument for the proof of Lemma \ref{comp} shows the claim for $V_{\infty}$.  

\subsection{The higher dimensional case}
The present paper does not give any definite conclusion for $n\geq 2$, except for 
Theorem \ref{chlim} that tells that the approximate MCF $\partial\mathcal E_{j_\ell}(t)$ 
is close to a measure-minimizing regular cluster in a small length scale of $o(1/j_\ell^2)$
for a.e.\,$t$. 
Even though we have the $L^2$ bound of smoothed mean curvature vector of 
$\partial\mathcal E_{j_\ell}$, unlike the case of $n=1$, it does not provide enough 
control of the converging $n$-dimensional surfaces. 
It still raises a natural question as to what is the additional regularity property of 
the Brakke flow obtained in \cite{KimTone,Stuvard-Tonegawa}. As a closely related result,
the measure-theoretic closure of smooth MCF of clusters is studied in \cite{SW}, where
it is proved that such class is compact under a set of natural assumptions. Formally, a major difficulty
of pursuing the similar line of idea for the present case is that $\partial \mathcal E_{j_\ell}(t)$
satisfies the approximate Brakke's inequality for test functions which cannot vary within the
length scale of $o(1/j_\ell)$ (see \eqref{caMCF5} which is an approximate Brakke's inequality
and the definition of the class of test functions $\mathcal A_j$). Because of this restriction, 
we cannot fully 
utilize the aspect that $\partial
\mathcal E_{j_\ell}(t)$ is an approximate MCF. On the other hand, it seems reasonable to 
expect that certain ``unstable'' singularity cannot persist in time in general. 

\subsection{Further regularity results}
\label{frr}
The present paper narrows down the possibility of angles of junctions to 0, 60 and 120 degrees,
but this does not necessarily mean that they all actually appear for a set of times with
positive Lebesgue measure. In fact, except for 
triple junctions of 120 degrees with unit density, we expect that the other types of junctions 
should not persist in time and are
likely to break up into triple junctions immediately. Also, higher multiplicity should not persist
as well, since setting the multiplicity equal to 1 simply reduces the total mass of varifold
and Brakke's formulation should not pose serious difficulty doing so. It may be the case that certain
restarting procedure is necessary to obtain a ``better'' Brakke flow which is a network with 
triple junctions with unit density for a.e.\,$t$. 

\bibliographystyle{siam}
\bibliography{ERMCF_biblio}

\end{document}